\documentclass{amsart}
\usepackage{amsmath,amssymb,amsthm,mathrsfs,mathabx}
\usepackage[active]{srcltx}

\usepackage{hyperref}

 \usepackage[pdftex]{graphicx}
 \usepackage{color}

\usepackage{amsmath}
\usepackage{amsthm}
\usepackage{amssymb}

\newcommand{\leaveout}[1]{}

\makeatletter
\renewcommand{\p@enumii}{}
\makeatother

\newcommand{\cB}{{\mathcal B}}

\newcommand{\cD}{{\mathcal D}}

\newcommand{\cG}{{\mathcal G}}

\newcommand{\cL}{{\mathcal L}}

\newcommand{\cK}{{\mathcal K}}

\newcommand{\cS}{{\mathcal S}}

\newcommand{\bW}{{\mathbf W}}
\newcommand{\bX}{{\mathbf X}}

\newcommand{\bu}{{\mathbf u}}
\newcommand{\bx}{{\mathbf x}}
\newcommand{\by}{{\mathbf y}}
\newcommand{\bz}{{\mathbf z}}
\newcommand{\bv}{{\mathbf v}}
\newcommand{\bw}{{\mathbf w}}

\newcommand{\bS}{{\mathbf S}}

\newcommand{\precc}{\llcurly}
\newcommand{\succc}{\ggcurly}

\newcommand{\DD}[1]{\mathbin{\frac{\rm d }{{\rm d }#1}}}
\newcommand{\ddt}{\DD t}

\newcommand{\ud}{\,{\mathrm d}}

\newcommand\R{{\mathbb R}}

\newcommand\C{{\mathbb C}}

\newcommand\Z{{\mathbb Z}}
\newcommand\rplus{{\R^{+}}}
\newcommand\cplus{{\C^{+}}}

\newcommand\zplus{{\Z^{+}}}

\newcommand\rminus{{\R^{-}}}

\newcommand{\dA}{\mathbf A}
\newcommand{\dB}{\mathbf B}
\newcommand{\dC}{\mathbf C}
\newcommand{\dD}{\mathbf D}
\newcommand{\dW}{\mathbf W}

\newcommand{\AB}{{A\& B}}
\newcommand{\CD}{{C\& D}}

\newcommand{\SysNode}{\bbm{\AB \\ \CD}}
\newcommand{\SmallSysNode}{\sbm{\AB \cr \CD}}

\newcommand{\dom}[1]{\operatorname{dom}(#1)}

\newcommand{\range}[1]{\operatorname{ran}(#1)}

\newcommand{\supp}[1]{\operatorname{supp}(#1)}

\newcommand{\re}[1]{\operatorname{Re}\,#1}

\newcommand{\Ker}[1]{\operatorname{ker}(#1)}

\newcommand{\ipd}[2]{\langle #1 , #2 \rangle}

\newcommand{\Ipd}[2]{\left\langle #1 , #2 \right\rangle}

\newcommand{\Ipdp}[2]{\left\langle #1 , #2 \right\rangle}

\newcommand{\set}[1]{\left\lbrace #1 \right\rbrace}

\newcommand{\bigmid}{\bigm\vert}
\newcommand{\Bigmid}{\Bigm\vert}
\newcommand{\biggmid}{\biggm\vert}

\newcommand{\bi}{\begin{itemize}}
\newcommand{\ei}{\end{itemize}}
\newcommand{\be}{\begin{enumerate}}
\newcommand{\ee}{\end{enumerate}}

\newcommand{\Afrak}{\mathfrak A}
\newcommand{\Bfrak}{\mathfrak B}
\newcommand{\Cfrak}{\mathfrak C}
\newcommand{\Dfrak}{\mathfrak D}

\newcommand{\Hfrak}{\mathfrak H}

\newcommand{\Tfrak}{\mathfrak T}

\newcommand{\Vfrak}{\mathfrak V}

\newcommand{\ya}{\reflectbox{\rm R}}     

\newcommand{\sbm}[1]{\left[\begin{smallmatrix}#1
\end{smallmatrix}\right]}

\newcommand{\bbm}[1]{\begin{bmatrix}#1\end{bmatrix}}

\newtheorem{lemma}{Lemma}[section]
\newtheorem{theorem}[lemma]{Theorem}
\newtheorem{corollary}[lemma]{Corollary}
\newtheorem{proposition}[lemma]{Proposition}

\theoremstyle{definition}
\newtheorem{example}[lemma]{Example}
\newtheorem{definition}[lemma]{Definition}
\newtheorem{remark}[lemma]{Remark}

\numberwithin{equation}{section}
\begin{document}

\title[The Infinite-Dimensional Bounded Real Lemmas in Continuous Time]{The Infinite-Dimensional Standard and Strict Bounded Real Lemmas in Continuous Time:\ The storage function approach}

\author[J.A. Ball]{J.A. Ball}
\address{J.A. Ball, Department of Mathematics, Virginia Tech, Blacksburg, VA 24061-0123, USA}
\email{joball@math.vt.edu}

\author[S. ter Horst]{S. ter Horst}
\address{S. ter Horst, Department of Mathematics, Research Focus Area:\ Pure and Applied Analytics, North-West University, Potchefstroom, 2531 South Africa and DSI-NRF Centre of Excellence in Mathematical and Statistical Sciences (CoE-MaSS)}
\email{Sanne.TerHorst@nwu.ac.za}

\author[M. Kurula]{M. Kurula}
\address{M. Kurula, {\AA}bo Akademi Mathematics, Henriksgatan 2, 20500 {\AA}bo, Finland}
\email{Mikael.Kurula@abo.fi}

\thanks{This work is based on the research supported in part by the National Research Foundation of South Africa (Grant Numbers 118513 and 127364).}


%
%
%



\begin{abstract}
The bounded real lemma (BRL) is a classical result in systems theory, which provides a linear matrix inequality criterium for dissipativity, via the Kalman-Yakubovich-Popov (KYP) inequality. The BRL has many applications, among others in $H^\infty$ control. Extensions to infinite dimensional systems, although already present in the work of Yakubovich, have only been studied systematically in the last few decades. In this context various notions of stability, observability and controllability exist, and depending on the hypothesis one may have to allow the KYP-inequality to have unbounded solutions which forces one to consider the KYP-inequality in a spatial form. In the present paper we consider the BRL for continuous time, infinite dimensional, linear well-posed systems. Via an adaptation of Willems' storage function approach we present a unified way to address both the standard and strict forms of the BRL. We avoid making use of the Cayley transform and work only in continuous time. While for the standard bounded real lemma, we obtain analogous results as there exist for the discrete time case, when treating the strict case additional conditions are required, at least at this stage. This might be caused by the fact that the Cayley transform does not preserve exponential stability, an important property in the strict case, when transferring a continuous-time system to a discrete-time system.
\end{abstract}

\maketitle


\tableofcontents

\section{Introduction}\label{sec:intro}

The study and elaboration of the Bounded Real Lemma (BRL) has a rich history, beginning with the work of Kalman \cite{Kal63}, Yakubovich \cite{Yak62} and of Popov \cite{Pop61}. From the beginning, the Kalman-Yakubovich-Popov (KYP) lemma was viewed more broadly as the quest to establish the equivalence between a frequency-domain inequality (FDI) and a Linear Matrix Inequality (LMI). In our case, this will actually be a Linear Operator Inequality.

A finite dimensional, linear input-output system in continuous time is frequently written in input/state/output form
\begin{equation}\label{eq:introiso}
	\Sigma:\quad
	\bbm{\dot \bx(t)\\\by(t)} = \bbm{A&B\\C&D} \bbm{\bx(t)\\\bu(t)},
	\quad t\geq 0,\quad \bx(0)=x_0,
\end{equation}
where the \emph{state} $\bx(t)$ at time $t$ takes values in the \emph{state space} $X=\C^n$ (with ${\mathbb C}$ denoting the set of complex numbers), the \emph{input}
$\bu(t)$ lives in the \emph{input space} $U=\C^m$, and the \emph{output} $\by(t)$ in the \emph{output space} $Y=\C^k$, and where $A$, $B$, $C$, $D$
are matrices of appropriate sizes. The \emph{initial time} is $t=0$ and $x_0\in X$ is the given \emph{initial state} of the system.
By the elementary
theory of differential equations, the unique solution of \eqref{eq:introiso} is
\begin{equation}  \label{eq:introiso'}
\left\{  \begin{aligned}
\bx(t) & = e^{At} x_0 + \int_0^t e^{A(t-s)} B \bu(s) \ud s, \\
\by(t) & = C e^{At} x_0 +\int_0^t C e^{A(t-s)} B \bu(s) \ud s + Du(t).
\end{aligned}  \right.
\end{equation}

Taking Laplace transforms in \eqref{eq:introiso'}, we get
$$
\left\{  \begin{aligned}
\widehat\bx(\lambda) & = (\lambda-A)^{-1} x_0 + (\lambda- A)^{-1}B\widehat \bu(\lambda), \\
\widehat\by(\lambda) & = C (\lambda-A)^{-1} x_0 + \widehat\Dfrak(\lambda)\bu(\lambda),
\end{aligned}  \right.
$$
where
\begin{equation}  \label{intro-transfunc}
\widehat \Dfrak(\lambda ) = C (\lambda - A)^{-1} B+D
\end{equation}
is called the {\em transfer function} of the linear system \eqref{eq:introiso}. In particular, when $x_0=0$, we get
\begin{equation}   \label{transfunc-freq}
  \widehat \by(\lambda) = \widehat \Dfrak(\lambda) \widehat \bu(\lambda),
\end{equation}
i.e., the transfer function maps the Laplace transform of the input signal into the Laplace transform of the output signal. Alternatrively, let us make the Ansatz that $\bu(t)=e^{\lambda t}u_0$, $\bx(t)=e^{\lambda t}x_0$ and $\by(t)=e^{\lambda t}y_0$ form a trajectory on $\R$, where $u_0$, $x_0$ and $y_0$ are constant vectors. Then $\dot \bx(t)=\lambda e^{\lambda t}x_0$ and the first equation in \eqref{eq:introiso} gives $x_0=(\lambda-A)^{-1}Bu_0$. Plug this into the second equation of \eqref{eq:introiso} to get  $y_0=\widehat\Dfrak(\lambda)u_0$. Hence, the transfer function maps the amplitude of the input wave to the amplitude of the output wave, and this gives a second interpretation of the transfer function as a frequency response function. This second interpretation can be extended to time-varying linear systems as well; see \cite{BGK95}. For finite dimensional systems, the Laplace transform version is more common, but for infinite-dimensional systems, the frequency response version is more accessible.

We will be particularly interested in the case where  $\widehat \Dfrak(\lambda)$ is analytic on the right half-plane  ${\mathbb C}^+$.
If it is the case that in addition $\| \widehat \Dfrak(\lambda) \| \le 1$ for all $\lambda$ in the open right half-plane ${\mathbb C}^+$, we say that
$\widehat \Dfrak$ is in the {\em Schur class} (with respect to ${\mathbb C}^+$), denoted as $\cS_{U,Y}$.  

What we shall call the
 {\em standard bounded real lemma (standard BRL)} is concerned with
characterizing in terms of  the system matrix $\sbm{ A & B \\ C & D}$ when it is the case that the associated transfer function $\widehat \Dfrak(\lambda)$ is in $\cS_{U,Y}$.
A variation of the problem is the {\em strict bounded real lemma} which is concerned with the problem of characterizing in terms of the system matrix
$\sbm{A & B  \\ C & D }$  when the associated transfer function $\widehat \Dfrak(\lambda)$ is in the {\em strict Schur class}  $\cS^0_{U, Y}$, i.e., when there exists a $\rho<1$ such that $\| \widehat \Dfrak(\lambda) \| \le \rho$ for all $\lambda \in {\mathbb C}^+$.
For the finite dimensional case, the problem is pretty well understood (see \cite{AnVo73, Will72a} for the standard case and \cite{PAJ91} for the strict case), while for the
infinite dimensional case the results are not as complete, but see \cite{ArSt07} for the standard case).  Our goal here is to provide a unified approach to the standard and the strict bounded real lemmas for infinite dimensional well-posed system with continuous time (as in \cite{StafBook}); in fact, at that level of generality, this appears to be the first attempt at a strict bounded real lemma.

We shall make use of the concept of \emph{storage function} as introduced by J.\ Willems in his study of dissipative systems \cite{Will72a,Will72b}, closely related
to independent work \cite{Arov79b} of  D.\ Arov  appearing around the same time.  Here we concentrate on the special case of ``scattering'' supply rate:
$s(u,y) = \| u \|^2 - \| y \|^2$.

\begin{definition}\label{def:storage}
The function $S:X\to[0,\infty]$ is a \emph{storage function} for $\Sigma$ if $S(0)=0$ and for all trajectories $(\bu,\bx,\by)$ of $\Sigma$ with initial time $0$
and for all $t>0$, it holds that
\begin{equation}\label{eq:storfndef}
	S\left(\bx(t)\right) +\int_0^t\|\by(s)\|_Y^2\ud s \leq S\left(\bx(0)\right)+
	\int_0^t\|\bu(s)\|_U^2\ud s.
\end{equation}
If $S(x)=\|x\|_X^2$ is a storage function for $\Sigma$, then $\Sigma$ is called \emph{passive}.
\end{definition}

An easy consequence of this notion of dissipativity (i.e., existence of a storage function) is what we shall call {\em input/output dissipativity}, namely: In case the system is initialized with the initial state $x_0$ set equal to $0$, then the energy drained out of the system over the interval $[0,t]$ via the output $\by$ cannot exceed the energy inserted into the system over the same interval via the input $\bu$: that is,
$$
\int_0^t\|\by(s)\|_Y^2\ud s \leq \int_0^t\|\bu(s)\|_U^2\ud s,\quad \text{subject to }x_0=0.
$$
This implies that the transfer function is in the Schur class; more details on this can be found in Proposition \ref{prop:storageimpliesschur} below.   A non-obvious point is that the
converse holds: if $\widehat \Dfrak \in \cS_{U,Y}$, then a storage function exists for $\Sigma$, and this will be one of the statements in our standard BRL.  Similarly, as we shall see that
$\widehat \Dfrak$ being in the strict Schur class is equivalent to $\Sigma$ having what we shall call a {\em strict storage function}
(see Definition \ref{def:storage-strict} below).

 For a suitable function $\bu$, let $\tau^t$ denote the backward-shift operator
$$
 (\tau^t \bu)(s) = \bu(t+s), \quad t \in {\mathbb R}, \, t+s \in \dom{\bu}.
$$
 By time-invariance of the system equations \eqref{eq:introiso} we see that for any $t_0 > 0$ the backward-shifted trajectory
$(\tau^{t_0}\bu, \tau^{t_0}\bx, \tau^{t_0}\by)$ is again a system trajectory whenever $(\bu, \bx, \by)$ is a system trajectory.
 Setting $t_1 = t_0 > 0$, $t_2 = t + t_0>t_1$
and rewriting the resulting version of \eqref{eq:storfndef} as
$$
S(\bx(t_2)) - S(\bx(t_1)) \le \int_{t_1}^{t_2} \|\bu(s)\|_U^2\ud s - \int_{t_1}^{t_2} \|\by(s)\|_Y^2\ud s,
$$
we see that the dissipation inequality \eqref{eq:storfndef} can be interpreted as saying that the net energy stored by the system state
over the interval $[t_1, t_2]$ 
is no more than the net energy supplied to the system by the outside environment over the same time interval. 

In order to state the {\em standard} and {\em strict bounded real lemmas} even for the finite dimensional case, we need to carefully distinguish different notions of positivity for Hermitian matrices.

\begin{definition}  \label{D:matrix-positive}
 For $H$ an $n \times n$ Hermitian matrix over ${\mathbb C}$, we write
\begin{itemize}
\item $H \succ 0$ if $\langle H x, x \rangle > 0$ for all nonzero $x$ in ${\mathbb C}^{n \times n}$ (equivalently for the finite dimensional case here,
for some $\delta > 0$ we have $\langle H x, x \rangle \ge \delta \| x \|^2$ for all $x \in {\mathbb C}^n$),

\item $H \prec 0$ if $-H \succ 0$,

\item $H \succeq 0$ if $\langle H x, x \rangle \ge 0$ for all $x \in {\mathbb C}^n$,

\item $H \preceq 0$ if $-H \succeq 0$.
\end{itemize}
\end{definition}

\begin{theorem}[Standard finite dimensional bounded real lemma; see e.g.\ \cite{AnVo73, Will72a}]\label{thm:stdlemmaintro}
For a finite-di\-mensional linear system $\Sigma$ with system matrix $\bS = \sbm{ A & B \\ C & D}$ as in \eqref{eq:introiso} which is
minimal (i.e., $\textup{rank}\, [B\ AB\ \cdots\ A^{n-1}B]=n$ (controllability) and $\textup{rank}\, [C^*\ A^*C^* \ \cdots \ A^{*n-1}C^*]=n$ (observability), the following conditions are equivalent:

\begin{enumerate}
\item[(1)] After unique analytic continuation (if necessary) to a domain $\cD(\widehat \Dfrak) \supset {\mathbb C}^+$,
$\widehat\Dfrak$ is in the Schur class $\cS_{U,Y}$.
\item[(2)] The following continuous-time Kalman-Yakubovich-Popov (KYP) inequality has a solution $H\succ0$:
\begin{equation}\label{eq:KYPintro}
	\bbm{HA+A^*H+C^*C & HB+C^*D \\ B^*H+D^*C & D^*D-I}
	\preceq 0.
\end{equation}
\item[(3)] The system $\Sigma$ is similar to a passive system $\Sigma^\circ$, i.e., there exist $X^\circ$ and an invertible $\Gamma \colon X \to X^\circ$ such that
\begin{equation}\label{eq:introinter}
	 \bbm{A^\circ&B^\circ\\C^\circ&D^\circ}: = \bbm{\Gamma &0\\0&I}\bbm{A&B\\C&D} \bbm{\Gamma^{-1}&0\\0&I}
\end{equation}
satisfies \eqref{eq:KYPintro} with $H = 1_{X^\circ}$.
\item[(4)] The system $\Sigma$ has a storage function.

\item[(5)] The system $\Sigma$ has a quadratic storage function (see below).
\end{enumerate}
\end{theorem}

Here by a {\em quadratic storage function} we mean a storage function $S$ of the special form $S(x) = \langle H x, x \rangle$, where $H \succeq 0$
is a Hermitian matrix.  If $H$ is positive definite ($H \succ 0$) then $S = S_H$ has the additional property that $S$ is coercive
(there is a $\delta > 0$ so that $S_H(x) \ge \delta \| x \|^2$ for all $x \in X$).  The connection between a solution $H \succeq 0$ of the KYP-inequality
\eqref{eq:KYPintro} and a quadratic storage function
is that any $H \succeq 0$ satisfying \eqref{eq:KYPintro} generates a quadratic storage function $S$ for $\Sigma$ according to
$S(x) = S_H(x):= \langle H x, x \rangle$. The {\em strict bounded real lemma} is concerned with an analogous characterization of the strict Schur class $\cS^0_{U,Y}$.

To formulate the strict result let us introduce the following terminology.

\begin{definition}\label{def:storage-strict}
Suppose  $S \colon X \to [0, \infty]$  is such that $S(0) = 0$ and $\Sigma$ is a well-posed linear system with system trajectories $(\bu, \bx, \by)$ with initiation
at $t = 0$.  Then we say that:
\begin{enumerate}
\item[(1)] $S$ is a {\em strict storage function} for $\Sigma$ if there is a $\delta > 0$ so that, for all system trajectories $(\bu, \bx, \by)$ of $\Sigma$ and $0 \le t_1 < t_2$ we have
\begin{equation}   \label{eq:stordefstrict}
S(\bx(t_2))  + \delta \int_{t_1}^{t_2} \| \bx(s) \|^2 \ud s + \int_{t_1}^{t_2} \| \by(s) \|^2 \ud s
\le S(\bx(t_1)) + (1 - \delta) \int_{t_1}^{t_2} \| \bu(s) \|^2  \ud s.
\end{equation}

\item[(2)] $S$ is a {\em semi-strict storage function} for $\Sigma$ if condition \eqref{eq:stordefstrict} holds but with the integral term involving
the state vector $\bx(s)$ omitted, i.e., if there is a $\delta > 0$ so that, for all system trajectories $(\bu, \bx, \by)$ and $0 \le t_1 < t_2$ we have
\begin{equation}   \label{eq:stordef-semi-strict}
S(\bx(t_2))  + \int_{t_1}^{t_2} \| \by(s) \|^2 \ud s
\le S(\bx(t_1)) + (1 - \delta) \int_{t_1}^{t_2} \| \bu(s) \|^2  \ud s.
\end{equation}
\end{enumerate}
\end{definition}

In the following result the equivalence (1) $\Leftrightarrow$ (2) is due to Petersen-Anderson-Jonckheere \cite{PAJ91} (at least for the special case $D = 0$);
we add the connections with similarity and storage functions for the strict setting. 

\begin{theorem}[Finite dimensional strict bounded real lemma]    \label{thm:stdlemmastrict-fin}
Suppose that $\Sigma$ is a finite dimensional linear system with system matrix $\bS = \sbm{ A & B \\ C & D}$ as in \eqref{eq:introiso} such that
the matrix $A$ is stable (i.e.,  $A$ has spectrum only in the open left half plane:  $\sigma(A) \subset {\mathbb C}^-:=\{\lambda\in\C \mid \re (\lambda)<0\}$).  Then the following conditions are equivalent:

\begin{enumerate}
\item[(1)] Possibly after unique analytic continuation to a domain $\dom{\widehat\Dfrak} \supset {\mathbb C}^+$,
$\widehat \Dfrak$ is in the strict Schur class $\cS^0_{U,Y}$.
\item[(2)] The following continuous-time \emph{strict} Kalman-Yakubovich-Popov (KYP) inequality has a solution $H\succ0$:
\begin{equation}\label{eq:KYPintro-strict}
	\bbm{HA+A^*H+C^*C & HB+C^*D \\ B^*H+D^*C & D^*D-I}
	\prec 0.
\end{equation}
\item[(3)] The system $\Sigma$ is similar to a \emph{strictly} passive system $\Sigma^\circ$, i.e., there exist $X^\circ$ and an invertible $\Gamma \colon X \to X^\circ$ such that \eqref{eq:introinter}
satisfies \eqref{eq:KYPintro-strict} with $H = 1_{X^\circ}$.

\item[(4)] The system $\Sigma$ has a quadratic, coercive strict storage function.

\item[(5)] The system $\Sigma$ has a semi-strict storage function.

\end{enumerate}
\end{theorem}

In the infinite dimensional case, we wish to allow one or each of the coefficient spaces, i.e., the input space $U$, the state space $X$, or the output space $Y$,
to be a infinite dimensional Hilbert space.
The  situation becomes more involved in at least three respects:
\begin{itemize}
\item The system matrix $\sbm{A&B\\C&D}$ is replaced by an (in general) unbounded \emph{system node}
(see \cite[Definition 4.7.2]{StafBook}, \cite[\S2]{ArSt07} or \S\ref{sec:system-node} below for details) between Hilbert spaces $U$, $X$ and $Y$.
Here we restrict ourselves to the setting of {\em well-posed systems}, i.e., in place of the system matrix $\sbm{ A & B \\ C & D}$
as in \eqref{eq:introiso} there is a well-defined one-parameter family of block $2 \times 2$ operator matrices
$$
\begin{bmatrix} \Afrak^t & \Bfrak^t \\ \Cfrak^t & \Dfrak^t \end{bmatrix} \colon \begin{bmatrix}  X \\ L^2([0,t], U) \end{bmatrix} \to
 \begin{bmatrix} X  \\ L^2([0,t], Y) \end{bmatrix}, \quad t > 0,
 $$
 which corresponds to the mapping such that
 $$
  \begin{bmatrix} \Afrak^t & \Bfrak^t \\ \Cfrak^t & \Dfrak^t \end{bmatrix}  \colon  \begin{bmatrix} \bx(0) \\  \pi_{[0,t]} \bu \end{bmatrix}
  \to \begin{bmatrix} \bx(t) \\ \pi_{[0,t]} \by \end{bmatrix}, \quad t > 0,
$$
whenever  $(\bu, \bx, \by)$ is a system trajectory.
It is often advantageous to work with the "integrated operators"  $\Afrak^t$, $\Bfrak^t$, $\Cfrak^t$, $\Dfrak^t$ instead of with the system node directly.
In case the system is finite dimensional and given by system matrix $\sbm{ A & B  \\ C & D}$, one can read off from  \eqref{eq:introiso'}  that the
integrated operators $\Afrak^t$, $\Bfrak^t$, $\Cfrak^t$, $\Dfrak^t$ are given by
\begin{align*}
&  \Afrak^t \colon x_0 \mapsto e^{At} x_0,\\  &\Bfrak^t \colon \bu|_{[0,t]} \mapsto
   \int_0^t e^{A(t - s)}  B u(s) \ud s,\quad
\Cfrak^t \colon x_0 \mapsto C e^{As} x_0|_{0 \le s \le t},\\
&  \Dfrak^t \colon \bu|_{[0,t]} \mapsto  \bigg( C \int_0^s  e^{A(s-s')}  B \bu(s') \ud s' + D \bu(s) \bigg) \bigg|_{0 \le s \le t}.
\end{align*}
To get some additional flexibility with respect to choice of location $t_0$ for the specification of the initial condition ($\bx(t_0) = x_0$), Staffans
(see \cite[page 30]{StafBook}) defines three ``master operators"
\begin{equation}\label{eq:introinteg}
\begin{aligned}
& \Bfrak \bu: = \int_{-\infty}^0 \Afrak^{-s} B \bu(s) \ud s, \quad
 \Cfrak x: = \bigg( t \mapsto C \Afrak^t x\bigg)_{t \ge 0}, \\
& \Dfrak \bu: = \bigg( t \mapsto \int_{-\infty}^t C \Afrak^{t-s} B \bu(s) \ud s + D \bu(t) \bigg)_{t \in {\mathbb R}}
\end{aligned}
\end{equation}
and observes that the analogues of $\Bfrak^t$, $\Cfrak^t$, $\Dfrak^t$ for the case where the initial condition is taken at $t = t_0$ rather than $t=0$ (denoted as $\Bfrak^t_{t_0}$, $\Cfrak^t_{t_0}$, $\Dfrak^t_{t_0}$) are all easily expressed in terms of the master operators;  for the case where $t_0=0$ the formulas are as in equation \eqref{BCD^t} below.

We let the collection of operators written in block matrix from
(even though it does not fit as the representation of a single operator between a two-component input space and a
two-component output space) $\sbm{ \Afrak & \Bfrak \\ \Cfrak  & \Dfrak}$ denote the associated well-posed linear system.

\item
Secondly, since the state space $X$ may be infinite dimensional, the solution $H$ of \eqref{eq:KYPintro} can become unbounded, both from below and from above.
In this case the notion of positivity for a (possibly unbounded) selfadjoint Hilbert-space operator becomes still more refined than that for the finite dimensional case (cf., Definition \ref{D:matrix-positive}) as follows.

\end{itemize}

\begin{definition}   \label{D:operator-positive}
For an unbounded, densely defined, selfadjoint operator $H$ on $X$ with domain $\dom{H}$ we say:
\begin{enumerate}
\item[(1)] $H$ is {\em positive semidefinite} (written $H \succeq 0$) when $\Ipdp{Hx}{x}\geq 0$ for all $x\in\dom{H}$;
\item[(2)] $H$ is {\em positive definite} (written $H \succ 0$) whenever $\Ipdp{Hx}{x}> 0$ for all $0\neq x\in\dom{H}$;
\item[(3)] $H$ is {\em strictly positive definite} (written $H \succc 0$) whenever there exists a $\delta>0$ so that $\Ipdp{Hx}{x} \geq \delta \|x\|^2$
for all $0\neq x\in\dom{H}$.
\end{enumerate}
\end{definition}

\begin{itemize}

\item[]

By \cite[Theorem 3.35 on p.\ 281]{Kato}, each positive semidefinite operator $H$ on $X$ admits a positive semidefinite square root $H^\frac{1}{2}$, for which we have $H=H^\frac{1}{2}H^\frac{1}{2}$, and hence \[
\dom{H}=\left\{x\in \dom{H^\frac{1}{2}} \bigmid H^\frac{1}{2} x \in \dom{H^\frac{1}{2}} \right\}\subset \dom{H^\frac{1}{2}}.
\]
Throughout this paper we use the standard ordering for possibly unbounded positive semidefinite operators (see, e.g., \cite[\S5]{ArKaPi06} or \cite[(2.17) on p.\ 330]{Kato}):\ {\em given positive semidefinite operators $H_1$ and $H_2$ on a Hilbert space $X$, we write $H_1 \preceq H_2$ if}
\[
\dom{H_2^{\frac{1}{2}}} \subset
\dom{H_1^{\frac{1}{2}}} \quad\mbox{and}\quad
\| H_1^{\frac{1}{2}} x\| \le \| H_2^{\frac{1}{2}}x \|\ \  \mbox{for all } x \in \dom{H_2^{\frac{1}{2}}}.
\]
In case $H_2$ and  $H_1$ are bounded, this amounts to the standard Loewner ordering for bounded selfadjoint operators. Similarly we define
$H_1 \prec H_2$  and $H_1 \precc H_2$, and we write $H_1 \succeq H_2$ (resp.\ $H_1 \succ H_2$ and $H_1 \succc H_2$) whenever $H_2 \preceq H_1$ (resp.\ $H_2 \prec H_1$ and $H_2 \precc H_1$).

\item  Thirdly,  with all of $A$, $B$, $C$, $D$, being possibly unbounded, it is more difficult to make sense of the formula \eqref{intro-transfunc} for the transfer function of the system $\Sigma$. However, there is a formula for the well-posed-system setup based on the interpretation of the transfer function as a ``frequency response function" which appeared at the beginning of the introduction. 
There is also a formula for the transfer function analogous to formula \eqref{intro-transfunc}  expressed directly in terms of
 the associated system node  $\bS$ (see the formula \eqref{node-transfunc} to come).  All these ideas are worked out in detail in Staffans' book
 \cite{StafBook} and the fragments needed here are reviewed in \S\ref{sec:prel} and \S\ref{sec:system-node} below.
\end{itemize}

In the case of unbounded positive semidefinite solutions $H$, the associated quadratic function $S_H$
should be allowed to take on the value infinity according to the formula:
$$
  S_H(x) =  \begin{cases}  \| H^{\frac{1}{2}}x \|^2_X & \text{if } x \in \dom{H^{\frac{1}{2}}}, \\
       \infty & \text{if } x \notin \dom{H^{\frac{1}{2}}}.  \end{cases}
 $$

 \begin{remark}  \label{R:HvsSH}
Note that then $H$ being {\em bounded} is detected in the associated quadratic function $S_H$ by
$S_H$ being {\em finite-valued,}  while $H$ being {\em strictly positive definite}
(i.e., $H \succc 0$) is detected in $S_H$ by $S_H$ being {\em coercive}, i.e., there is a $\delta > 0$ so that $S_H(x)  \ge \delta \| x \|^2$ for all $x \in X$.
\end{remark}

Also, for the case where $H$ is unbounded, the similarity  $\Gamma$ should  be weakened to a {\em pseudo-similarity} defined as follows.

\begin{definition}   \label{def:pseudosim}
 Two well-posed systems $\Sigma=\sbm{\Afrak&\Bfrak\\\Cfrak&\Dfrak}$ and $\Sigma^\circ=\sbm{\Afrak^\circ&\Bfrak^\circ\\\Cfrak^\circ&\Dfrak^\circ}$,
 with state spaces $X$ and $X^\circ$, respectively, are \emph{pseudo-similar} if $\Dfrak^\circ=\Dfrak$ and there exists a closed, densely defined and injective
 linear operator $\Gamma:X\supset \dom{\Gamma}\to  X^\circ$ with dense range, called a \emph{pseudo-similarity}, with the following properties:
\begin{enumerate}
\item[(1)] $\range{\Bfrak}\subset\dom{\Gamma}$ and $\Bfrak^\circ=\Gamma\Bfrak$, or equivalently
$\range{\Bfrak^t} \subset \dom{\Gamma}$ and $\Bfrak^{\circ t} = \Gamma \Bfrak^t$ for each $t$.

\item[(2)] for all $t\geq0$, $\Afrak^t\dom{\Gamma}\subset \dom{\Gamma}$ and $\Afrak^{\circ t}\Gamma=\Gamma\Afrak^t\big|_{\dom{\Gamma}}$, and

\item[(3)] $\Cfrak^\circ\Gamma=\Cfrak\big|_{\dom{\Gamma}}$, or equivalently, $\Cfrak^{\circ t} \Gamma = \Cfrak^t \big|_{\dom{\Gamma}}$ for all $t > 0$.
\end{enumerate}
If $\Gamma$ is bounded with a bounded  inverse, then $\Sigma$ and $\Sigma^\circ$ are said to be \emph{similar}.
(In this case the condition that $\dom{\Gamma}=X$ is automatically satisfied.)
\end{definition}

This definition is reproduced from \cite[Definition 9.2.1]{StafBook}, but with the condition that the range of $\Gamma$ is dense added and
a couple of redundant assumptions dropped; observe that Staffans also states on page 512 of \cite{StafBook} that $\Gamma^{-1}$ is
a pseudo-similarity if $\Gamma$ is a pseudo-similarity, that property (1) in Definition \ref{def:pseudosim} implies that $\widetilde\Bfrak$
maps into $\range\Gamma$ and item (2) implies that $\range\Gamma$ is invariant under $\widetilde\Afrak^t$. Hence the tw o
pseudo-similarity definitions are equivalent.

We make the following additional definitions:
\begin{itemize}
\item
For each $\alpha\in\R$, we define $\C_\alpha:=\set{z\in\C\mid \re z>\alpha}$ (so in particular   $\cplus =\C_0$).

\item  We let
$H^\infty(\C_\alpha;\cB(U,Y))$ denote the $\cB(U,Y)$-valued functions which are analytic and bounded on $\C_\alpha$.
\end{itemize}
Thus the {\em Schur class} consists of those
functions $F\in H^\infty(\cplus;\cB(U,Y))$ such that  $F(\lambda)$ is a contraction from $U$ into $Y$ for all $\lambda\in\cplus$, and in this case we write $F\in\cS_{U,Y}$. In fact, for convenience, we identify two analytic functions which coincide on some set in the intersection of their domains which has an interior cluster point. In particular, we write $F\in\cS_{U,Y}$ if the restriction $F\big|_{\dom{F}\bigcap \cplus}$ has a unique extension to a function in $\cS_{U,Y}$.

In the infinite dimensional situation, following \cite{StafBook} we use the frequency response idea at the beginning of the introduction to define the transfer function
$\widehat\Dfrak$ by the formula
$$
	\widehat\Dfrak(\lambda)u_0:=
		(\overline\Dfrak e_\lambda u_0)(0),
	\quad \lambda\in\C_{\omega_\Afrak},\ u_0\in U,
$$
where $\overline\Dfrak$ is a suitable version of the input/output map $\Dfrak$; see Proposition \ref{prop:Transfer} for the details. We can now formulate our first main result.

\begin{theorem}[Standard infinite dimensional bounded real lemma]\label{thm:stdlemma}
For a \emph{minimal} well-posed system $\Sigma=\sbm{\Afrak&\Bfrak\\\Cfrak&\Dfrak}$ with transfer function $\widehat\Dfrak$ the following are equivalent:
\begin{enumerate}
\item[(1)] The transfer function satisfies $\widehat\Dfrak\in\cS_{U,Y}$ (in the generalized sense described above).

\item[(2)] The continuous-time KYP-inequality has a `spatial' solution $H$ in the following sense: $H$ is a closed, possibly unbounded, densely defined, and positive definite operator on $X$, such that for all $t>0$:
\begin{equation}\label{eq:HhalfDomCond}
\begin{aligned}
	\Afrak^t\,\dom{H^{\frac{1}{2}}} \subset \dom{H^{\frac{1}{2}}},\quad
	\Bfrak^t\,L^2([0,t];U) \subset \dom{H^{\frac{1}{2}}},
\end{aligned}
\end{equation}
and the following \emph{spatial form} of the KYP-inequality holds:
\begin{equation}\label{eq:KYP}
	\left\|\bbm{H^{\frac{1}{2}}&0\\0&I}
		\bbm{\Afrak^t&\Bfrak^t\\\Cfrak^t&\Dfrak^t}\bbm{x\\\bu}\right\|
	\leq\left\|\bbm{H^{\frac{1}{2}}&0\\0&I}\bbm{x\\\bu}\right\|,\quad \bbm{x\\\bu}\in\bbm{\dom{H^{\frac{1}{2}}}\\L^2([0,t];U)},
\end{equation}
where the norms are those of $\sbm{X\\L^2([0,t];Y)}$ and $\sbm{X\\L^2([0,t];U)}$, respectively.
\item[(3)] The system $\Sigma$ is \emph{pseudo-similar} to a passive system.
\item[(4)] The system $\Sigma$ has a storage function.
\item[(5)] The system $\Sigma$ has a \emph{quadratic} storage function.
\end{enumerate}
When these equivalent conditions hold, an operator $H$ defining a quadratic storage function in item \textup{(5)} will also be a spatial solution of the
KYP-inequality in item \textup{(2)} and vice versa. For every pseudo-similarity $\Gamma$ to a passive system, the operator $H:=\Gamma^*\Gamma$ is a spatial solution to the KYP-inequality in item \textup{(2)} and it can serve as the operator defining the quadratic storage function in item \textup{(5)}.
\end{theorem}

Note that the spatial solution $H$ of the KYP-inequality in item (2) of the preceding theorem is required to be independent of $t$. 

In \S\ref{sec:L2maps} below (see in particular Definition \ref{def:L2min}), we will introduce the concept of $L^2$-exact controllability and $L^2$-exact observability for continuous-time systems, which are weaker than exact controllability and exact observability in infinite time, but still strong enough to guarantee a bounded solution of the KYP-inequality. Thus we get the following alternative infinite dimensional version of the standard bounded real lemma, a result which we believe is new in the continuous-time setting:

\begin{theorem}[$L^2$-minimal infinite dimensional bounded real lemma]     \label{thm:stdlemmaL2reg}
For an $L^2$-minimal well-posed system $\Sigma=\sbm{\Afrak&\Bfrak\\\Cfrak&\Dfrak}$ with transfer function $\widehat\Dfrak$, the following conditions are equivalent:
\begin{enumerate}
\item[(1)] The transfer function of $\Sigma$ satisfies $\widehat\Dfrak\in\cS_{U,Y}$.
\item[(2)] A \emph{bounded, strictly positive definite} solution $H$ to the following standard KYP-inequality exists:
\begin{equation}\label{eq:KYPbdd}
	\bbm{\Afrak^t&\Bfrak^t\\\Cfrak^t&\Dfrak^t}^*\bbm{H&0\\0&I}
	\bbm{\Afrak^t&\Bfrak^t\\\Cfrak^t&\Dfrak^t}
		\preceq
		\bbm{H&0\\0&I},\quad t\geq 0,
\end{equation}
with the adjoint computed w.r.t. the inner product in $L^2([0,t];K)$, where $K=U$ or $K=Y$.
\item[(3)] The system $\Sigma$ is \emph{similar} to a passive system.
\end{enumerate}
When these conditions hold, in fact $\cplus\subset\dom{\widehat\Dfrak}$, so that $\widehat\Dfrak|_{\cplus}$ is itself in $\cS_{U,Y}$, rather than just having a unique restriction-followed-by-extension in $\cS_{U,Y}$.

For each bounded, strictly positive definite solution $H$ to the KYP-inequality in item \textup{(2)}, the operator $\Gamma:=H^{\frac{1}{2}}$ establishes similarity to a passive system
as in item \textup{(3)}. Conversely, for every similarity $\Gamma$ in item \textup{(3)}, $H:=\Gamma^*\Gamma$ is a bounded, strictly positive definite solution to the KYP-inequality in item \textup{(2)}. 

All solutions $H$ to the spatial KYP-inequality in item \textup{(2)} of Theorem \ref{thm:stdlemma} are in fact bounded,
strictly positive definite solutions of \eqref{eq:KYPbdd}, and there exist bounded, strictly positive definite solutions $H_a$ and $H_r$ of \eqref{eq:KYPbdd} such that
\[
H_a \preceq H \preceq H_r.
\]
\end{theorem}

\begin{remark}
The $L^2$-minimality assumption in Theorem \ref{thm:stdlemmaL2reg} brings the results much closer to the finite dimensional setting, while only assuming
minimality makes the situation more subtle. For instance, while each pseudo-similarity provides a spatial solution to the
KYP-inequality \eqref{eq:KYP}, the converse may not hold, as it does not appear to be the case that every spatial KYP-solution $H$ can be used to define a
passive well-posed system $\Sigma'$ via \eqref{eq:introinter};   see the proof of Theorem \ref{thm:stdlemmaL2reg} for more details in the bounded case. 
Specifically, to prove strong continuity if the semigroup of the candidate passive system, more conditions
seem necessary. Also, assuming only minimality, there are results on a 'largest' and 'smallest' solution to the spatial KYP-solution, but these serve as extremal solutions only for subclasses of spatial KYP-solutions; see Remark \ref{R:ASconnect} below for more details.
\end{remark}

It is straightforward to formulate a naive infinite dimensional version of the strict BRL. While the implications (2) $\Leftrightarrow$ (3) and (2) $\Rightarrow$ (1) are then straightforward, the implication (1) $\Rightarrow$ (2) or (3) appears to require
some extra hypotheses.
We present three possible strengthenings of the hypothesis (1) so that the implication (1) $\Rightarrow$ (2) or (3) holds in the infinite dimensional setting.
The naive expectation is that one should strengthen the stability assumption on $A$ in the discrete-time case  to the assumption that the
operator $C_0$-semigroup be exponentially stable for the continuous-time case.  However this appears to be not sufficient in general.
We shall additionally assume that the operator $C_0$-semigroup $\{\Afrak^t \mid t \ge 0\}$ embeds into an operator $C_0$-group
$\{ {\widetilde \Afrak}^t \mid  t \in {\mathbb R}\}$ (meaning that $\{ {\widetilde \Afrak}^t \mid t \in {\mathbb R}\}$ is a $C_0$-group of operators
such that ${\widetilde \Afrak}^t = \Afrak^t$ for $t \ge 0$).   Equivalently, the $C_0$-semigroup \{$\Afrak^t \mid t \ge 0\}$
is such that $\Afrak^t$ is invertible for some $t>0$;  see Proposition \ref{P:semigroup-group} below for additional
information. We note that this invertibility condition always holds in finite dimensions, and hence the notions strict and semi strict collapse to one notion of strictness in the finite dimensional case.

 In addition we introduce auxiliary operators
$$
  \Cfrak^t_{1_X, A} \colon X \to L^2([0,t], X), \quad \Dfrak^t_{A,B} \colon L^2([0,t]; U) \to L^2([0,t]; X)
$$
given by
\begin{align*}
& \Cfrak^t_{1_X, A} \colon x \mapsto ( s \to 1_X \Afrak^s x =  \Afrak^s x)_{0 \le s \le t} \in L^2([0,t], X),   \notag \\
& \Dfrak_{A,B}^t \colon ( s \to \bu(s))_{0 \le s \le t} \mapsto 
\left(s \to \int_0^s \Afrak^{s-r} B \bu(r) \ud r\right)_{0 \le s \le t}.
\label{CDfrak-aux}
\end{align*}
Here $\sbm{ A \& B \\ C \& D }$ is the system node associated with the well-posed system (details in \S\ref{sec:system-node} below) and we shall be assuming that the
$C_0$-semigroup $\Afrak^t$ generated by $A$ is exponentially stable.  Under these conditions the state trajectories
$(\bu, \bx, \by)$ associated with $\Sigma$ are such that $\bx \in L^2({\mathbb R}^+, X)$ and $\by \in L^2({\mathbb R}^+, Y)$
as long as $\bu \in L^2({\mathbb R}^+, U)$.  In system-trajectory terms, the operator $\begin{bmatrix}  \Cfrak^t_{1_X, A} & \Bfrak^t_{A,B} \end{bmatrix}$
has the following property:  if $(\bu, \bx, \by)$ is any system trajectory, then
\begin{equation}   \label{x(0)u-x}
\begin{bmatrix}  \Cfrak^t_{1_X, A} & \Dfrak^t_{A,B} \end{bmatrix} \colon \begin{bmatrix} \bx(0) \\ \bu|_{[0,t]} \end{bmatrix} \to
\bx|_{[0,t]} \in L^2([0,t], X)
\end{equation}

Our version of the strict BRL for the infinite dimensional continuous-time setting is as follows:

\begin{theorem}[Infinite dimensional strict bounded real lemma]\label{thm:stdlemmastrict}
Consider the following statements for a well-posed system $\Sigma=\sbm{\Afrak&\Bfrak\\\Cfrak&\Dfrak}$:
\begin{enumerate}
\item[(1)] The transfer function $\widehat\Dfrak$ of $\Sigma$ is in $\cS_{U,Y}^0$ and $\cplus\subset\dom{\widehat\Dfrak}$.
\item[(2a)] There exists a bounded $H \succc 0$ on $X$ which  satisfies the strict KYP-inequality associated with $\Sigma$, i.e.,
there is a $\delta >0$ such that 
\begin{align}
&  \bbm{\Afrak^t&\Bfrak^t\\\Cfrak^t&\Dfrak^t}^*\bbm{H&0\\0&1_{L^2([0,t],Y)}} \bbm{\Afrak^t&\Bfrak^t\\\Cfrak^t&\Dfrak^t}  \notag \\
&  \quad + \delta \begin{bmatrix} (\Cfrak_{1_X,A}^t)^*) \\ (\Dfrak_{A,B}^t)^* \end{bmatrix} \begin{bmatrix}  \Cfrak_{1_X,A}^t &  \Dfrak_{A,B}^t \end{bmatrix}
  \preceq  \bbm{H&0\\0& (1 - \delta) 1_{L^2([0,t], U)}},\quad t>0.
\label{eq:strictKYP}
\end{align}

\item[(2b)] There exists a bounded $H \succc 0$ on $X$ which satisfies the semi-strict KYP-inequality for $\Sigma$, i.e.,
there is a $\delta > 0$ so that  for all $t> 0$ we have:
\begin{align}
&  \bbm{\Afrak^t&\Bfrak^t\\\Cfrak^t&\Dfrak^t}^*\bbm{H&0\\0&1_{L^2([0,t],Y)}} \bbm{\Afrak^t&\Bfrak^t\\\Cfrak^t&\Dfrak^t}
  \preceq  \bbm{H&0\\0& (1 - \delta) 1_{L^2([0,t], U)}}.
\label{eq:semi-strictKYP}
\end{align}
\item[(3a)] $\Sigma$ is \emph{similar} to a strictly passive system, i.e., one satisfying \eqref{eq:strictKYP}
with $H=1_X$ and some $\delta>0$.
\item[(3b)] $\Sigma$ is \emph{similar} to a semi-strictly passive system, i.e., one satisfying \eqref{eq:semi-strictKYP} with $H = 1_X$.

\item[(4a)] $\Sigma$ has a finite-valued, coercive, quadratic, strict storage function.

\item[(4b)]  $\Sigma$ has a finite-valued, coercive, quadratic, semi-strict storage function.

\item[(5a)] $\Sigma$ has a strict storage function.

\item[(5b)]  $\Sigma$ has a semi-strict storage function.
\end{enumerate}
Then  we have the following implications:
\[
\begin{array}{ccccccccc}
(2a)&\Longleftrightarrow&(3a)&\Longleftrightarrow&(4a)&\Longrightarrow&(5a)\\
\Downarrow&&\Downarrow&&\Downarrow&&\Downarrow\\
(2b)&\Longleftrightarrow&(3b)&\Longleftrightarrow&(4b)&\Longrightarrow&(5b)&\Longrightarrow&(1).
\end{array}
\]

Furthermore, all 9 statements in the list
\textup{(1)}--\,\textup{(5)} are equivalent if we assume in addition that
$\Afrak^t$ is exponentially stable and at least one of the following three conditions holds:
\begin{itemize}
\item[(H1)] $\Afrak^t$ can be embedded into a $C_0$-group;

\item[(H2)]  $\Sigma$ is $L^2$-controllable;

\item[(H3)] $\Sigma$ is $L^2$-observable.
\end{itemize}
\end{theorem}

\begin{remark}   \label{R:strict-storages}  Let us sketch here the connection between the strict operator KYP-inequality \eqref{eq:strictKYP}
and the strict storage-function inequality \eqref{eq:stordefstrict}.

As already observed in Remark \ref{R:HvsSH},
 $H \succeq 0$ being  bounded corresponds to the associated quadratic storage function $S_H(x) =
\| H^{\frac{1}{2}}x \|^2$ being finite-valued on $X$, and  $H \succc 0$ corresponds to $S_H$ being coercive.

 Given a  well-posed system $\Sigma$,  by the definition of the $\sbm{ \Afrak & \Bfrak \\ \Cfrak & \Dfrak}$  system trajectories, $(\bu, \bx, \by)$
 are determined from the initial condition $\bx(0) = x_0$ and the input signal $\bu$ according to
 \begin{align*}
   \bx(t) & = \Afrak^t x_0 + \Bfrak^t \bu|_{[0,t]}  \\
   \by(t) & = \Cfrak^t x_0 + \Dfrak^t \bu|_{[0,t]},\quad t \ge 0.
\end{align*}

If we look at the quadratic form coming from the selfadjoint operator on the left-hand side of the operator inequality
\eqref{eq:strictKYP} evaluated  at $\sbm{\bx(0) \\ \bu|_{[0,t]}}$ coming from a system trajectory $(\bu, \bx, \by)$, we get
\begin{align*}
& \langle H (\Afrak^t x_0 + \Bfrak^t \bu|_{[0,t]}), \Afrak^t x_0 + \Bfrak^t \bu|_{[0,t]} \rangle_X
 + \| \Cfrak^t x_0 + \Dfrak^t \bu|_{[0,t]} \|^2_{L^2([0,t], Y)}    \\
 & \quad + \delta \| \bx|_{[0,t]} \|^2_{L^2([0,t], X)}
 = \langle H \bx(t), \bx(t) \rangle_X + \| \by|_{[0,t]} \|^2_{L^2([0,t],Y)} + \delta \| \bx|_{[0,t]} \|^2_{L^2([0,t],X)}
\end{align*}
while the right-hand side gives us
$$
 \langle H \bx(0), \bx(0) \rangle_X + (1 - \delta) \| \bu \|^2_{L^2([0,t],U)}\,.
 $$
 Thus the strict KYP-inequality \eqref{eq:strictKYP} for a bounded $H \succc 0$, when viewed in terms of the respective quadratic forms evaluated at
 $\sbm{ \bx(0) \\ \bu|_{[0,t]}}$,
becomes exactly
\begin{align*}
& \langle H \bx(t), \bx(t) \rangle_X + \| \by|_{[0,t]} \|^2_{L^2([0,t],Y)} + \delta \| \bx|_{[0,t]} \|^2_{L^2([0,t],X)}  \\
& \quad \le  \langle H \bx(0), \bx(0) \rangle_X + (1 - \delta) \| \bu|_{[0,t]} \|^2_{L^2([0,t],U)}.
\end{align*}
Setting $S_H(x) = \| H^{\frac{1}{2}} x \|^2 = \langle H x , x \rangle$,
 we see that the last inequality is exactly the defining inequality \eqref{eq:stordefstrict} for $S_H$ to be a strict storage  function.  Thus
 {\em the class of bounded $H \succc 0$ satisfying the strict KYP inequality \eqref{eq:strictKYP} is exactly the class of $H$ for which
 the associated quadratic function $S_H$ is a finite-valued, coercive strict storage functions for $\Sigma$.}

A similar analysis gives the corresponding statement for the semi-strict setting:  {\em the class of bounded $H \succc 0$ satisfying the semi-strict KYP inequality
\eqref{eq:semi-strictKYP} is exactly the class of $H$ for which the associated quadratic function $S_H$ is a finite-valued, coercive, semi-strict storage function.}
\end{remark}

Arov and Staffans \cite{ArSt07} also treat the standard BRL for infinite dimensional, continuous-time systems (Theorem \ref{thm:stdlemma}
above), but from a complementary point of view.  There the authors introduce system nodes $\sbm{ A \& B \\ C \& D}$ first,  and then define the associated system
(and the associated operators $\Sigma = \sbm {\Afrak & \Bfrak \\ \Cfrak & \Dfrak}$) through smooth system trajectories associated with the system-node trajectories.
They introduce the notion of  pseudo-similarity at the level of system nodes and obtain the equivalence of pseudo-similarity to a dissipative system node
with the existence of a solution to a spatial KYP-inequality expressed directly in terms of the system node operators (a spatial infinite dimensional analogue of the spatial KYP-inequality \eqref{eq:KYPintro}).  To complete the analysis they use Cayley transform computation to reduce the result to the discrete-time
situation studied in \cite{ArKaPi06} (see  Remark \ref{R:ASdifKYP} below for additional details). In the present paper, on the other hand, all details are worked out
directly in the continuous-time systems setting rather than using Cayley transforms to map into discrete time. This is necessary in our stydy of the strict BRL, because 
exponential stability in continuous time is in general \emph{not}
mapped into exponential stability in discrete time; see Example \ref{ex:counter} below. 

We extend the concept of $L^2$-storage function originally introduced by Willems \cite{Will72a, Will72b} and developed further for discrete-time infinite dimensional systems in  \cite{BGtH18b} to
continuous-time, infinite dimensional systems. We show that Willems' available storage function $S_a$ (see \cite{Will72a,Will72b}) is of a special type which we call \emph{$L^2$-regular}, whereas Willems' required supply $S_r$ is not. In response to the latter, we introduce an $L^2$-regularized version $\underline S_r\leq S_r$ of the required supply and prove that all $L^2$-regularized storage
functions $S$ satisfy $S_a\leq S\leq \underline S_r$ under some additional assumptions. Moreover, we prove that $S_a$ and $\underline S_r$ are quadratic. Our variational approach to the explicit solution of the density operators determining $S_a$ and $\underline{S}_r$ in \S\ref{sec:storage} is much in the same spirit as in the discussion in \cite[\S3]{Pan96}. 

Extensions to the infinite dimensional, Hilbert space setting were begun already by Yakubovich in \cite{Yak74,Yak75}, but the theory has been systematized and refined in many iterations after these seminal papers. The paper of Curtain \cite{Cur93} for instance treats the strict BRL for the case where ``$B$ and $C$ are bounded" (i.e., $B \in \cB(U,X)$  and $C \in \cB(X,Y)$) and the resulting feedthrough operator $D \in \cB(U,Y)$ is taken to be $0$.   Her KYP-inequality
can be seen (via a Schur-complement calculation) to be contained in our strict KYP-inequality criterion (see \eqref{eq:strictKYPnode} below) when specialized to her situation.

In addition to the BRL as presented here, the so-called KYP lemma appears in the context of many other topics in control theory. e.g., the design of  a certain type of Lyapunov function leading to stabilization of a linear system via a nonlinear state-feedback control as in the original problem of Lur'e,  linear-quadratic optimization problems, feedback design, etc.; we refer to \cite{GuLi06} for an informative survey. The paper \cite{IwHa05} for instance gives a far-reaching extension of the original form of the KYP-lemma, allowing the FDIs to be given only on finite frequency intervals and the class of systems allowed to be more general, by exploiting the $S$-procedure, which also goes back to work of Yakubovich (see \cite{GaYa66,Yak71}).

The Bounded Real Lemma (more generally the KYP lemma) has now been adapted to a number of additional applications. Let us mention that, specifically, in \cite{GuOp13}, the bounded real lemma is applied to model reduction, more precisely to balanced bounded real truncation, and the relation of the minimal and maximal storage functions to optimal control theory is described; see also \cite{StafALSOpt} for this connection and an alternative version of the strict bounded real and positive real lemmas.
Finally, we mention that there is also an extension \cite{BGtH18c} of the present approach to discrete-time dichotomous and bicausal systems, where it is essential
that solutions of the KYP-inequality be indefinite; such a situation is considered for both discrete-time and continuous-time systems in \cite{Pro15} to handle applications where a stabilizability assumption is missing.  It should be of interest to extend the results here to the dichotomous setting, thereby getting a continuous-time analogue of \cite{BGtH18c}.

The paper is organized as follows. In \S\ref{sec:prel}, the basics of well-posed systems are recalled.  In \S\ref{sec:system-node} the complementary differential approach via system nodes is reviewed, because some issues coming up in the sequel are
more easily resolved via the system-node approach.   In \S\ref{sec:L2maps} we develop
the concept of $L^2$-minimality for the continuous-time setting (analogous to developments in \cite{BGtH18b} for the
discrete-time setting). Some examples of $L^2$-minimal systems are discussed in \S\ref{sec:L2minexamples}. 
In \S\ref{sec:storage}, we extend the concept of $L^2$-regularized storage function from \cite{BGtH18b} to continuous time and
we use this to study $S_a$ and $\underline S_r$. Finally, in \S \ref{sec:proofs} we prove our main results stated in the present introduction. Part of the proofs are based on an operator optimization problem, which is the topic of Appendix \ref{sec:OpOpt}.

\smallskip

\paragraph{\bf Notation and terminology.}
For $t \in \R$, we define the backward shift operator $\tau^t$ acting on a function $\bu$ with $\dom{\bu}\subset \R$ by
\[
(\tau^t \bu)(s)=\bu(t+s),\qquad s\in\R,\, t+s\in\dom{\bu}.
\]
Given $J\subset \R$, we define the projection $\pi_J$ acting on a function $\bu$ with $J\subset \dom \bu \subset \R$ by
\[
(\pi_Ju)(s):=\begin{cases} \bu(s),\quad s\in J, \\ 0,\quad s\in \R\setminus J.\end{cases}
\]
Set $\R^+:=[0,\infty)$ and $\R^-:=(-\infty,0)$. We abbreviate $\pi_+:=\pi_{\R^+}$, $\pi_-:=\pi_{\R^-}$ and define $\tau_+^t:=\pi_+\tau^t$  and $\tau_-^t:=\tau^t\pi_-$ for $t\geq0$, both acting on functions with support anywhere in $\R$. The multiplicative interaction between these operations is given by
\[
\tau^t \pi_J =\pi_{J+t} \tau^t,\quad t\in\R,\, J\subset\R,\quad \mbox{with } J+t:=\{x+t \mid x\in J\}.
\]
Furthermore, we let $\ya$ denote the reflection operator:
\begin{equation}  \label{ya}
(\ya \bu)(-t)=\bu(t), \quad t\in\dom{\bu}.
\end{equation}

Let $K$ be a Hilbert space. For every, not necessarily bounded, interval $J\subset\R$ we write $L^2(J;K)$ for the usual Hilbert space of $K$-valued measurable, square integrable functions on $J$ with values in $K$, considering this space as a subspace of $L^2_K:=L^2(\R;K)$ by zero extension, without writing out the injection explicitly. We abbreviate $L^{2+}_K:=L^2(\R^+;K)$, and $L^{2-}_K:=L^2(\R^-;K)$. With $L^2_{loc,K}$ we denote the
space of $K$-valued measurable functions $\bu$ such that $\pi_J \bu\in L^2_K$ for every bounded interval $J$.  The symbols $L^2_{\ell,K}$, $L^2_{r,K}$ and $L^2_{\ell,r,K}$ stand for the spaces of functions $\bu\in L^2_K$ with support bounded to the left ($\supp \bu \subset (L,\infty)$ for some $L\in\R$), support bounded to the right ($\supp \bu \subset (-\infty,L)$ for some $L\in\R$), or with support bounded on both sides, respectively. Similarly we define $L^2_{\ell,loc,K}$, $L^2_{r,loc,K}$, $L^{2\pm}_{loc,K}$, $L^{2\pm}_{\ell,K}$, etc. However, note that some spaces may coincide,
e.g., $L^2_{\ell,r,loc,K}=L^2_{\ell,r,K}$, $L^{2+}_{\ell,loc,K}=L^{2+}_{loc,K}$, $L^{2+}_{r,loc,K}=L^{2+}_{r,K}$, etc. Convergence of $\bz_k$ to $\bz$ in $L^2_{\ell,loc,K}$ means that there is some $L\in\R$ such that $\supp{\bz}, \supp{\bz_k}\subset (L,\infty)$ for all $k$, and $\pi_{[L,T]}\bz_k\to \bz$ in $L^2_K$ for all $T>L$, and convergence in $L^2_{r,loc,K}$ is defined similarly. Moreover, $L^{2-}_{\ell,K}=L^{2-}_{\ell,loc,K}$ and $L^{2+}_{loc,K}=L^{2+}_{\ell,loc,K}$ are considered as subspaces of $L^{2}_{\ell,loc,K}$ with support contained in $\overline\rminus$ and $\rplus$, respectively, and we let these spaces inherit the topology of $L^{2}_{\ell,loc,K}$. \
For an interval $J\subset \R$, we write $C(J,K)$ for the space of continuous functions on $J$ with values in $K$.

Throughout, for Hilbert spaces $U$ and $V$ we write $\cB(U,V)$ for the Banach space of bounded linear operators mapping $U$ into $V$ with the operator norm simply denoted by $\|\ \|$. For a contraction operator $T$ in $\cB(U,V)$, that is, with $\|T\|\leq 1$, we write $D_T$ for the {\em defect operator} of $T$ which is
defined to be the unique positive semidefinite square root of the bounded, positive semidefinite operator $I-T^*T$, i.e., $D_T := ( I - T^* T)^{\frac{1}{2}}$.

\section{Well-posed linear systems}\label{sec:prel}

In this section we provide some background on well-posed systems, more specifically, causal, time-invariant $L^2$-well-posed linear systems. We recall this class of systems in Definition \ref{def:WPsys}; for a more detailed study and motivation of this class of systems we refer the reader to \cite{StafBook}.
It may be a helpful experience for the reader to verify that the system determined by \eqref{eq:introiso} and \eqref{eq:introinteg} fits Definitions \ref{def:WPsys} and \ref{def:WPtraj} below.

\begin{definition}\label{def:WPsys}
Let $U$, $X$ and $Y$ be separable Hilbert spaces. A quadruple $\Sigma=\sbm{\Afrak&\Bfrak\\\Cfrak&\Dfrak}$ is called a \emph{well-posed system} if it has the following properties:
\begin{enumerate}
\item[(1)] The symbol $\Afrak$ indicates a family $t\mapsto \Afrak^t$, which is a $C_0$-semigroup on $X$.

\item[(2)] The \emph{input map} $\Bfrak:L^{2-}_{\ell,U}\to X$ is a linear map satisfying $\Afrak^t\Bfrak =\Bfrak\tau_-^t$ on $L^{2-}_{\ell,U}$, for all $t\geq0$.

\item[(3)] The \emph{output map} $\Cfrak:X\to L^{2+}_{loc,Y}$ is a linear map satisfying $\Cfrak\Afrak^t=\tau_+^t\Cfrak $ on $X$, for all $t\geq0$.

\item[(4)] The \emph{transfer map (input/output map)} $\Dfrak:L^2_{\ell,loc,U}\to L^2_{\ell,loc,Y}$ is a linear map satisfying the following identities on $L^2_{\ell,loc,U}$:
\begin{enumerate}
\item[(a)] $\tau^t\Dfrak =\Dfrak\tau^t$ for all $t\in\R$ (time invariance),
\item[(b)] $\pi_-\Dfrak\pi_+=0$ (causality) and
\item[(c)] $\pi_+\Dfrak\pi_-=\Cfrak\Bfrak\pi_-$ (Hankel operator factorization).
\end{enumerate}

\item[(5)] The operators $\Bfrak$, $\Cfrak$, and $\Dfrak$ are continuous with respect to the topology of $L^2_{\ell,loc}$.
\end{enumerate}

We remark that the intertwinement in condition (2) in the preceding definition, $\Afrak^t\Bfrak \bu =\Bfrak\tau_-^t \bu$ for $\bu\in L^{2-}_{\ell,U}$, is written in this form in \cite[Definition 2.2.1]{StafBook}, but in fact the projection in $\tau_-^t=\tau^t\pi_-$ is redundant for such $\bu$, since $\pi_-\bu=\bu$. It is also possible to consider $\Bfrak$ as an operator with domain $L^2_{\ell,loc,U}$, without breaking this intertwinement property, by setting $\Bfrak:=\Bfrak\pi_-$; however, we do \emph{not} make this convention here.
On the other hand, $\Cfrak$ can be interpreted as an operator from $X$ into $L^2_{\ell,loc,Y}$, since $L^{2+}_{loc,Y}$ can be identified with a subspace of $L^2_{\ell,loc,Y}$ by zero extension on $\rminus$.

Given the well-posed system $\Sigma$, we define
\begin{equation}\label{BCD^t}
\begin{aligned}
\Bfrak^t&:=\Bfrak\pi_-\tau^t|_{L^2([0,t],U)}:L^2([0,t],U)\to X,\quad  t\in\R^+,\\
\Cfrak^t&:=\pi_{[0,t]}\Cfrak:X\to L^2([0,t],Y),\quad t\in\R^+,\qquad\text{and}\\
\Dfrak^t&:=\pi_{[0,t]}\Dfrak|_{L^2([0,t],U)}:L^2([0,t],U)\to L^2([0,t],Y),\quad t\in\R^+.
\end{aligned}
\end{equation}
\end{definition}

In order to stay compatible with the notation in \cite{StafBook}, we abbreviate $\Bfrak^t\pi_{[0,t]}\bu$ to $\Bfrak^t \bu$, so that we can apply $\Bfrak^t$ to arbitrary $\bu\in L^2_{loc,U}$ rather than only $\bu\in L^2([0,t];U)$. Note that we divert in \eqref{BCD^t} from the notation in \cite[Definition 2.2.6]{StafBook}: what we define as $\Bfrak^t$, $\Cfrak^t$ and $\Dfrak^t$ corresponds to $\Bfrak^t_0$, $\Cfrak^t_0$ and $\Dfrak^t_0$ in \cite{StafBook}, with the additional feature that we restrict $\Bfrak^t$ and $\Dfrak^t$ to functions in $L^2([0,t],U)$.

With a slight modification of the formulas in \cite[Theorem 2.2.14]{StafBook} it is possible to recover $\Bfrak$, $\Cfrak$ and $\Dfrak$ from $\Bfrak^t$, $\Cfrak^t$ and $\Dfrak^t$ via:
\begin{equation}\label{t-back}
\begin{aligned}
\Bfrak \bu= \lim_{t\to \infty} \Bfrak^t \tau^{-t} \pi_{[-t,0]}\bu, \ \ \bu\in L^{2-}_{\ell,U},&\quad \Cfrak x=\lim_{t\to\infty} \Cfrak^t x,\ \ x\in X,\\
\Dfrak \bu = \lim_{t\to \infty}  \tau^{t} \Dfrak^{2t} \tau^{-t} \pi_{[-t,t]}& \bu, \quad  \bu\in L^{2}_{\ell,loc,U}.
\end{aligned}
\end{equation}
The limits for $\Bfrak$ and $\Cfrak$ follow from Theorem 2.2.14 in \cite{StafBook}. For $\Dfrak$, a slightly different argument is needed, which we will now give. Fix $\bu\in L^2_{\ell,loc,U}$ and let $L$ be such that $\supp \bu\subset [L,\infty)$. For all $t>|L|$, we then get from the time invariance and causality of $\Dfrak$ that
\begin{align*}
\tau^{t} \Dfrak^{ 2t} \tau^{-t} \pi_{[-t,t]} \bu
& = \tau^{t} \pi_{[0,2t]} \Dfrak \tau^{-t} \pi_{[L,t]}   \bu
=  \pi_{[-t,t]} \tau^{-t} \tau^{t} \Dfrak  \pi_{[L,t]}   \bu\\
&=  \pi_{[L,t]} \Dfrak  \pi_{[L,t]} \bu.
\end{align*}
Now fix $T>L$ arbitrarily. When $t\to\infty$, we get $\pi_{[L,T]}\pi_{[L,t]}\bu=\pi_{[L,T]}\bu$ for all $t>T$, so that $\pi_{[L,t]}\bu\to \bu$ in $L^2_{\ell,loc,U}$. By the continuity of $\Dfrak$, we then get for $t>\max\{|L|,|T|\}$ that
$$
	\pi_{[L,T]}\tau^{t} \Dfrak^{2t} \tau^{-t} \pi_{[-t,t]} \bu = \pi_{[L,T]}\Dfrak \pi_{[L,t]} \bu \to
	 \pi_{[L,T]}\Dfrak \bu.
$$
Hence, in $L^2_{\ell,loc,Y}$, we have
\[
\lim_{t\to\infty} \tau^{t} \Dfrak^{2t} \tau^{-t} \pi_{[-t,t]} \bu = \Dfrak \bu.
\]

Next we define what we mean by a solution, or a trajectory, of a well-posed system. 

\begin{definition}\label{def:WPtraj}
By an \emph{(input/state/output) trajectory on $\rplus$} of a well-posed linear system $\Sigma$ with \emph{initial state} $x_0\in X$, we mean a triple $(\bu,\bx,\by)$ with \emph{input signal} $\bu\in L^{2+}_{loc,U}$, \emph{state signal} $\bx\in C(\rplus;X)$ and \emph{output signal} $\by\in L^{2+}_{loc,Y}$ that satisfies
\begin{equation}\label{eq:trajonRplus}
\begin{aligned}
	\bx(t) &= \Afrak^tx_0+\Bfrak^t \pi_{[0,t]} \bu,\quad t\geq 0, \\
	\by &=\Cfrak x_0+\Dfrak  \pi_+ \bu =\Cfrak x_0+\Dfrak \bu.
\end{aligned}
\end{equation}
By an \emph{(input/state/output) trajectory of $\Sigma$ on $\R$} (with initial state $x_{-\infty}=0$) we mean a triple $(\bu,\bx,\by)$ with input signal $\bu\in L^2_{\ell,loc,U}$, state trajectory $\bx\in C(\R;X)$ and output signal $\by\in L^2_{\ell,loc,Y}$ that satisfies
\begin{equation}\label{eq:trajonR}
	\bx(t):=\Bfrak\pi_- \tau^t \bu,\quad t\in\R,\qquad \by:=\Dfrak \bu.
\end{equation}
\end{definition}

Note that a trajectory $(\bu,\bx,\by)$ on $\rplus$ is uniquely determined by the initial state $x_0$ and the input signal $\bu$, while a trajectory on $\R$ is uniquely determined by $\bu$, and then one can intuitively think of $\lim_{t\to-\infty} \bx(t)=0$ as a kind of initial state. We mention a few rules on how trajectories on $\R$ and $\R^+$ can be manipulated, which will be useful in the sequel. The proofs are straightforward and left to the reader.

\begin{enumerate}
  \item[(1)] If $(\bu,\bx,\by)$ is a trajectory on $\R$ and $t\in\R$ with $\bx(t)=0$, then $\pi_{[t,\infty)}(\bu,\bx,\by)$ is also a trajectory on $\R$.

 \item[(2)] A triple $(\bu,\bx,\by)$ is a trajectory on $\R$ if and only if the support of $\bu$ is bounded to the left by some $t\in\R$ and $\tau^t( \bu, \bx, \by)$ is a trajectory on
 $\rplus$ with initial state zero.

\item[(3)] The triple $(\bu,\bx,\by)$ is a trajectory on $\R$ if and only if $\tau^s(\bu,\bx,\by)$ is a trajectory on $\R$ for some/all $s\in\R$.

\item[(4)] If $(\bu,\bx,\by)$ and $(\bv,\bz,\bw)$ are trajectories on $\rplus$ and $\bx(t)=\bz(0)$ for some $t>0$ then $\pi_{[0,t)}(\bu,\bx,\by)+ \tau^{-t}(\bv,\bz,\bw)$ is a trajectory on $\rplus$.

\item[(5)] If $(\bu,\bx,\by)$ is a trajectory on $\R$ and $(\bv,\bz,\bw)$ is a trajectory on $\rplus$ with $\bz(0)=\bx(0)$ then $\pi_-(\bu,\bx,\by)+(\bv,\bz,\bw)$ is a trajectory on $\R$.
\end{enumerate}

In order to discuss additional features of the well-posed system $\Sigma$, we need an alternative representation of $\Bfrak$, $\Cfrak$ and $\Dfrak$, as bounded linear Hilbert space operators, and we now proceed to construct this representation. First set $e_\lambda(t):=e^{\lambda t}$ for $\lambda\in\C,$ $t\in\R$, and define the Hilbert space  $L^2_{\omega,K}$ by
$$
L^2_{\omega,K}=\set{ e_{\omega} \bu \mid \bu\in L^2_K} \text{ with }
\Ipdp{e_{\omega} \bu}{e_{\omega} \bv}_{L^2_{\omega,K}}:=\Ipdp{\bu}{\bv}_{L^2_K} \text{ for } \bu,\bv\in L^2_K.
$$
Similarly we define $L^{2\pm}_{\omega,K}$ by
replacing $L^2_K$ by $L^{2\pm}_K$. Note that, as sets, we have the inclusions $L^2_{\ell,r,K}\subset L^2_{\omega,K}\subset L^2_{loc,K}$ for all $\omega\in\R$, with each
inclusion being dense in their respective topologies, with similar dense inclusions for the corresponding $L^{2\pm}$--spaces.

It is well-known, see e.g., Theorem 2.5.4 in \cite{StafBook}, that every $C_0$-semigroup $\Afrak$ has a  \emph{growth bound}
\begin{equation}\label{eq:omegaAdef}
	\omega_\Afrak:=\lim_{t\to\infty} \frac{\ln\|\Afrak^t\|}{t}<\infty,
\end{equation}
meaning that for every $\omega>\omega_\Afrak$ there is some $M>0$ such that $\|\Afrak^t\|\leq Me^{\omega t}$ for all $t\geq0$. We call $\Sigma$, or $\Afrak$, \emph{exponentially stable} if $\omega_\Afrak<0$. In this connection, we also point out that a passive system has a \emph{contractive semigroup}, i.e., $\|\Afrak^t\|\leq 1$ for all $t\geq0$, and this implies that $\omega_\Afrak\leq0$. In particular, all $\alpha\in\C_{\omega_\Afrak}$ lie in the \emph{resolvent set} $\rho(A)$ of the generator $A$ of $\Afrak$, meaning that $\alpha-A$ has a bounded inverse on the state space $X$; see \cite[Theorem 3.2.9(i)]{StafBook}.

Fix a real number $\omega$. In case $\omega>0$, then $L^{2-}_{\omega,K}\subset L^{2-}_{K}$ with dense and continuous embedding, and $L^{2-}_{-\omega,K}$ is the dual of $L^{2-}_{\omega,K}$ with pivot space $L^{2-}_{K}$, so that the duality pairing satisfies
\begin{equation}\label{eq:omegadual}
\Ipdp \bv\bu_{L^{2-}_{-\omega,K},L^{2-}_{\omega,K}}= \Ipdp \bv\bu_{L^{2-}_{K}},\quad \bv\in L^{2-}_{K},\, \bu\in L^{2-}_{\omega,K}.
\end{equation}
See for instance \cite[\S3.6]{StafBook} or \cite[\S2.9]{TuWeBook} for detailed constructions of the dual with respect to a pivot space. If we have an exponentially stable system, then it is possible to take $\omega=0$ and in that case the three spaces in \eqref{eq:omegadual} coincide. In fact, for an exponentially stable system it is possible to take $\omega<0$, in which case instead $L^{2-}_{-\omega,K}$ is the densely and continuously embedded subspace and $L^{2-}_{\omega,K}$ is the dual subspace of $L^{2-}_{K}$. Then, for $L^{2+}_K$ the embeddings are reversed, so that $L^{2+}_{-\omega,K}\subset L^{2+}_{K}\subset L^{2+}_{\omega,K}$ and $L^{2+}_{\omega,K}\subset L^{2+}_{K}\subset L^{2+}_{-\omega,K}$, and duality pairings with respect to the pivot space $L^{2+}_{K}$ exist in analogy to \eqref{eq:omegadual}.

Let now $\Sigma = \sbm{ \Afrak & \Bfrak \\ \Cfrak & \Dfrak}$ be a well-posed system  and fix a real number $\omega$ with $\omega>\omega_\Afrak$.
By Theorem 2.5.4 in \cite{StafBook}, $\range\Cfrak$ is contained in $L^{2+}_{\omega,Y}$, while $\Bfrak$ extends to a unique continuous linear operator from $L^{2-}_{\omega,U}$ into $X$, and the restriction of $\Dfrak$ to $L^2_{\ell,loc,U}\bigcap L^{2}_{\omega,U}$ has a unique linear extension that maps $L^{2}_{\omega,U}$ continuously into $L^{2}_{\omega,Y}$. We can thus reinterpret the operators $\Bfrak$, $\Cfrak$ and $\Dfrak$ as
\begin{equation}\label{eq:FrakTildeDef}
	\widetilde\Bfrak\in\cB(L^{2-}_{\omega,U},X), \quad
	\widetilde\Cfrak\in\cB(X,L^{2+}_{\omega,Y}),\quad
	\widetilde\Dfrak\in\cB(L^2_{\omega,U},L^2_{\omega,Y}),
\end{equation}
and this reinterpretation can also be reversed, so that the original three operators can be recovered from their tilde versions. In case the operators $\Bfrak$, $\Cfrak$ and $\Dfrak$ can be reinterpreted in the above fashion as bounded operators as in \eqref{eq:FrakTildeDef}, then we say that $\Bfrak$, $\Cfrak$ and $\Dfrak$ are {\em $\omega$-bounded}, respectively. Moreover, the $C_0$-semigroup $\Afrak^t$ is called {\em $\omega$-bounded} in case $\sup_{t\geq 0}\|e^{-\omega t}\Afrak^t\| <\infty$.

The following proposition shows how the {\em frequency-response-function} approach at the beginning of the introduction can be used to define a transfer function for an infinite dimensional well-posed system $\Sigma$ directly via the integrated system operators $\Afrak$, $\Bfrak$, $\Cfrak$, $\Dfrak$, thereby avoiding completely the system node $\bS = \sbm{ A \& B \\ C \& D}$ to be discussed in \S\ref{sec:system-node}.

\begin{proposition}\label{prop:Transfer}
For a well-posed system $\Sigma=\sbm{\Afrak&\Bfrak\\\Cfrak&\Dfrak}$ and for all $\omega>\omega_\Afrak$, $\widetilde \Dfrak$ uniquely induces an operator $\overline\Dfrak: H^1_{\omega,loc}(\R;U) \to  H^1_{\omega,loc}(\R;Y)$, where
\begin{equation}\label{eq:H1def}
	 H^1_{\omega,loc}(\R;K):=\set{f\in L^2_{loc,K}\mid
	 	\dot f\in L^2_{loc,K},~\pi_-f\in L^{2-}_{\omega,K}},
\end{equation}
and the action of $\overline\Dfrak$ is independent of $\omega>\omega_\Afrak$.
The \emph{transfer function} $\widehat\Dfrak$ of $\Sigma$, given by
$$
	\widehat\Dfrak(\lambda)u_0:=
		(\overline\Dfrak e_\lambda u_0)(0),
	\quad \lambda\in\C_{\omega_\Afrak},\ u_0\in U,
$$
is well-defined and when restricted to the half-plane $\C_{\omega}$, for $\omega>\omega_\Afrak$, gives a function in $H^\infty(\C_\omega;\cB(U,Y))$. Furthermore we recover the Laplace-transform interpretation \eqref{transfunc-freq}  of $\widehat \Dfrak (\lambda)$ as follows:
for $\bu \in L^{2+}_{\omega,U}$ we have
\begin{equation}   \label{CTtransfunc}
 \widehat{ \Dfrak \bu}(\lambda) =  \widehat \Dfrak(\lambda) \widehat \bu(\lambda),\quad \lambda \in {\mathbb C}_\omega.
\end{equation}
\end{proposition}

Proposition \ref{prop:Transfer} follows from Lemmas 4.5.1, 4.5.3 and 4.6.2 and Corollary 4.6.10 together with Definition 4.6.1 in \cite{StafBook}.
We emphasize that the domain of the transfer function defined in Proposition \ref{prop:Transfer} is $\C_{\omega_\Afrak}$, and at
the same time remind the reader that we identify two analytic functions agreeing on a set of points in the intersection of their
respective domains having a common interior cluster point. The key starting point to the preceding proposition is that
$$
	\frac{\tau^h\overline\Dfrak \bu-\overline\Dfrak \bu}{h}=
	\overline\Dfrak \frac{\tau^h\bu-\bu}{h},
$$
due to time invariance; see the proof of \cite[Lemma 4.5.1]{StafBook}.

Let us  identify the spaces $X$, $U$ and $Y$ with their duals.  Then the adjoints of the operators in \eqref{eq:FrakTildeDef} with respect to the appropriate duality pairings  belong to the following spaces:
$$
	\widetilde\Bfrak^*\in\cB(X,L^{2-}_{-\omega,U}), \quad
	\widetilde\Cfrak^*\in\cB(L^{2+}_{-\omega,Y},X),\quad
	\widetilde\Dfrak^*\in\cB(L^2_{-\omega,Y},L^2_{-\omega,U}).
$$
Since $\widetilde\Bfrak$, $\widetilde\Cfrak$ and $\widetilde\Dfrak$ are bounded linear Hilbert space operators, their adjoints are well defined. Noting that $L^{2-}_{-\omega,U}\subset L^{2-}_{loc,U}$ and $L^{2+}_{r,Y}\subset L^{2+}_{-\omega,Y}$, we can view the adjoints as operators of the following forms:
\begin{equation}\label{eq:SunOpDef}
\begin{aligned}
\Bfrak^\circledast:=\widetilde\Bfrak^*:X\to L^{2-}_{loc,U},\quad &
\Cfrak^\circledast:=\widetilde\Cfrak^*\big|_{L^{2+}_{r,Y}}: L^{2+}_{r,Y}\to X,\\ \Dfrak^\circledast:=\widetilde \Dfrak^*|_{L^2_{r,loc,Y}}&:L^2_{r,loc,Y}\to L^2_{r,loc,U};
\end{aligned}
\end{equation}
using \cite[Theorem 6.2.1]{StafBook}, we indeed see that $\widetilde\Dfrak^*$ has a restriction followed by an extension to an operator that maps $L^2_{r,loc,Y}$ continuously into $L^2_{r,loc,U}$. Using the reflection operator $\ya$ as in \eqref{ya}, we define the \emph{causal dual system} $\Sigma^d$ of $\Sigma$ via
\begin{equation}\label{eq:CausalDualDef}
	\Sigma^d=\bbm{\Afrak^d&\Bfrak^d\\
		\Cfrak^d& \Dfrak^d}:=\bbm{\Afrak^*&\Cfrak^\circledast\ya\\
		\ya\,\Bfrak^\circledast&\ya\,\Dfrak^\circledast\ya},
\end{equation}
where $\Afrak^*$ is the \emph{dual semigroup} of $\Afrak$, i.e., $(\Afrak^*)^t=(\Afrak^t)^*$, $t\geq0$. Here we depart from \cite{BGtH18b} by using the causal dual system instead of the anti-causal dual system, which would not have the reflections $\ya$ in \eqref{eq:CausalDualDef}. The reason is that we prefer to have all of the theory in \cite{StafBook} at our disposal.

\begin{theorem}\label{T:causdual}
Let $\Sigma=\sbm{\Afrak&\Bfrak\\\Cfrak&\Dfrak}$ be a well-posed system. Then the causal dual system $\Sigma^d$ of $\Sigma$ is a well-posed system with input space $Y$, state space $X$ and output space $U$. Moreover, the causal dual of $\Sigma^d$ is equal to $\Sigma$ and the transfer function of $\Sigma^d$ is $\widehat\Dfrak^d(\lambda)=\widehat\Dfrak(\overline\lambda)^*$, $\overline\lambda\in\rho(A)$, and in particular, $\|\widehat\Dfrak^d\|_{H^\infty(\C_\omega;\cB(Y,U))}=\|\widehat\Dfrak\|_{H^\infty(\C_\omega;\cB(U,Y))}$ for all $\omega>\omega_\Afrak$. If $\Sigma$ is passive, then $\Sigma^d$ is passive too.
\end{theorem}

For the proof, see Theorems 6.2.3, 6.2.13 and Lemma 11.1.4 in \cite{StafBook}.

\begin{lemma}\label{L:dualadj}
Let $\Sigma=\sbm{\Afrak&\Bfrak\\\Cfrak&\Dfrak}$ be a well-posed system with causal dual system $\Sigma^d=\sbm{\Afrak^d&\Bfrak^d\\\Cfrak^d&\Dfrak^d}$. Define $\Bfrak^t$, $\Cfrak^t$ and $\Dfrak^t$ as in \eqref{BCD^t} and define $(\Bfrak^d)^t$, $(\Cfrak^d)^t$ and $(\Dfrak^d)^t$ analogously for the dual system $\Sigma^d$. Then
\[
\begin{bmatrix} (\Afrak^d)^t & (\Bfrak^d)^t\\ (\Cfrak^d)^t & (\Dfrak^d)^t  \end{bmatrix}^* =
\begin{bmatrix}
1_X & 0 \\
0 &  \Lambda^t_Y
\end{bmatrix}^*
\begin{bmatrix} \Afrak^t & \Bfrak^t\\ \Cfrak^t & \Dfrak^t  \end{bmatrix}
\begin{bmatrix}
1_X & 0 \\
0 &  \Lambda^t_U
\end{bmatrix},\quad t>0,
\]
where for a separable Banach space $K$ we define $\Lambda^t_K\in\cB(L^2([0,t],K))$ to be the unitary operator given
by $\Lambda^t_K=\tau^{-t}\ya|_{L^2([0,t],K)}$. 
\end{lemma}

\begin{proof}
We need to prove that for each $t>0$:
\[
((\Afrak^d)^t)^*=\Afrak^t,~ 
((\Cfrak^d)^t)^*=\Bfrak^t \Lambda^t_U,~
((\Bfrak^d)^t)^*=(\Lambda^t_Y)^* \Cfrak^t,~ ((\Dfrak^d)^t)^*= (\Lambda^t_Y)^* \Dfrak^t \Lambda^t_U.
\]
The first identity follows directly from the definition of $(\Afrak^d)^t$. Next note that
\[
(\Cfrak^d)^t=\pi_{[0,t]} \ya \Bfrak^\circledast:X\to L^{2}([0,t],U).
\]
Thus, for all $t>0$, $\bu\in L^{2}([0,t],U)$ and $x\in X$ we have
\begin{align*}
\Ipd{(\Cfrak^d)^t x}{\bu}_{ L^{2}([0,t],U)}
& =\Ipd{\pi_{[0,t]} \ya \Bfrak^\circledast x}{\bu}_{L^{2+}_U}
=\Ipd{\pi_{[0,t]} \ya \widetilde\Bfrak^* x}{\bu}_{L^{2+}_{-\omega,U},L^{2+}_{\omega,U}} \\
&=\Ipd{   x}{\widetilde\Bfrak\ya\pi_{[0,t]} \bu}_X =\Ipd{   x}{\widetilde\Bfrak\ya \bu}_X
=\Ipd{   x}{\Bfrak\ya \bu}_X
\end{align*}
using that $\Bfrak$ and $\widetilde\Bfrak$ coincide on $L^{2-}_{\ell,U}$ in the last step. It thus follows for $t>0$ and $\bu\in L^{2}([0,t],U)$ that
\begin{align*}
((\Cfrak^d)^t)^* \bu & =\Bfrak \ya \bu =\Bfrak \tau^t \tau^{-t} \ya \bu = (\Bfrak \pi_- \tau^t) \tau^{-t} \ya \bu = \Bfrak^t \Lambda^t_U \bu,
\end{align*}
and this proves the second identity. The third identity follows by an almost identical argument.

It remains to prove the last identity. For this purpose, let $\by\in L^{2}([0,t],Y)$ and $\bu\in L^{2}([0,t],U)$. Then
\begin{align*}
\Ipd{(\Dfrak^d)^t \by}{\bu}_{ L^{2}([0,t],U)} & =\Ipd{\pi_{[0,t]}\Dfrak^d \by}{\bu}_{ L^{2}([0,t],U)}
=\Ipd{\pi_{[0,t]}\ya\,\Dfrak^\circledast\ya \by}{\bu}_{ L^{2}([0,t],U)}\\
& =\Ipd{\ya\pi_{[-t,0]}\Dfrak^\circledast\ya \by}{\bu}_{ L^{2}([0,t],U)}\\
&=\!\Ipd{\pi_{[-t,0]}\Dfrak^\circledast\ya \by}{\ya \bu}_{ L^{2}([-t,0],U)}\!\!
=\!\Ipd{\widetilde{\Dfrak}^*\ya \by}{\ya \bu}_{L^{2-}_{-\omega,U},L^{2-}_{\omega,U}}\\
&=\Ipd{ \by}{\ya\widetilde{\Dfrak} \ya \bu}_{L^{2-}_{-\omega,Y},L^{2-}_{\omega,Y}}
=\Ipd{ \by}{\pi_{[0,t]}\ya\,\Dfrak \ya \bu}_{ L^{2}([0,t],Y)}.
\end{align*}
It follows that
\begin{align*}
((\Dfrak^d)^t)^*\bu & = \pi_{[0,t]}\ya\,\Dfrak \ya \bu = \ya\pi_{[-t,0]}\tau^{t}\Dfrak \tau^{-t} \ya \bu
= \ya\tau^{t}\pi_{[0,t]}\Dfrak \tau^{-t} \ya \bu\\
&= (\Lambda^t_Y)^*\Dfrak^t \Lambda^t_U \bu,
\end{align*}
which proves the last identity.
\end{proof}

The following notions will be important in the sequel:

\begin{definition}\label{def:WPmin-approx}
A well-posed system $\Sigma=\sbm{\Afrak&\Bfrak\\\Cfrak&\Dfrak}$ is \emph{(approximately) controllable} if the finite-time reachable subspace
\[
\operatorname{Rea}(\Sigma):=\range\Bfrak=\textup{span}\left\{\range{((\Cfrak^d)^t)^*} \mid t>0 \right\}
\]
is dense in $X$. Following \cite{ArNu96}, we say that the system $\Sigma$ is \emph{(approximately) observable} if the finite-time observable subspace
\[
\operatorname{Obs}(\Sigma):=\textup{span}\left\{\range{(\Cfrak^t)^*} \mid t>0 \right\} = \range{\Bfrak^d}
\]
is dense in $X$, and it is \emph{(approximately) minimal} if it is both controllable and observable.
\end{definition}

Note that the equalities in the definitions of $\textup{Rea}\,(\Sigma)$ and $\textup{Obs}\,(\Sigma)$ are dual, and that they follow directly from Lemma \ref{L:dualadj} and formulas 
\eqref{t-back}, and that these equalities imply the following corollary:

\begin{corollary}\label{cor:ApproxContrObsDual}
The well-posed system $\Sigma$ is controllable (resp.\ observable) if and only if $\Sigma^d$ is observable (resp.\ controllable). In particular, $\Sigma$ is minimal if and only if $\Sigma^d$ is minimal.
\end{corollary}

The following lemma shows that our definitions agree with the other common definitions of controllability and observability:

\begin{lemma}\label{lem:ObsControl1-1}
The well-posed system $\Sigma$ is controllable if and only if $\Bfrak^d$ is one-to-one and observable if and only if $\Cfrak$ is one-to-one.
\end{lemma}

\begin{proof}
We prove the statement regarding observability; for controllability  the claim follows by duality. For $x\in X$ we have
$$
\begin{aligned}
\Cfrak x=0 \quad & \Longleftrightarrow\quad  \Cfrak^t x=\pi_{[0,t]} \Cfrak x =0\quad \mbox{for all $t>0$}\\
& \Longleftrightarrow \quad \Ipdp{\Cfrak^t x}{y} =0 \quad \mbox{for all $t>0$ and $y\in L^{2}([0,t],Y)$}\\
&\Longleftrightarrow\quad x\perp \range{(\Cfrak^t)^*} \quad \mbox{for all $t>0$}\\
&  \Longleftrightarrow\quad x\perp \textup{Obs}\,(\Sigma),
\end{aligned}
$$
which proves our claim.
\end{proof}

\section{The $L^2$-input and $L^2$-output maps of a well-posed linear system}\label{sec:L2maps}

The concepts of $\ell^2$-exact controllability, $\ell^2$-exact observability, and $\ell^2$-exact minimality were recently introduced for discrete-time systems in \cite{BGtH18a}. We will now extend these concepts to well-posed continuous-time systems.

Define the (in general unbounded) \emph{$L^2$-output map} as
\begin{equation}\label{Wo}
\begin{aligned}
\dW_o  := \Cfrak\big|_{\dom{\dW_o}}:X\supset \dom{\dW_o }\to L^{2+}_Y,\\
\text{with }\dom{\dW_o } := \set{x\in X\mid \Cfrak x\in L^{2+}_Y};
\end{aligned}
\end{equation}
i.e., we restrict $\Cfrak$ to the $x\in X$ that are mapped into $L^{2+}_Y$, rather than into $L^{2+}_{loc,Y}$, and view the resulting operator as mapping with codomain $L^{2+}_Y$. Note that $\Ker{\dW_o}=\Ker{\widetilde\Cfrak}=\Ker{\Cfrak}$ and hence $\Sigma$ is observable if and only if $\dW_o $ is one to one, or equivalently, if and only if $\widetilde\Cfrak$ is one-to-one.

\begin{proposition}\label{P:Wo}
Let $\dW_o$ be the $L^2$-output map of a well-posed system $\Sigma=\sbm{\Afrak&\Bfrak\\\Cfrak&\Dfrak}$. Then $\dW_o$ is closed.  Additionally assume that $\dW_o$ is densely defined. In this case:
\begin{enumerate}
\item[(1)] The operator $\dW_o$ has a closed and densely defined adjoint $\dW_o^*$.

\item[(2)] A function $\by\in L^{2+}_Y$ lies in $\dom{\dW_o^*}$ if and only if there exists an $x_o\in X$ such that 
\begin{equation}\label{eq:Wo*}
\lim_{t\to\infty}\Ipdp{x}{
\Bfrak^d\pi_{[-t,0]}\ya \by}_X
=\Ipdp{x}{x_o}_X,\quad x\in\dom{\dW_o}.
\end{equation}
When $\by\in\dom{\dW_o^*}$, we have $\dW_o^* \by=x_o$, where $x_o$ is given by \eqref{eq:Wo*}.

\item[(3)]  It holds that $L^{2+}_{r,Y}\subset \dom{\dW_o^*}$, that $\dW_o^*\big|_{L^{2+}_{r,Y}}=\Bfrak^{d}\ya$, and that $\dW_o^* L^{2+}_{r,Y}= \range{\Bfrak^d}=\textup{Obs}\,(\Sigma)$.

\item[(4)] For all $s>0$ and $\by\in \dom{\dW_o^*}$ we have
\[
\tau^{-s} \dom{\dW_o^*} \subset \dom{\dW_o^*},\quad \dW_o^* \tau^{-s}\by=(\Afrak^{s})^* \dW_o^* \by.
\]
\end{enumerate}
\end{proposition}

Before giving the proof, we remark that by Lemma \ref{L:dualadj}, the limit in \eqref{eq:Wo*} can be rewritten as
\begin{equation}\label{eq:Wo*2}
\lim_{t\to\infty}\Ipdp{x}{\Bfrak^d\pi_{[-t,0]}\ya \by}_X
=\lim_{t\to\infty}\Ipdp{x}{(\Cfrak^t)^*\pi_{[0,t]} \by}_X,
\end{equation}
because the expressions inside of the limit operators are the same.

\begin{proof}
To see that $\dW_o$ is closed, let $\dom{\dW_o} \ni x_k\to x$ in $X$ and $\dW_o x_k \to \by$ in $L^{2+}_Y$. Fix $b>0$ arbitrarily and observe that $\pi_{[0,b]}\Cfrak$ is a bounded operator from $X$ to $L^{2_+}_{Y}$, by the well-posedness of $\Sigma$. Hence
\[
\pi_{[0,b]} \Cfrak x
= \lim_{k\to\infty} \pi_{[0,b]} \Cfrak x_k
= \lim_{k\to\infty} \pi_{[0,b]} \dW_o x_k
=  \pi_{[0,b]} \by.
\]
Now let $b\to\infty$ to get that $\Cfrak x=\by\in L^{2+}_{Y}$. This shows that $x\in\dom{\dW_o}$ and $\dW_o x =\by$. Hence $\dW_o$ is closed, as claimed. 

In the remainder of the proof we assume that $\dom{\dW_o}$ is dense in $X$ and we prove items (1)--(4). Note that item (1) follows directly from \cite[Theorems 13.9 and 13.12]{RudinFA73}, since $\dW_o$ is closed and densely defined.

We now proceed with the explicit characterization of $\dW_o^*$ given in item (2). Let $x\in \dom{\dW_o}$ and $\by\in L^{2+}_Y$. We have $\Cfrak x=\dW_o x\in L^{2+}_Y$. Hence
\begin{align*}
\Ipdp{\dW_o x}{\by}_{L^{2+}_Y}
&\!\! = \! \Ipdp{\Cfrak x}{\by}_{L^{2+}_Y}
\!\!=\! \lim_{t\to\infty} \Ipdp{\pi_{[0,t]}\Cfrak x}{ \by}_{L^{2+}_{Y}}
\!\!=\! \lim_{t\to\infty} \Ipdp{\Cfrak^t x}{ \pi_{[0,t]}\by}_{L^{2}([0,t],Y)}\\
&=\lim_{t\to\infty} \Ipdp{ x}{(\Cfrak^t)^* \pi_{[0,t]} \by}_X.
\end{align*}
Then $\by\in \dom{\dW_o^*}$ if and only if there exists an $x_0\in X$, such that for all $x\in \dom{\dW_o}$, we have
\[
\Ipdp{x}{x_0}_X=\Ipdp{\dW_o x}{\by}_{L^{2+}_Y}=\lim_{t\to\infty} \Ipdp{ x}{(\Cfrak^t)^* \pi_{[0,t]} \by}_X.
\]
This proves item (2), and we next prove item (3).

In case $\by\in L^{2+}_{r,Y}$, say $\supp{\by}\subset[0,T]$, then $(\Cfrak^t)^*\pi_{[0,t]}\by$ is independent of $t$ for $t>T$ and thus $x_0:=\lim_{t\to\infty} (\Cfrak^t)^*\pi_{[0,t]}\by=(\Cfrak^T)^*\by$ exists and satisfies \eqref{eq:Wo*} by \eqref{eq:Wo*2}. Hence $L^{2+}_{r,Y}\subset \dom{\dW_o^*}$ and for $\by\in L^{2+}_{r,Y}$, it by \eqref{eq:Wo*2} holds that 
$$
	\dW_o^*\by=\lim_{t\to\infty}\Bfrak^d \ya \pi_{[0,t]} \by=
	\Bfrak^d \ya \by,
$$
and then clearly 
$$
	\bW_o^*L^{2+}_{r,Y}=\Bfrak^d\ya L^{2+}_{r,Y}=
	\Bfrak^d L^{2-}_{\ell,Y}=\range{\Bfrak^d}\,;
$$
this proves all of item (3).

Finally, we prove item (4). Fix $s,t>0$, $x\in X$ and $\by\in \dom{\dW_o^*}\subset L^{2+}_Y$. Then we have
\begin{align*}
&\Ipdp{x}{(\Cfrak^{t+s})^* \pi_{[0,t+s]} \tau^{-s} \by}_X
=\Ipdp{x}{(\Cfrak^{t+s})^*  \tau^{-s} \pi_{[0,t]} \by}_X\\
&\qquad\qquad= \Ipdp{\tau_+^s\Cfrak^{t+s}x}{ \pi_{[0,t]} \by}_{L^{2}([0,t],Y)}
= \Ipdp{\tau^s_+\pi_{[0,s+t]}\Cfrak x}{ \pi_{[0,t]} \by}_{L^{2}([0,t],Y)}\\
&\qquad\qquad= \Ipdp{\pi_{[0,t]}\tau_+^s\Cfrak x}{\pi_{[0,t]} \by}_{L^{2}([0,t],Y)}
= \Ipdp{\pi_{[0,t]}\Cfrak \Afrak^s x}{\pi_{[0,t]} \by}_{L^{2}([0,t],Y)}\\
&\qquad\qquad= \Ipdp{\Cfrak^t \Afrak^s x}{\pi_{[0,t]} \by} _{L^{2}([0,t],Y)}
= \Ipdp{\Afrak^s x}{ (\Cfrak^t)^* \pi_{[0,t]} \by} _X.
\end{align*}
Moreover, for $x\in\dom{\dW_o}$, we have $\Afrak^sx\in\dom{\dW_o}$, since $\Cfrak \Afrak^sx=\tau_+^s\Cfrak x\in L^{2+}_Y$. Using all of this, we find for $x\in\dom{\dW_o}$ and $x_o\in X$ satisfying \eqref{eq:Wo*} that
\begin{align*}
\Ipdp{x}{(\Afrak^s)^* x_o} _X
&= \Ipdp{\Afrak^s x}{ x_o} _X
= \lim_{t\to\infty} \Ipdp{\Afrak^s x}{ (\Cfrak^t)^* \pi_{[0,t]} \by} _X\\
&= \lim_{t\to\infty} \Ipdp{x}{(\Cfrak^{t+s})^*  \pi_{[0,t+s]} \tau^{-s} \by}_X
= \lim_{r\to\infty} \Ipdp{x}{(\Cfrak^{r})^*  \pi_{[0,r]} \tau^{-s} \by}_X.
\end{align*}
Since the limit exists for every $x\in\dom{\dW_o}$, it follows that $\tau^{-s} \by\in\dom{\bW_o^*}$
and $\bW_o^*\tau^{-s} \by=(\Afrak^s)^* x_o= (\Afrak^s)^* \bW_o^* \by$, which proves item (4).
\end{proof}

The $L^2$-input map is defined similarly, via the causal dual system. We first define the \emph{adjoint $L^2$-input map} $\dW_c^\bigstar$, using $\bigstar$ to indicate that $\dW_c^\bigstar$ is defined directly and not as the adjoint of an operator $\dW_c$:
\begin{equation}\label{eq:WcStarDef}
\begin{aligned}
\dW_c^\bigstar := \ya\Cfrak^d\big|_{\dom{\dW_c^\bigstar}}:X\supset \dom{\dW_c^\bigstar}\to L^{2-}_U,\\
\text{with }	\dom{\dW_c^\bigstar} := \set{x\in X\mid \Cfrak^d x\in L^{2+}_U}.
\end{aligned}
\end{equation}
Defining $\dW_o^d$ and $\dW_c^{d\bigstar}$ similarly as in \eqref{Wo} and \eqref{eq:WcStarDef}, respectively, for the causal dual system $\Sigma^d$, one obtains
\begin{equation}\label{DualObsCon}
\bW_o^d=\ya \dW_c^\bigstar\quad\mbox{and}\quad \bW_c^{d\bigstar}=\ya \dW_o,
\end{equation}
and in particular, $\Sigma$ is minimal if and only if $\dW_o$ and $\dW_c^\bigstar$ are both injective.

By duality, from Proposition \ref{P:Wo}, we obtain the following result:

\begin{proposition}\label{P:Wc}
Let $\dW_c^\bigstar$ be the adjoint $L^2$-input map of a well-posed system $\Sigma=\sbm{\Afrak&\Bfrak\\\Cfrak&\Dfrak}$. Then $\dW_c^\bigstar$ is closed.  Additionally assume that $\dW_c^\bigstar$ is densely defined. In this case:
\begin{enumerate}
\item[(1)] The operator $\dW_c^\bigstar$ has a closed and densely defined adjoint, denoted by $\dW_c$, such that $\dW_c^\bigstar=\dW_c^*$.

\item[(2)] A function $\bu\in L^{2-}_U$ lies in $\dom{\dW_c}$ if and only if there exists an $x_c\in X$ such that 
\begin{equation}\label{eq:Wc}
\lim_{t\to\infty}\Ipdp{x}{\Bfrak\pi_{[-t,0]} \bu}_X
=\Ipdp{x}{x_c}_X,\quad x\in\dom{\dW_c^\bigstar}.
\end{equation}
When $\bu\in\dom{\dW_c}$, we have $\dW_c \bu=x_c$, where $x_c$ is given by \eqref{eq:Wc}.

\item[(3)]  It holds that $L^{2-}_{\ell,U}\subset \dom{\dW_c}$, that $\dW_c\big|_{L^{2-}_{\ell,U}}=\Bfrak$, and that $\dW_c L^{2-}_{\ell,U}= \range{\Bfrak}=\textup{Rea}\,(\Sigma)$.

\item[(4)] For all $s>0$ we have $\tau^{s} \dom{\dW_c} \subset \dom{\dW_c}$ and $\dW_c \tau^{s}\bu=\Afrak^{s} \dW_c \bu$ for all $\bu\in \dom{\dW_c}$.
\end{enumerate}
\end{proposition}

Again, it holds that
$$
	\Ipdp{x}{\Bfrak\pi_{[-t,0]} \bu}_X
	=\Ipdp{x}{((\Cfrak^d)^t)^*\pi_{[0,t]}\ya \bu}_X,
	\quad x\in X,\, \bu\in L^{2-}_U,\, t\geq0.
$$

We have the following easy corollary:

\begin{corollary}\label{cor:L2inputmaptraj}
Assume that the adjoint $L^2$-input map $\dW_c^\bigstar$ of a well-posed system $\Sigma=\sbm{\Afrak&\Bfrak\\\Cfrak&\Dfrak}$ is densely defined. For every system trajectory $(\bu,\bx,\by)$ of $\Sigma$ on $\R$, we have $\pi_-\bu\in\dom{\dW_c}$ and $\bx(0)=\dW_c\pi_-\bu$.
\end{corollary}

\begin{proof}
Let $(\bu,\bx,\by)$ be a trajectory of $\Sigma$ on $\R$. By Definition \ref{def:WPtraj}, we then have $\pi_-\bu\in L^{2-}_{\ell,U}\subset \dom{\Bfrak}\subset\dom{\dW_c}$. By item (3) of Proposition \ref{P:Wc} and \eqref{eq:trajonR}, $\bx(0)=\Bfrak\pi_-\bu=\dW_c\pi_-\bu$.
\end{proof}

In the remainder of this section we shall assume that $\widehat\Dfrak\big|_{\dom{\widehat\Dfrak}\bigcap\cplus}$ has a unique analytic extension to a function in
$H^\infty(\cplus;\cB(U,Y))$, also denoted by $\widehat\Dfrak$. With our convention to identify analytic functions that coincide on some set with an interior cluster
point, we simply write $\widehat\Dfrak\in H^\infty(\cplus;\cB(U,Y))$. In that case, $\widehat\Dfrak$ defines a bounded pointwise multiplication operator
\begin{equation}\label{MultOp}
M_{\widehat\Dfrak}:L^2(i\R;U) \to L^2(i\R;Y),\quad (M_{\widehat\Dfrak} f)(\lambda)= \widehat\Dfrak (\lambda) f(\lambda), \quad \lambda\in i\R,
\end{equation}
with operator norm $\| M_{\widehat\Dfrak} \|$ 
equal to the supremum norm $\|\widehat\Dfrak\|_\infty$ of
$\widehat\Dfrak$ over $\cplus$. 
Further, let $\cL:L^2(\R;K)\to L^2(i\R;K)$ denote the unitary \emph{bilateral Laplace transform}
$$
	(\cL \bu)(\lambda)=\int_{-\infty}^\infty
		e^{-\lambda t}\,\bu(t)\ud t,\quad
		\lambda\in i\R,
$$
which in particular maps $L^{2+}_K$ unitarily onto $H^{2+}_K:=H^2(\cplus;K)$. We then define the \emph{$L^2$-transfer map} $L_\Sigma$ by
\begin{equation}\label{L2inout}
L_\Sigma:=\cL^*M_{\widehat\Dfrak}\cL\in\cB(L^2_U,L^2_Y).
\end{equation}
We now derive various properties of this operator.

\begin{theorem}\label{thm:hankel}
Let $\Sigma$ be a well-posed linear system with transfer function $\widehat\Dfrak\in H^\infty(\cplus;\cB(U,Y))$. The following statements are true:
\begin{enumerate}
\item[(1)] The operator $L_\Sigma$ in \eqref{L2inout} is the unique continuous linear extension to an operator in $\cB(L^2_U,L^2_Y)$ of the restriction of $\Dfrak$ to $L^2_{\ell,U}$. Moreover, we have $\|L_\Sigma\|=\|\widehat\Dfrak\|_\infty$ and $L_\Sigma$ is causal, i.e., $\pi_- L_\Sigma \pi_+=0$, and time-invariant, i.e., $\tau^t L_\Sigma =L_\Sigma \tau^t$ for all $t\in\R$.

\item[(2)] It holds that $\range{\Bfrak}\subset \dom{\dW_o}$. The restriction to $L^{2-}_{\ell,U}$ of the Hankel operator $\pi_+\Dfrak\pi_-$ has a unique extension to an operator in $\cB(L^{2-}_U,L^{2+}_Y)$, which equals
\begin{equation}\label{HankelBounded}
\Hfrak_\Sigma:=\pi_+L_\Sigma\big|_{L^{2-}_U}\ \ \mbox{and satisfies}\ \
\|\Hfrak_\Sigma\|\leq\|\widehat\Dfrak\|_\infty,\quad \Hfrak_\Sigma|_{L^{2-}_{\ell,U}}=\dW_o \Bfrak.
\end{equation}

\item[(3)] For the causal dual system $\Sigma^d$ we have $\widehat\Dfrak^d\in H^\infty(\cplus;\cB(Y,U))$, the unique extension $L_{\Sigma^d}$ in $\cB(L^{2}_Y,L^{2}_U)$ of $\Dfrak^d$ restricted to $L^2_{\ell,Y}$ satisfies
\begin{equation}\label{eq:Lsig-d}
	L_{\Sigma^d}=\ya L_\Sigma^*\ya,
\end{equation}
and the $L^2$-analogue of the Hankel operator of the causal dual is $\Hfrak_{\Sigma^d}:=\pi_+L_{\Sigma^d}\big|_{L^{2-}_Y}=\ya\Hfrak_\Sigma^*\ya$. Moreover, we have $\range{\Bfrak^d}\subset \dom{\dW_c^\bigstar}$ and
\begin{equation}\label{Hfact1}
\ya\Hfrak_{\Sigma^d} |_{L^{2-}_{\ell,Y}}=
\Hfrak_\Sigma^*\ya |_{L^{2-}_{\ell,Y}}=\dW_c^\bigstar \Bfrak^d.
\end{equation}

\item[(4)] Furthermore, if $\dom{\dW_c ^\bigstar}$ is dense in $X$, then $\range{\dW_c } \subset\dom{\dW_o }$ and
\begin{equation}\label{Hfact2}
	\Hfrak_\Sigma\big|_{\dom{\dW_c }}=\dW_o \dW_c.
\end{equation}
If $\dom{\dW_o}$ is dense in $X$, then $\range{\dW_o^* } \subset\dom{\dW_c^\bigstar}$ and
\begin{equation}\label{Hfact3}
\Hfrak_\Sigma^*\big|_{\dom{\dW_o^* }}=\dW_c^\bigstar \dW_o^*.
\end{equation}
\end{enumerate}
\end{theorem}

\begin{proof}
Since $\widehat\Dfrak\in H^\infty(\cplus;\cB(U,Y))$, the operator $L_\Sigma$ maps $L^{2+}_U$ into $L^{2+}_Y$; hence $L_\Sigma$ is causal. Moreover, for every $\omega\in\R$, since $M_{\widehat\Dfrak}$ intertwines $M_{e_\omega 1_U}$ and $M_{e_\omega 1_Y}$, where $(e_\omega 1_K)(z)=e^{\omega z} 1_K $, we get that $L_\Sigma$ commutes with $\tau^t$ (suppressing the spaces $U$ and $Y$ in the notation); hence $L_\Sigma$ is time invariant. Now let $\bu\in L^2_U$ have $\supp \bu\subset [N,\infty)$ for some $N\in\R$. Then $\bu\in\dom{L_\Sigma}\bigcap\dom{\Dfrak}$ and $\tau^N\bu\in L^{2+}_U\subset L^{2+}_{\omega,U}$ for $\omega>\min\,\{0,\omega_\Afrak\}$. By \cite[Corollary 4.6.10(iii)]{StafBook} we have $M_{\widehat\Dfrak} \cL (\tau^N \bu)=\cL (\Dfrak\pi_+ \tau^N \bu)=\cL (\Dfrak \tau^N \bu)$. Hence
$$
	\tau^N L_\Sigma \bu=
	L_\Sigma \tau^N \bu=
	\cL^*M_{\widehat\Dfrak}\cL(\tau^N\bu)=
	\cL^* \cL (\Dfrak \tau^N \bu)=
	\Dfrak \tau^N \bu =
	\tau^N \Dfrak \bu.
$$
It follows that $L_\Sigma \bu= \Dfrak \bu$ for every $\bu\in L^2_{\ell,U}$. Since the latter subspace is dense in $L^2_U$, the only extension to a bounded linear operator on $L^2_U$ of the restriction of $\Dfrak$ to $L^2_{\ell,U}$ is $L_\Sigma$. Since $\mathcal{L}$ is unitary, we have $\|L_\Sigma\|=\|M_{\widehat\Dfrak}\|=\|\widehat\Dfrak\|_\infty$.
This proves item (1).

By \eqref{HankelBounded} and item (1), the operator $\Hfrak_\Sigma$ coincides with $\pi_+\Dfrak\pi_-$ on $L^{2-}_{\ell,U}$, and hence $\Hfrak_\Sigma$ is the unique extension to an operator in $\cB(L^{2-}_U,L^{2+}_Y)$ of $\pi_+\Dfrak\pi_-$ restricted to $L^{2-}_{\ell,U}$. Observing that $\pi_+$ is a contraction on $L^2_K$, we obtain that $\|\Hfrak_\Sigma\|\leq\|L_\Sigma\|=\|\widehat\Dfrak\|_\infty$. To see that the factorization of $\Hfrak_\Sigma|_{L^{2-}_{\ell,U}}$ in \eqref{HankelBounded} holds, let $\bu\in L^{2-}_{\ell,U}$ and note that Definition \ref{def:WPsys}.4(c) gives that $\Cfrak\Bfrak \bu=\pi_+\Dfrak\pi_- \bu= \Hfrak_\Sigma \bu$, which is in $L^{2+}_Y$ by the boundedness of $\Hfrak_\Sigma$. Hence $\Bfrak \bu \in\dom{\dW_o}$ and $\Hfrak_\Sigma \bu= \dW_o \Bfrak \bu$. This establishes item (2).

That $\widehat\Dfrak^d\in H^\infty(\cplus;\cB(Y,U))$ follows directly from $\widehat\Dfrak ^d(\lambda)=\widehat\Dfrak(\overline\lambda)^*$ in Theorem \ref{T:causdual}. By item (2) of the present theorem, which has already been proved, the restriction of $\Dfrak^d=\ya\Dfrak^\circledast\ya$ to $L_{\ell,Y}^2$ has a unique extension to $L_{\Sigma^d}\in\cB(L^{2}_Y,L^{2}_U)$. Moreover, $L_{\Sigma^d}=\ya L_\Sigma^*\ya$, because for all $\bu\in L^{2}_{\ell,U}$, $\by\in L^{2}_{r,Y}$ and some $\omega>\max\set{0,\omega_\Afrak}$,
\begin{align*}
	\Ipdp{L_\Sigma^* \by}{\bu}_{L^{2}_U} &=
	\Ipdp{ \by}{\Dfrak \bu}_{L^{2}_Y} =
	\Ipdp{ \by}{\widetilde\Dfrak \bu}_{L^{2}_{-\omega,Y},L^{2}_{\omega,Y}}\\
	&=
	\Ipdp{\widetilde\Dfrak^*\by}{ \bu}_{L^{2}_{-\omega,U},L^{2}_{\omega,U}}
	=	\Ipdp{\Dfrak^\circledast \by}{\bu}_{L^{2}_U},
\end{align*}
so that $L_\Sigma^*$ and $\Dfrak^\circledast$ coincide on $L^{2}_{r,Y}$ by the density of $L^2_{\ell,U}$ in $L^2_U$; then also $\Dfrak^d=\ya\Dfrak^\circledast\ya$ and $\ya L_\Sigma^*\ya$ coincide on $L^2_{\ell,Y}$, so that $L_{\Sigma^d}=\ya L_\Sigma^*\ya$. Letting $\iota_\pm:L^{2\pm}_K\to L^2_K$ denote the injection, we can write $\Hfrak_{\Sigma}=\pi_+L_{\Sigma}\iota_-$, and then $\Hfrak_{\Sigma}^*=\pi_-L_{\Sigma}^*\iota_+$, so that
\begin{equation}\label{eq:L2HandDualFact}
	\Hfrak_{\Sigma^d}=\pi_+L_{\Sigma^d}\big|_{L^{2-}_Y}
	=\ya\pi_-L_\Sigma^*\iota_+\ya = \ya\Hfrak_\Sigma^*\ya.
\end{equation}
Now \eqref{Hfact1} follows from \eqref{eq:L2HandDualFact} and \eqref{HankelBounded}, using the first identity in \eqref{DualObsCon}, and hence item (3) is true.

Now assume that $\dom{\dW_c ^\bigstar}$ is dense in $X$, hence $\dW_c$, the adjoint of $\dW_c ^\bigstar$, is closed and densely defined. From item (3) in Proposition \ref{P:Wc} and \eqref{HankelBounded}, it follows that $\Hfrak_\Sigma$ and $\dW_o \dW_c$ coincide on $L^{2-}_{\ell,U}$. We now show that $\range{ \dW_c}\subset \dom{\dW_o}$ and that $\Hfrak_\Sigma$ and $\dW_o \dW_c$ also coincide on $\dom{\dW_c}$. Let $\bu\in \dom{\dW_c}\subset L^{2-}_U$ and $x_c= \dW_c \bu \in \range{\dW_c}$. Choose $T>0$ and $\by\in L^{2}([0,T];Y)$ arbitrarily. Then Lemma \ref{L:dualadj} and item (3) yield
$$
	(\Cfrak^T)^*\by
	=(\Bfrak^d)^T(\Lambda^t_K)^*\by
	\in \dom{\dW_c^\bigstar},
$$
while item (2) of Proposition \ref{P:Wc} and the boundedness of $\Cfrak^T$ give 
\begin{align*}
\Ipd{\by}{\Cfrak^T x_c}_{L^{2+}_Y}
&
=\lim_{t\to \infty} \Ipd{(\Cfrak^T)^* \by}{ \Bfrak \pi_{[-t,0]}\bu}_X
=\lim_{t\to \infty} \Ipd{\by}{\pi_{[0,T]}\Cfrak \Bfrak \pi_{[-t,0]}\bu}_{L^{2}([0,T];Y)}\\
&=\lim_{t\to \infty} \Ipd{\by}{\pi_{[0,T]} \Hfrak_\Sigma \pi_{[-t,0]}\bu}_{L^{2}([0,T];Y)}
=\Ipd{\by}{\pi_{[0,T]} \Hfrak_\Sigma \bu}_{L^{2}([0,T];Y)},
\end{align*}
using the boundedness of $\Hfrak_\Sigma$ in the last identity. Since the above computation holds for all $\by$ and all $T$, we have
$\pi_{[0,T]}\Cfrak x_c=\Cfrak^T x_c=\pi_{[0,T]} \Hfrak_\Sigma \bu$ for all $T>0$. This shows that $\Cfrak x_c = \Hfrak_\Sigma \bu \in L^{2+}_Y$. In particular, we have $x_c\in\dom{\dW_o}$ and $\dW_o \dW_c \bu= \dW_o x_c= \Cfrak x_c=\Hfrak_\Sigma \bu$. Equality \eqref{Hfact3} is obtained by applying \eqref{Hfact2} to $\Sigma^d$, using that $\Hfrak_\Sigma^*=\ya \Hfrak_\Sigma^d \ya$, as proved above, and the identities in \eqref{DualObsCon}.
\end{proof}

\begin{corollary}\label{C:MinExp}
Let $\Sigma$ be a well-posed system  with $\widehat\Dfrak\in H^\infty(\cplus;\cB(U,Y))$. If $\Sigma$ is controllable, then $\dW_o$ is densely defined; if $\Sigma$ is observable, then $\dW_c^\bigstar$ is densely defined.
\end{corollary}

\begin{proof}
By Theorem \ref{thm:hankel}, the finite-time reachable subspace $\textup{Rea}\,(\Sigma)=\range{\Bfrak}$ is contained in $\dom{\dW_o}$ and the finite time observable subspace $\textup{Obs}\,(\Sigma)=\range{\Bfrak^d}$ is contained in $\dom{\dW_c^\bigstar}$. Thus the claim follows directly from Definition \ref{def:WPmin-approx}.
\end{proof}

We now present two cases where the $L^2$-input and $L^2$-output map are both bounded.

\begin{lemma}\label{lem:PasvContr}
For a well-posed system $\Sigma$, the following hold:
\begin{enumerate}

\item[(1)] If $\Sigma$ is exponentially stable, then $\dW_c\in\cB(L^{2-}_U,X)$ and $\dW_o\in\cB(X,L^{2+}_Y)$.

\item[(2)] If $\Sigma$ is passive, then $\dW_c$ and $\dW_o$ are everywhere-defined contractions.

\end{enumerate}
\end{lemma}

\begin{proof}
Concerning item (1), if $\Sigma$ is exponentially stable, then $\omega_\Afrak<0$ so that we can choose $\omega=0$ in order to obtain from \eqref{eq:FrakTildeDef} that $\widetilde\Cfrak\in\cB(X,L^{2+}_Y)$ and $\widetilde\Bfrak\in\cB(L^{2-}_U,X)$. Then $\dW_o =\widetilde\Cfrak$ and $\dW_c ^\bigstar=\widetilde\Bfrak^*$ are bounded, too, and we have $\dW_c=(\dW_c^\bigstar)^*\in\cB(L^{2-}_U,X)$.

For item (2), note that a passive system satisfies \eqref{eq:storfndef} with $S(x)=\|x\|^2_X$ by definition. For trajectories $(\bu,\bx,\by)$ on $\rplus$ with $\bu=0$, we in particular obtain $\int_0^t\|\by(s)\|^2\ud s\leq \|\bx(0)\|^2$, and letting $t\to\infty$, we get $\by\in L^{2+}_Y$. Moreover, by \eqref{eq:trajonRplus} and the definition \eqref{Wo} of $\dW_o$ we have $\|\by\|_{L^{2+}_Y}^2=\|\dW_o\bx(0)\|_{L^{2+}_Y}^2\leq\|\bx(0)\|_X^2$. This proves that $\dW_o$ is an everywhere-defined contraction, and applying the same argument to the passive dual $\Sigma^d$, using \eqref{DualObsCon}, gives that $\dW_c^\bigstar$ is a contraction, hence $\dW_c$ is a well-defined contraction, too.
\end{proof}

The following definition presents the analogues of exact $\ell^2$-controllability and exact $\ell^2$-observability from \cite{BGtH18a} in the context of well-posed systems.

\begin{definition}\label{def:L2min}
The well-posed system $\Sigma$ is \emph{(exactly) $L^2$-controllable} if $\dW_c ^\bigstar$ is densely defined and $\range{\dW_c }=X$. The system $\Sigma$ is \emph{(exactly) $L^2$-observable} if $\dW_o $ is densely defined and $\range{\dW_o ^*}=X$. The system $\Sigma$ is \emph{(exactly) $L^2$-minimal} if it is both $L^2$-controllable and $L^2$-observable.
\end{definition}

By \eqref{DualObsCon}, $\Sigma$ is $L^2$-controllable ($L^2$-observable) if and only if $\Sigma^d$ is $L^2$-observable ($L^2$-controllable). Some differences between $\ell^2$-controllability/observability and approximate controllability/observability for discrete-time systems are described in \cite[Proposition 2.7]{BGtH18a}; here we prove analogous results in the present context, and we also provide new information on these relationships.

\begin{corollary}\label{C:L2min}
For each well-posed system $\Sigma$ as in Definition \ref{def:WPsys}, $L^2$-controllability ($L^2$-observability) implies (approximate) controllability (observability). In particular, $L^2$-minimality of $\Sigma$ implies minimality of $\Sigma$. When we additionally assume that $\widehat\Dfrak\in H^\infty(\cplus;\cB(U,Y))$, the following statements are true:

\begin{enumerate}
  \item[(1)] If $\Sigma$ is $L^2$-controllable then $\dW_o$ is bounded.

  \item[(2)] If $\Sigma$ is $L^2$-observable then $\dW_c$ is bounded.

  \item[(3)] If $\Sigma$ is $L^2$-minimal then $\dW_c^* $ and $\dW_o$ are both bounded and bounded below.
\end{enumerate}
\end{corollary}

Hence, the assumptions on denseness of the domains of $\dW_c ^\bigstar$ and $\dW_o$ impose no restriction in the study of the bounded real lemma, since in the standard version (Theorem \ref{thm:stdlemma}) we assume minimality (or even $L^2$-minimality in Theorem \ref{thm:stdlemmaL2reg}) and in the strict version (Theorem \ref{thm:stdlemmastrict}) we assume exponential stability; see Lemma \ref{lem:PasvContr}.

\begin{proof}[Proof of Corollary \ref{C:L2min}]
Assume that $\Sigma$ is $L^2$-observable; then by Definition \ref{def:L2min}, $\dom{\dW_o}$ is dense in $X$ and $\range{\dW_o^*}=X$. Since $\dW_o$ is closed, the comment after \eqref{Wo} gives
that $\Sigma$ is (approximately) observable. If instead $\Sigma$ is $L^2$-controllable, then $\Sigma^d$ is $L^2$-observable, and further $\Sigma^d$ is observable by what we just proved; hence $\Sigma$ is controllable by Corollary  \ref{cor:ApproxContrObsDual}.

Now assume that $\Sigma$ is $L^2$-controllable and that $\widehat\Dfrak\in H^\infty(\cplus;\cB(U,Y))$. Then $\dom{\dW_c^\bigstar}$ is dense by definition, and according to Theorem \ref{thm:hankel}, we have $X=\range{\dW_c}\subset \dom{\dW_o}$, so that $\dW_o$ is bounded by the closed graph theorem. This completes the proof of item (1), and the proof of item (2) is easy using duality.

In conclusion we prove item (3). By assumption the ranges of $\dW_c$ and $\dW_o^*$ are equal to $X$. From items (1) and (2) we obtain that $\dW_c$
and $\dW_o^*$ are bounded. The boundedness of $\dW_c$ and $\dW_o^*$ together with $\range{\dW_c}=X=\range{\dW_o^*}$ yields that $\dW_c$
and $\dW_o^*$ have bounded right inverses, or, equivalently, $\dW_c^*$ and $\dW_o$ have bounded left inverses, and hence the latter are bounded below.
\end{proof}

\section{System nodes and well-posed linear systems}  \label{sec:system-node}

The well-posed systems considered in the present paper can alternatively be formulated in a differential representation via a so-called system node $\sbm{  A \& B \\ C \& D }$. In this section we review some of the details of system nodes and describe some related topics relevant for the paper, including a reformulation of the KYP-inequality in terms of system nodes. See Chapters 3 and 4 of \cite{StafBook} for full details and many more results on system nodes.

\subsection{Construction of the system node}\label{SubS:SysNodeConstruct}
Let $\Sigma=\sbm{\Afrak & \Bfrak\\ \Cfrak & \Dfrak}$ be a well-posed linear system as in Definitions \ref{def:WPsys} and \ref{def:WPtraj}. Let $A$ on $X$ be the {\em infinitesimal generator} of the $C_0$-semigroup $\Afrak^t$, that is,
$$
\dom{A}=\left\{x\in X \biggmid \lim_{h\to 0} \frac{1}{h}(\Afrak^tx -x) \text{ exists} \right\},\quad
Ax=\lim_{h\to 0} \frac{1}{h}(\Afrak^t x-x).
$$
Now fix the
\emph{rigging parameter} $\beta\in\rho(A)$ arbitrarily and define the \emph{interpolation space} $X_1:=\dom{A}$ with the Hilbert space norm
$\|x\|_1:=\|(\beta-A)x\|_X$; then $\alpha -A$ is an isomorphism from $X_1$ to $X$ for all $\alpha\in\rho(A)$. Next complete $X$ in the norm
$\|x\|_{-1}:=\|(\beta-A)^{-1}x\|_X$ to get the \emph{extrapolation space} $X_{-1}$. Then we have the chain of inclusions
\begin{equation}  \label{Xsub-1}
X_1\subset X\subset X_{-1}
\end{equation}
 with dense and continuous embeddings, and the spaces $X_1$, $X$ and $X_{-1}$ form a {\em Gelfand triple}. Moreover, the generator $A$ extends uniquely to a bounded operator $A_{-1}$ in
$\cB(X, X_{-1})$ which in turn is the infinitesimal generator of a $C_0$-semigroup $\Afrak_{-1}^t$ on $X_{-1}$ which extends $\Afrak^t$. The resolvent set $\rho(A_{-1})$ equals $\rho(A)$; see \cite[\S3.6]{StafBook} for further details.

By Theorems 4.2.1 and 4.4.2 in \cite{StafBook} there exist bounded operators $B\in\cB(U,X_{-1})$, the {\em control operator}, and $C\in \cB(X_1,Y)$, the {\em observation operator}, that are uniquely determined by the formulas
\begin{equation}\label{BfrakCfrak}
\Bfrak \bu=\int_{-\infty}^0 \Afrak_{-1}^{-s} B \bu(s)  \ud s, \quad \bu\in L^{2-}_{\ell,U},\qquad (\Cfrak x)(t)=C\Afrak^t x,\quad x\in X_1.
\end{equation}
Note that while $B$ maps into $X_{-1}$ and $\Afrak_{-1}^t$ acts on $X_{-1}$, the result after integration in the first formula still ends up in $X$.

With $A$ and $B$ defined as above we can form a closed operator $A \& B \colon \sbm{X \\ U} \supset \dom{\AB} \to X$ by
$$
  \dom{\AB}= \left\{ \begin{bmatrix} x \\ u \end{bmatrix} \biggmid A_{-1} x + B u \in X \right \}
  \quad\text{and}\quad
 \AB\begin{bmatrix} x \\ u \end{bmatrix} = A_{-1} x + B u.
$$
Choose a fixed $\alpha \in \C_{\omega_\Afrak}$.  For $\sbm{ x \\ u } \in \dom{\AB}$, we then have
\begin{align*}
 x - (\alpha - A_{-1})^{-1} B u =& (\alpha - A_{-1})^{-1}
 \big( \alpha x - (A_{-1} x +  Bu)\big) \\
 & \in
 (\alpha - A)^{-1} X = X_1 = \dom{C}.
\end{align*}
From $\Dfrak$, we can compute the 
transfer function $\widehat\Dfrak\in H^\infty(\C_\omega;\cB(U,Y))$, $\omega>\omega_\Afrak$, of $\Sigma$ 
via Proposition \ref{prop:Transfer}. Since $\alpha\in\C_{\omega_\Afrak}$, we can evaluate $\widehat\Dfrak(\alpha)$, and then define 
\begin{equation}   \label{C&D}
  C \& D \colon \begin{bmatrix} x \\ u \end{bmatrix} \mapsto C \big(x - (\alpha - A_{-1})^{-1} B u\big)+ \widehat \Dfrak(\alpha)u.
\end{equation}
Note that if $x \in X_1$, then $\sbm{ x \\ 0} \in
\dom{\AB}$ and $\CD \sbm{ x \\ 0 } = C x$.  In general there is no sensible way to separate out an independent feedthrough operator $D \in \cB(U,Y)$
except in some special cases, e.g., if at least one of $B \colon X \to U$ and  $C \colon X \to Y$ is bounded
(see Theorems 4.5.2 and 4.5.10  in \cite{StafBook}), or if $\Sigma$ is {\em regular} 
(see Chapter 5 in \cite{StafBook}). Rather we think of $C \& D$ as an extension of the operator $C$ defined on $X_1 \cong \sbm{ X_1 \\ 0} \subset \sbm{ X \\ U}$ to the operator $C \& D$ defined
on $\dom{\AB}\supset \sbm{ X_1 \\ 0 }$
and mapping into $X$.
After the above steps, we can introduce the {\em system node} $\sbm{  A \& B \\ C \& D } \colon \sbm{ X \\ U}\supset
\dom{\sbm{  A \& B \\ C \& D }} \to \sbm{ X \\ Y}$ with
$$
\dom{\sbm{  A \& B \\ C \& D }} = \dom{\AB} =  \dom{\CD}
$$
and action
$$
\begin{bmatrix} \AB \\ \CD \end{bmatrix}  \colon \begin{bmatrix} x \\ u \end{bmatrix} \mapsto \begin{bmatrix} \AB \sbm{x \\ u}
\\ \CD \sbm{ x \\ u } \end{bmatrix}.
$$

We next recall Definition 4.7.2 in \cite{StafBook}. 

\begin{definition}\label{def:sysnode}
Suppose that $\bS:=\SmallSysNode$ is an operator mapping  a dense subspace $\dom{\bS}$ of $\sbm{ X \\ U}$ into $\sbm{ X \\ Y}$.
 We shall say that $\bS$ is a {\em system node} if it has the following properties:
\begin{itemize}
\item[(1)] $\bS$ is closed as an operator from $\sbm{ X \\ U}$ into $\sbm{ X \\ Y}$.

\item[(2)]  The operator $A \colon X \supset \dom{A} \to X$ defined by $A x = A \& B \sbm{ x \\ 0}$ on
$\dom{A} = \{ x \in X \bigmid \sbm{ x \\ 0 } \in \dom{\bS} \}$ has domain dense in $X$, and $A$ as an unbounded operator on $X$
generates a $C_0$-semigroup on $X$.

\item[(3)]  The operator $A \& B$ (with $\dom{\AB} = \dom{\bS}$) can be extended to an operator
$$
\begin{bmatrix} A_{-1} & B \end{bmatrix} \in \cB(\sbm{ X \\ U }, X_{-1})
$$
 (where $X_{-1}$ is the extrapolation space introduced in \eqref{Xsub-1}).

\item[(4)] $\dom{\bS} = \big\{ \sbm{ x \\ u } \in \sbm { X \\ U} \bigmid A_{-1} x + B u \in X\}$.
\end{itemize}
\end{definition}

Given a system node $\bS = \sbm{ A \& B \\ C \& D}$ we may define its transfer function $\widehat \Dfrak_\bS(\lambda)$ by
\begin{equation} \label{node-transfunc}
   \widehat \Dfrak_\bS(\lambda)u =  C \& D \bbm{ (\lambda - A_{-1})^{-1} B \\ 1_U} u,\quad \lambda\in\rho(A).
\end{equation}
If $\widehat\Dfrak$ is constructed as in Proposition \ref{prop:Transfer}, then $\widehat\Dfrak_\bS$ is an extension of $\widehat\Dfrak$ from $\C_{\omega_\Afrak}$ to all of $\rho(A)$; see \cite[Lemma 4.7.5(iii)]{StafBook}. 

We end this subsection with a result which says that a system node works as the connecting operator of a well-posed system.

\begin{lemma} \label{L:traj-node}
{\rm  (See \cite[Theorem 4.6.11(i)]{StafBook}.)}  Suppose that $\Sigma = \sbm{ \Afrak & \Bfrak \\ \Cfrak & \Dfrak}$ is a well-posed system with associated
system node $\bS = \sbm{ A \& B \\ C \& D}$.  Let $(\bu, \bx, \by)$ be a system trajectory over ${\mathbb R}^+$ with state initial condition $\bx(0) = x_0$
and with $\bu$ continuous with distributional derivative $\dot \bu$ in $L^{2+}_{loc, U}$ and such that $\sbm{ x_0 \\ \bu(0) } \in \dom{\bS}$.
Then $\bx$ is continuously differentiable with values in $X$, $\sbm{ \bx(t) \\ \bu(t) } \in \dom{\bS}$
for all $t>0$,  $\by$ is continuous with distributional derivative $\dot \by$ in $L^{2+}_{loc,Y}$, and
\begin{equation} \label{node-sys-eq}
   \begin{bmatrix} \dot \bx(t)  \\ \by(t) \end{bmatrix} = \bS \begin{bmatrix} \bx(t) \\ \bu(t) \end{bmatrix},\quad t \ge 0.
\end{equation}
\end{lemma}

\subsection{Reconstruction of the well-posed system}\label{SubS:ReconstructWP}

With the system node $\bS=\sbm{  \AB \\ \CD } $ constructed from $\Sigma=\sbm{\Afrak & \Bfrak\\ \Cfrak & \Dfrak}$ as above, it is possible to recover
$\Afrak$, $\Bfrak$, $\Cfrak$,  $\Dfrak$ and the transfer function $\widehat{\Dfrak}$ from $\sbm{  \AB \\ \CD }$. We first sketch this construction, and only afterwards, we discuss the rigour of the construction.

Clearly $\Afrak^t$ is the $C_0$-semigroup generated by $A$ and $\Bfrak$ and $\Cfrak$ can be recovered via \eqref{BfrakCfrak}, taking for $\Cfrak$ the unique continuous extension from $X_1$ to $X$ mapping into $L^{2+}_{loc,Y}$. 
Finally, by \cite[Theorem 4.7.14]{StafBook} and its proof, $\Dfrak$ can be recovered as the
unique extension to a continuous  operator from $L^2_{\ell,loc,U}$ to $L^2_{\ell,loc,Y}$ of the operator
\begin{equation}  \label{Dfrak0}
	\Dfrak \bu = t\mapsto \CD\bbm{\Bfrak^t\bu\\\bu(t)},
	\quad t\in\R,
\end{equation}
defined for $\bu\in H^1_{0,loc}(\R;U)$ with support bounded to the left; see \eqref{eq:H1def} for the definition of this space.

We have seen that the operator $\sbm{\AB\\ \CD}$ arising from a well-posed system $\Sigma$ as described in \S\ref{SubS:SysNodeConstruct} is a system node. However, in general, for a system node to give rise to a well-posed system via the above construction more is needed. We shall follow Definition 10.1.1 in \cite{StafBook} and use the following terminology: given $A$ equal to the generator of $C_0$-semigroup on $X$ and
operators $B \in \cB(U, X_{-1})$ and $C \in \cB(X_1, Y)$, we say that:
\begin{itemize}
\item $B$ is an {\em $L^2$-admissible}  (here abbreviated to {\em admissible}) {\em control operator} for $A$ if the operator $\Bfrak$  defined as in \eqref{BfrakCfrak} maps $L^{2-}_{\ell, U}$ into $X$.

\item $C$ is an {\em $L^2$-admissible} (here abbreviated to {\em admissible}) {\em observation operator}  for $A$ if the operator $\Cfrak$ defined as in \eqref{BfrakCfrak} is continuous as an operator from
$X$ to $L^{2+}_{loc, Y}$.
\end{itemize}

The following result describes what additional conditions must be imposed on a system node, in order to conclude that it induces a well-posed system.

\begin{theorem}  \label{T:absnode}  Suppose that $\bS = \sbm{ A \& B \\ C \& D}$ is a system node as defined above.  Suppose that the semigroup
$t \mapsto \Afrak^t$ has growth bound $\omega_\Afrak$ and let $\omega$ be any real number satisfying $\omega > \omega_\Afrak$.  Then there is a well-posed system
$\sbm{ \Afrak & \Bfrak \\ \Cfrak & \Dfrak}$ such that $\bS$ is the system node arising from $\Sigma$ if and only if
\begin{enumerate}
\item[(1)] the operator $B \colon U \to X_{-1}$ is  admissible for $A$,

\item[(2)] the operator $C \colon X_1 \to Y$ is admissible for $A$, and

\item[(3)]  the system-node transfer function $\widehat \Dfrak_\bS$ \eqref{node-transfunc} is in $H^\infty(\C_\omega;\cB(U,Y))$.
\end{enumerate}
Explicitly, when conditions \textup{(1)}, \textup{(2)}, \textup{(3)} are satisfied, the associated well-posed system $\Sigma = \sbm{ \Afrak & \Bfrak \\ \Cfrak & \Dfrak}$ is given by
\begin{itemize}
\item $t \mapsto \Afrak^t$ is the $C_0$-semigroup generated by $A$,
\item $\Bfrak$ and $\Cfrak$ are given by formulas \eqref{BfrakCfrak}, and
\item $\Dfrak \in \cB(L^2_{\ell, loc, U}, L^2_{\ell, loc, Y})$ is a continuous extension of the operator  acting on smooth input functions $\bu$ given by
the formula \eqref{Dfrak0}.
\end{itemize}
In this case the associated system $\Sigma$ 
is $\omega$-bounded, i.e., \eqref{eq:FrakTildeDef} holds.
\end{theorem}

\begin{proof}
Assume that  $\bS$ satisfies conditions (1), (2) and  (3) in the statement of the theorem. Conditions (1) and (2) just say that conditions (i) and (ii) in Theorem 4.7.14 of \cite{StafBook} are met; once we have proved condition (iii) of this theorem, we may conclude $\bS$ is an $L^2$-well-posed system node, which, by Definition 4.7.2 in \cite{StafBook}, implies that the constructed system $\Sigma$ is well-posed. As a consequence of the Paley-Wiener Theorem \cite[Theorem 10.3.4]{StafBook}, it follows from Theorem 10.3.5 in \cite{StafBook} that condition (3) is equivalent to $\widehat \Dfrak_\bS$ being the transfer function of an operator $\Dfrak$ in $\textup{TIC}^2_\omega(U,Y)$, that is,  a causal, time-invariant operator in $\cB(L^2_{\omega,U},L^2_{\omega,Y})$. It then follows from Lemma 2.6.4 in \cite{StafBook} that $\Dfrak$ has a unique ``extension after restriction'' to an operator in $\textup{TIC}^2_{loc}(U,Y)$, which means it is a continuous, causal, time-invariant operator from $L^2_{\ell,loc,U}$ into $L^2_{\ell,loc,U}$, which is precisely what is required for the remaining condition (iii) in Theorem 4.7.14 of \cite{StafBook}. We may thus conclude that $\Sigma$ constructed from $\bS$ is a well-posed system, which generates the system node $\bS$ in the way described in Subsection \ref{SubS:SysNodeConstruct}. It then follows from the reverse construction in Subsection \ref{SubS:ReconstructWP} preceding this theorem that the operator $\Dfrak$ is indeed given by \eqref{Dfrak0}.

That the operators $\Afrak$, $\Bfrak$, $\Cfrak$ and $\Dfrak$ that constitute the well-posed system $\Sigma$ are $\omega$-bounded, follows from the discussion in Section \ref{sec:prel} after Definition \ref{def:WPtraj}.

Conversely, suppose that $\Sigma$ constructed from $\bS$ in the theorem is a well-posed system. Then it has $\omega_\Afrak$ as growth bound, so that $\Afrak$, $\Bfrak$, $\Cfrak$ and $\Dfrak$ are $\omega$-bounded, by the above argument. The properties (1)--(3) now follow from Theorem 10.3.6 in \cite{StafBook}.
\end{proof}

\subsection{Duality between admissible control/observation operators for $A$/$A^*$} \label{SubS:DualitySysNodes} Here we briefly point out the duality between admissible input pairs $(A,B)$ and admissible output pairs $(C,A)$; see also \cite[Theorem 6.2.13]{StafBook}. Let $A$ be the generator of a $C_0$-semigroup $\Afrak$, $B \in \cB(U, X_{-1})$ and $C \in \cB(X_{1},Y)$.

Let us define $A^*$ in the standard way as an unbounded operator on $X$, and let $X_1^d\subset X\subset X_{-1}^d$ be the Gelfand triple as in
\eqref{Xsub-1},  but for $A^*$ and using the parameter $\overline\beta\in\rho(A^*)$ in place of the operator $A$ and the parameter $\beta \in \rho(A)$.
Next define $B^* \in \cB(X_1^d, U)$ by identifying $U$ and $X$ with their duals and by viewing $X_{-1}^d$ as the dual of $X_1$ via the $X$-inner product
to define the duality pairing:
$$
	\Ipdp{x}{z}_{X_1,X_{-1}^d}=\Ipdp xz_X,
	\qquad x\in X_1,~z\in X.
$$
Define $C^* \in \cB(Y, X_{-1}^d)$ analogously.
When this is done it is a matter of verification to see that the operator $B^*$ is an admissible observation operator for $A^*$  if and only if
$B$ is an admissible control operator for $A$.
Similarly, if $C$ is an admissible observation operator for  $A$ , then $C^*$ is an admissible control operator for $A^*$, and vice versa.

Together with the transfer function
$$
	\widehat\Dfrak^\sharp(\lambda):=
	\widehat\Dfrak(\overline\lambda)^*,
	\qquad\lambda\in\rho(A^*),
$$
evaluated at some arbitrary $\alpha\in\rho(A^*)$, the operators $A^*$, $C^*$ and $B^*$ amount to an infinitesimal version of the duality between $\Sigma$ and
$\Sigma^d$ described in Theorem \ref{T:causdual}; in fact, the system node for the causal dual $\Sigma^d$ is
$$
\SysNode^*:\bbm{X\\Y}\supset\dom {\SmallSysNode^*}\to\bbm{X\\U},
$$
in the standard sense of unbounded adjoints.

\subsection{KYP-inequalities in terms of system nodes}\label{SubS:KYP-SySNodes}

In this subsection we show how the standard  KYP-inequality  \eqref{eq:KYP},   the strict KYP-inequality \eqref{eq:strictKYP}, and  for the semi-strict
KYP-inequality \eqref{eq:semi-strictKYP} can be expressed in terms of the system node $\bS = \sbm{ A \& B \\ C \& D}$ rather than in terms of the
well-posed system $\Sigma = \sbm{ \Afrak & \Bfrak \\ \Cfrak & \Dfrak }$, at least for the case where $H$ is bounded and strictly positive-definite.
The main tool will be Lemma \ref{L:traj-node}. 

\begin{theorem}  \label{thm:node-KYPs}
Suppose that $\Sigma = \sbm{ \Afrak & \Bfrak \\ \Cfrak & \Dfrak}$ is a well-posed system with corresponding system node
$\bS = \sbm{ A \& B \\ C \& D}$.  Then the $\Sigma$-KYP inequalities \eqref{eq:KYP}, \eqref{eq:strictKYP} and \eqref{eq:semi-strictKYP}
 correspond to $\bS$-KYP inequalities as follows.
\begin{enumerate}

\item[(1)] A bounded selfadjoint operator  $H \succc 0$ solves the  standard KYP inequality
\eqref{eq:KYP} if and only if $H$ satisfies the standard $\bS$-KYP inequality:
\begin{equation}    \label{eq:KYPnode-b}
2 \re \langle H (A \& B) \sbm{ x \\ u }, x \rangle + \| (C \& D) \sbm{ x \\ u } \|^2  \le \| u \|^2,\quad
   \sbm{ x \\ u } \in \dom{\bS}.   
 \end{equation}

\item[(2)] A bounded selfadjoint operator $H \succc 0$ on $X$ satisfies the strict KYP inequality \eqref{eq:strictKYP} if and only if
$H$ satisfies the strict $\bS$-KYP inequality:
\begin{equation}   \label{eq:strictKYPnode}
2 \re  \langle H ( A \& B ) \sbm{ x \\ u },  x \rangle
+ \| C \& D \sbm{ x \\ u } \|^2 + \delta \| x \|^2 \le \langle Hx, x \rangle  + (1 - \delta) \| u \|^2
\end{equation}
for all $\sbm{ x \\ u } \in \dom{\bS}$.

\item[(3)] A bounded selfadjoint operator $H \succc 0$ on $X$ satisfies the semi-strict KYP inequality \eqref{eq:semi-strictKYP} if and only if
$H$ satisfies the semi-strict $\bS$-KYP inequality:
$$
2 \re \langle H ( A \& B ) \sbm{ x \\ u },  x \rangle
+ \| C \& D \sbm{ x \\ u}  \|^2  \le \langle Hx, x  \rangle + (1 - \delta) \| u \|^2
$$
for all   $\sbm{ x \\ u } \in \dom{\bS}$.
\end{enumerate}
\end{theorem}

\begin{proof}[Proof of statement \textup{(1)}]
Suppose first that $H \succc 0$ is a selfadjoint operator satisfying the standard KYP inequality
 \eqref{eq:KYP}.
  Let us apply \eqref{eq:KYP} to the case
where $x = \bx(0)$ and $\bu$ is equal to the input signal for a smooth trajectory $(\bu, \bx, \by)$ in the sense of Lemma \ref{L:traj-node}.
Recalling the definition of the action of
$\sbm{ \Afrak^t & \Bfrak^t \\ \Cfrak^t & \Dfrak^t}$, we see that
\begin{equation}  \label{KYP-traj}
\| H^{\frac{1}{2}} \bx(t) \|^2 + \int_0^t \| \by(s) \|^2 \ud s  \le \| H^{\frac{1}{2}} \bx(0) \|^2 + \int_0^t \| \bu(s) \|^2 \ud s
\end{equation}
for all $t \ge 0$.   As $\bx$ is continuously differentiable and $\bu$ and $\by$ are continuous, we may move $\|H^{\frac12}\bx(0)\|^2$ over to the left-hand side in \eqref{KYP-traj}, divide by $t$, let $t\to0$, and finally observe that
\begin{equation}\label{eq:diffStateH}
\frac{\ud}{\ud s} \langle H \bx(s),  \bx(s) \rangle =
 2 \re \langle H \dot x(s), x(s) \rangle\,,
\end{equation}
in order to arrive at
$$
2 \re \langle H  \dot \bx(0), \bx(0) \rangle  +  \| \by(0) \|^2  \le \| \bu(0) \|^2. 
$$
Plugging in the differential system equations \eqref{node-sys-eq} then leads to
$$
2 \re \left\langle H ( A \& B) \bbm{ x_0 \\ \bu(0)}, x_0 \right\rangle + \left\| C \& D \bbm{ x_0 \\ \bu(0)}\right \|^2 \le \| \bu(0) \|^2\,,
$$
where $\sbm{ x_0 \\ \bu(0) }$ can be an arbitrary element of $\dom{\bS}$, thereby arriving at \eqref{eq:KYPnode-b} as wanted.

Conversely, if $H$ satisfies \eqref{eq:KYPnode-b},  we evaluate \eqref{eq:KYPnode-b} at $\sbm{ x \\ u } = \sbm{ \bx(s) \\ \bu(s) }$
taken from a smooth system trajectory $(\bu(t), \bx(t), \by(t))$ as in Lemma \ref{L:traj-node} to get
$$
2 \re \left\langle H (A \& B) \bbm{ \bx(s) \\ \bu(s) }, \bx(s) \right\rangle + \left\| C \& D \bbm{ \bx(s) \\ \bu(s) } \right\|^2 \le \| \bu(s) \|^2.
$$
Due to the differential system equations \eqref{node-sys-eq} we can rewrite this last expression as
\begin{equation}   \label{diffKYP}
2 \re \langle H \dot \bx(s), \bx(s) \rangle  + \| \by(s) \|^2 \le \| \bu(s)\|^2
\end{equation}
for all $s \ge 0$.  Again using \eqref{eq:diffStateH}, we can integrate \eqref{diffKYP} from $s=0$ to $s=t$ to arrive at
$$
\langle H \bx(t), \bx(t) \rangle - \langle H \bx(0), \bx(0) \rangle + \int_0^t \| \by(s) \|^2 \ud s \le \int_0^t \| \bu(s) \|^2 \ud s
$$
which we can interpret as saying that
$$
\bigg\| \begin{bmatrix} H^{\frac{1}{2}} & 0 \\ 0 & I \end{bmatrix}  \begin{bmatrix} \Afrak^t  & \Bfrak^t \\ \Cfrak^t & \Dfrak^t \end{bmatrix} \begin{bmatrix} x_0 \\  \bu \end{bmatrix}
\bigg \| \le \bigg\| \begin{bmatrix} H^{\frac{1}{2}} & 0 \\ 0 & I \end{bmatrix} \begin{bmatrix} x_0 \\ \bu \end{bmatrix} \bigg\|,
$$
i.e.,  the KYP-inequality \eqref{eq:KYP} holds for all $\sbm{ x_0 \\ \bu } \in \sbm{ X \\ L^2([0,t], U)}$ such that $\bu$ is sufficiently smooth
(in the sense used in Lemma \ref{L:traj-node}) and $\sbm{ x_0 \\ \bu(0) } \in \dom{\bS}$. 
Noting that the collection all such $\sbm{x_0\\\bu}$ is dense in $
\sbm{X\\L^2([0,t], U)}$, we see that \eqref{eq:KYP} continues to hold on the space $\sbm{ X \\ L^2([0,t], U) }$ as wanted.

\smallskip

\noindent
{\em Proof of \textup{(2)} and \textup{(3)}:}    The proofs of statements (2) and (3) follow in much the same way as that for (1).
For the case of statement (2), if we assume that $H \succc 0$ satisfies the strict bounded real lemma \eqref{eq:strictKYP},
apply the associated quadratic form to a vector of the form $\sbm{ \bx(0) \\ \bu }$ coming from a smooth system trajectory $(\bu, \bx, \by)$, and then also
take into account the interpretation \eqref{x(0)u-x}  for the operator $\begin{bmatrix} \Cfrak^t_{1_X,A} & \Dfrak^t_{A,B} \end{bmatrix}$,
we can interpret \eqref{eq:strictKYP} as saying that
$$
\bigg\| \begin{bmatrix} H^{\frac{1}{2}} & 0 \\ 0 & I \end{bmatrix} \begin{bmatrix} \bx(t) \\ \by(t) \end{bmatrix} \bigg\|^2
+ \delta \int_0^t \| \bx(s) \|^2 \ud s \le \| H^{\frac{1}{2}} \bx(0) \|^2 + (1 - \delta) \int_0^t \| \bu(s) \|^2 \ud s.
$$
The above argument for statement (1) then leads us to the conclusion that the differential form \eqref{eq:strictKYPnode} is equivalent to the integrated form \eqref{eq:strictKYP}.

Statement (3) follows in much the same way.  One repeats the argument used for statement (2) while ignoring the term
$$
\delta \begin{bmatrix} (\Cfrak^t_{1_X, A})^* \\ ( \Dfrak^t_{A,B})^* \end{bmatrix}
\begin{bmatrix}   \Cfrak^t_{1_X, A}  &  \Dfrak^t_{A,B} \end{bmatrix}
$$
in \eqref{eq:strictKYP}  and the term   $\delta \| x \|^2$ in  \eqref{eq:strictKYPnode}.
\end{proof}

\begin{remark}  \label{R:ASdifKYP}
Arov and Staffans \cite{ArSt07} have worked out a generalized KYP-inequality for the infinite dimensional, continuous-time setting with solution $H$ possibly unbounded
formulated directly in terms of the system node $\bS= \sbm{ A \& B \\ C \& D }$
(see Definition 5.6 and Theorem 5.7 there) to characterize when the transfer function of $\bS$ is in the Schur class. It suffices to say here that the definition of solution
there involves several auxiliary conditions in addition to the actual spatial operator inequality, all of which collapse to the inequality \eqref{eq:KYPnode-b} in case $H$ is bounded.
\end{remark}

\section{Examples of systems with $L^2$-minimality}\label{sec:L2minexamples}

In this section we consider a few concrete cases where the system $\Sigma$ is $L^2$-minimal. In the first case we assume that the $C_0$-semigroup $\Afrak$ can be embedded into a $C_0$-group. We shall first recall some facts about $C_0$-groups; for further details we refer to \cite[\S II.3]{EN00} and \S6.2 in \cite{JZ12}. By a $C_0$-group we mean a family of linear operators $\{\Afrak^t \mid t \in  {\mathbb R} \}$ on $X$ such that
$$
\Afrak^0=1_X,\qquad  \Afrak^t \Afrak^s = \Afrak^{t+s} \text{ for all } t, \, s \in {\mathbb R}
$$
and which is strongly continuous at $0$:
$$
  \lim_{t \to 0} \Afrak^t x = x \text{ for all } x \in X,
$$
where the limit is now taken from both sides and not just from the right as in the semigroup case. The {\em generator} of the $C_0$-group $\{ \Afrak^t \mid t \in {\mathbb R} \}$ is defined to be the operator $A$ with domain $\dom{A}$ given by
$$
\dom{A} = \left\{ x \in X \biggmid \lim_{t \to 0} \frac{1}{t} (\Afrak^t x - x) \text{ exists in } X \right\},
$$
again with a two-sided limit, and with action then given by
$$
  A x =  \lim_{t \to 0} \frac{1}{t} (\Afrak^t x - x),\quad  x \in \dom{A}.
$$
Among various characterizations contained in the generation theorem for groups \cite[p.\ 79]{EN00}, an operator $A$ is a generator of a $C_0$-group if and only if $A$ and $-A$ are both generators of $C_0$-semigroups, say $\Afrak_+^t$ and $\Afrak_-^t$, respectively, in which case we recover $\Afrak^t$ as
$$
  \Afrak^t x = \begin{cases}  \Afrak^t_+  x & \text{for } t \ge 0, \\ \Afrak_-^{-t} x &\text{for } t \le 0. \end{cases}
$$
The well-known case of a unitary group $\Afrak^{-t} = (\Afrak^t)^* = (\Afrak^t)^{-1}$ is the special case where the  generator $A$ is skew-adjoint, $A^* = -A$.

The above characterization of a $C_0$-group $\Afrak^t$ implies that the spectrum of the generator $A$ is contained in a strip along the imaginary axis:
\begin{equation}\label{eq:strip}
-\omega_\Afrak^-\leq \re(\lambda) \leq \omega_\Afrak^+,\quad \mbox{for some}\quad \omega_\Afrak^-,\omega_\Afrak^+\in\R
\end{equation}
determined by the respective growth bounds of $\Afrak_+^t$ and $\Afrak_-^t$, see \eqref{eq:omegaAdef}, and moreover 
\begin{equation}\label{eq:2bounds}
\| \Afrak^{t}_+ x \| \le M_+ e^{\omega^+ t} \| x \| \quad\mbox{and}\quad
\| \Afrak_-^{t} x \| \le M_- e^{\omega^- t} \| x \|,\quad t\geq 0,\, x\in X,
\end{equation}
for all $\omega^\pm>\omega_\Afrak^\pm$ and corresponding $M_\pm>0$. Using the group property, one can derive an upper and lower growth bound for the semigroup part:

\begin{lemma}\label{L:twosidedgrowth}
Let $\Afrak^t$ be a $C_0$-group with left and right growth bounds given by $\omega_\Afrak^-, \omega_\Afrak^+$. Then for every $\omega^\pm>\omega_\Afrak^\pm$ there are constants
$\delta,\rho>0$ such that
\begin{equation}\label{eq:2sidebound}
\delta\, e^{- \omega^- t} \| x \| \leq  \| \Afrak^t x \| \le 
\rho \, e^{\omega^+  t}  \| x ||,\quad t\geq 0,\, x\in X.
\end{equation}
\end{lemma}

\begin{proof}
Let $M_->0$ and $M_+>0$ be as in \eqref{eq:2bounds}. Set $\rho=M_+$ and $\delta=M_-^{-1}$. The right-hand bound follows immediately. For the left-hand bound, in the second inequality in \eqref{eq:2bounds} replace $x$ by $\Afrak^t x$ and use that $\Afrak_-^t=\Afrak^{-t} = (\Afrak^t)^{-1}$ to arrive at
$
 \| x \| \le M_- e^{\omega^- t} \| \Afrak^t  x \|$, or equivalently, $\| \Afrak^t  x \| \ge \delta \, e^{- \omega^- t} \| x \|$.
\end{proof}

We say that a $C_0$-semigroup $\Afrak^t$ {\em embeds} into a $C_0$-group, if there exists a $C_0$-group (usually also denoted by $\Afrak$) which coincides with the original semigroup $\Afrak^t$ for  $t \in \rplus$. The following proposition characterizes when a $C_0$-semigroup can be embedded into a $C_0$-group.

\begin{proposition}  \label{P:semigroup-group}
For a $C_0$-semigroup $\Afrak^t$ the following are equivalent:
\begin{enumerate}
\item[(1)]  $\Afrak^t$ embeds into a $C_0$-group;

\item[(2)] $\Afrak^t$ is invertible (in $\cB(X)$) for all $t \ge 0$;

\item[(3)] $\Afrak^t$ is invertible for some $t > 0$.
\end{enumerate}
\end{proposition}

\begin{proof}
Clearly (2) implies (3). The proposition on page 80 of \cite{EN00} states the implication (3) $\Rightarrow$ (1)  
and the remaining implication (1) $\Rightarrow$ (2) is easy: 
for $t\geq 0$ we have $\Afrak^t \Afrak^{-t}=\Afrak^0=1_X=\Afrak^{-t} \Afrak^{t}$, so that $\Afrak^t$ is invertible.
\end{proof}

If $\Afrak^t$ is a $C_0$-semigroup that embeds into a $C_0$-group, then it should at least satisfy \eqref{eq:2sidebound}; the upper bound comes for free from the one-sided strong continuity. However, it is not necessarily the case that a $C_0$-semigroup $\Afrak^t$ satisfying \eqref{eq:2sidebound} embeds into a $C_0$-group. Indeed, take $\Afrak^t = \tau_+^{-t}$ to be the right translation semigroup
on $L^2({\mathbb R}^+)$.  Then $\tau_+^{-t}$  ($t \ge 0$) is isometric and hence satisfies the lower estimate $\| \tau_+^{-t} x \| \ge \delta e^{-\omega  t} \| x \|$
with $\delta = 1$ and $\omega = 0$, but $\tau_+^{-t}$ is not onto, and hence not invertible on $L^2({\mathbb R}^+)$ for any $t > 0$.

We next give some sufficient conditions  which guarantee the $L^2$-controllability and/or $L^2$-observability of a given well-posed linear system $\Sigma$. In fact, we will show that under the assumptions of the proposition, the system is exactly controllable and/or exactly observable in any time $t>0$; see Definition 9.4.1 in \cite{StafBook}.

\begin{proposition}  \label{P:L2min-C0group}
Suppose that $\Sigma$ is a minimal well-posed linear system with transfer function $\widehat \Dfrak$ in $H^\infty({\mathbb C}^+; \cB(U, Y))$ and with its $C_0$-semigroup $\Afrak^t$ invertible on $X$ for some (and hence all) $t>0$. Then:
\begin{enumerate}
\item[(1)] Assume there exists a closed subspace $U_0$ of $U$ such that the control operator $B\in\cB(U,X_{-1})$ maps $U_0$ onto $X$ (viewed as an algebraic subspace of $X_{-1}$). Then $\Sigma$ is $L^2$-controllable.

\item[(2)] Assume there exists a closed subspace $Y_0$ of $Y$ such that, for the observation operator $C\in\cB(X_1,Y)$, the operator $P_{Y_0}C$ extends to a bounded operator from $X$ into $Y_0$ which is bounded below. Then $\Sigma$ is $L^2$-observable.

\item[(3)] Assume that $B$ and $C$ satisfy the conditions of (1) and (2), respectively. Then $\Sigma$ is $L^2$-minimal.
\end{enumerate}
\end{proposition}

\begin{proof}
Note that statement (2) follows from (1) applied to $\Sigma^d$ and that statement (3) follows simply by combining statements (1) and (2).
Thus it suffices to consider in detail only statement (1). %
We may moreover consider the restricted system where the input signals are restricted to values in $U_0$, since $L^2$-controllability of the restricted system implies $L^2$-controllability of the original system as long as $\bW_c^\bigstar$ is densely defined for the original system. Hence we will without loss of generality assume that $B$ maps $U$ onto $X$ in the sequel.

Since $\Sigma$ is observable, by Corollary \ref{C:MinExp} we see that  $\bW_c^\bigstar$  is indeed densely defined.  Then we may apply Proposition \ref{P:Wc} to get that $L^{2 -}_{\ell,U} \subset \dom{\bW_c}$ and $\bW_c |_{L^{2 - }_{\ell, U}} = \Bfrak$; then
$$
  \operatorname{Rea}(\Sigma) = \range \Bfrak \subset \range{\bW_c}.
$$
To show the $L^2$-controllability condition $ \range{\bW_c} = X$, we will actually show that $\Sigma$ is exactly controllable in any finite time $t>0$:\ For any $x\in X$ and $\delta > 0$,
we will construct an input signal $\bu \in L^2([-\delta, 0], U)$ such that $\Bfrak \bu = x$. For this, let $x \in X$, and use the surjectivity of $B \in \cB(U, X)$ to find a  $u \in U$ such that $Bu = x$.  We are done if we can find $\bu \in L^2([-\delta, 0], U)$ such that $\Bfrak \bu = Bu$, i.e.,
$$
  \int_{-\delta}^0 \Afrak^{-s} B \bu(s) \ud s = B u.
$$
As $B$ is surjective,  $B$ has a bounded right inverse $B^\dagger$, and it is easily checked that the function
$$
  \bu(s) = \frac{1}{\delta}  B^\dagger \Afrak^s B u,\quad \text{for } -\delta \le s \le 0
$$
does the job.
\end{proof}

\begin{remark}  \label{R:drawbacks}
For the infinite dimensional setting, the conditions on $B$ and $C$  in  Proposition \ref{P:L2min-C0group} are rather strong.  Indeed,
if $X$ is infinite dimensional, the surjectivity of $B$ forces that also the input space $U$ is infinite dimensional, and similarly, injectivity of $C$ forces $\dim (Y)=\infty$.  However these hypotheses are not so offensive in our application to the proof of
the strict infinite dimensional BRL (Theorem \ref{thm:stdlemmastrict} with proof to come in \S\ref{sec:proofs}),
as the idea is to embed the nominal system $\Sigma$ (which may have finite dimensional input and/or output spaces) into an auxiliary system $\Sigma_\varepsilon$ which does have infinite dimensional input and output spaces. The one remaining restrictive hypothesis in Proposition  \ref{P:L2min-C0group} (compared to the discrete-time setting of \cite{BGtH18b})  is that the semigroup can be embedded in a $C_0$-group.  This appears to be unavoidable if one wants to achieve $L^2$-controllability  ($L^2$-observability) with a bounded control (observation) operator. The following example agrees on this observation.
\end{remark}

\begin{example}\label{ex:counter}
Here we give an example of a strict Schur-class function $\widehat\Dfrak$ from $U:=\ell^2(\zplus)$ to $Y:=U$. Later on, in Example \ref{ex:HaHr} below, we shall complete the example by finding explicit the maximal and minimal, bounded and boundedly invertible solutions of the KYP inequality, as expected by Theorem \ref{thm:stdlemmastrict}.

Take $X:=U$, with the canonical orthonormal basis $\{\phi_n \mid n = 0, 1, 2, \dots \}$ where $\phi_n \in \ell^2$ has a one in position $n$
and zeros elsewhere.  Thus each vector $x \in X = \ell^2({\mathbb Z}_+)$ can be represented as $x = \sum_{=0}^\infty x_n \phi_n$
where $x_n = \langle x, \phi_n \rangle_{\ell^2({\mathbb Z}_+)}$ and $\sum_{n=0}^\infty |x_n|^2 < \infty$.  Define $A$ by
$$
  A \colon \sum_{n=0}^\infty x_n \phi_n = \sum_{n=0}^\infty -(n+1)x_n \phi_n
$$
with $ \dom{A}  = \{ x \in X \mid A x \in X\}$, i.e.,
$$
\dom{A} = \left\{ x = \sum_{n=0}^\infty x_n \phi_n \in \ell^2({\mathbb Z}_+) \biggmid  \sum_{n=0}^\infty (n+1)^2 |x_n|^2 < \infty \right\}.
$$
In particular $\phi_n\in\dom{A}$ for all $n$. By \cite[\S4.9]{StafBook}, $A$ generates an exponentially stable diagonal contraction
semigroup $\Afrak$ on $X$, which is determined by the condition
\begin{equation}\label{eq:diagsemigr}
	\Afrak^t\phi_n=e^{-(n+1)t}\phi_n,\qquad n=0,1,\ldots,
\end{equation}
since this function is the unique solution of the Cauchy problem $\dot x=Ax$ with $x(0)=\phi_n$:
$$
	\ddt e^{-(n+1)t}\phi_n=
	-(n+1)e^{-(n+1)t}\phi_n=
	Ae^{-(n+1)t}\phi_n,\qquad t\geq0.
$$
Moreover, $\|\Afrak^t\|=e^{-t}$, so that $\Afrak$ is also a contraction semigroup, and moreover
$$
	\lim_{t\to\infty} \frac{\ln\|\Afrak^t\|}t=-1,
$$
which shows that $\C_{-1}\subset\rho(A)$.

Note that the Cayley transform $\dA$ of the operator $A$ is determined by
$$
	\dA\phi_n=(1_X+A)(1_X-A)^{-1}\phi_n=-\frac{n}{2+n}\phi_n,
$$
and since $-n/(2+n)\to -1$ as $n\to\infty$, the spectral radius of $\dA$ is 1.  Hence the Cayley transform does not always map the generator of an exponentially stable semigroup to an operator which is exponentially stable in the discrete-time sense. Therefore, it is not possible to reduce the study of the strict bounded real lemma in continuous time to the discrete-time case in \cite[Theorem 1.6]{BGtH18a} by means of the Cayley transform, as was  done for the non-strict case in \cite{ArSt07}. Moreover, the semigroup $\Afrak$ cannot be embedded into a group, since \eqref{eq:strip} is violated.

Now observe that
$$
	\int_0^\infty \|\Afrak^t\phi_n\|^2\ud t=\frac{1}{2n+2},
$$
and hence the unbounded operator $C:=2(-A)^{\frac12}$ gives $\bW_ox=t\mapsto C\Afrak^t x$ bounded both from above and below, as an operator from $X$ into $L^{2+}_Y$, but with norm $\sqrt 2$ it is not the output map of a passive system; see Lemma \ref{lem:PasvContr}. However, $C$ is an infinite time admissible observation operator for $\Afrak$ and the pair $(C,A)$ is $L^2$-observable. If $C$ is made essentially more unbounded, then it is no longer an admissible observation operator for $\Afrak$, and if $C$ is made essentially more bounded, then we lose $L^2$-observability. By duality, $B:=\frac 12(-A_{-1})^{\frac12}$ is an admissible control operator for $\Afrak$ and $(A,B)$ is an $L^2$-controllable pair; note that $A_{-1}$ is described by the same formula as $A$, but the domain is extended to all of $X$. 

We now have the operators $A$, $B$ and $C$. To get a system node we still need to fix the special point $\alpha\in\C_{\omega_\Afrak}$ and the corresponding value of the transfer function $\widehat\Dfrak(\alpha)$; for convenience we take $\alpha=0$. The domain of the system node is
$$
	\dom{\AB}=\set{\bbm{x\\u}\in \begin{bmatrix}  X \\ U \end{bmatrix} \bigmid
		A_{-1}x+Bu \in X},\quad \AB=\bbm{A_{-1}&B}\big|_{\dom{\AB}},
$$
and the combined feedthrough/observation operator becomes
\begin{equation}\label{eq:CDex}
	\CD\bbm{x\\u}=C\left(x+A_{-1}^{-1}Bu\right)+\widehat\Dfrak(0)u,
	\quad \bbm{x\\u}\in\dom{\AB}.
\end{equation}
Specializing \eqref{eq:CDex} to $x=x_n\phi_n$ and $u=u_m\phi_m$ gives
\begin{equation}\label{eq:CDexPhi}
	\CD\bbm{x_n\phi_n\\u_m\phi_m}=2\sqrt{n+1}\,x_n\phi_n
	+(\widehat\Dfrak(0)-1_U)\,u_m\phi_m,\quad x_n,u_n\in\C.
\end{equation}

On the other hand, specializing \eqref{eq:CDex} to $x=(\lambda-A_{-1})^{-1}Bu$, we get from \eqref{node-transfunc} that the transfer function is
$$
\begin{aligned}
	\widehat\Dfrak(\lambda)u &=
	C\left((\lambda-A_{-1})^{-1}Bu+A_{-1}^{-1}Bu\right)+\widehat\Dfrak(0)u\\
	&=(-A)^{\frac12}\lambda\,(\lambda-A)^{-1}A_{-1}^{-1}(-A_{-1})^{\frac12}u+\widehat\Dfrak(0)u\\
	&=
	-\lambda\,(\lambda-A)^{-1}u+\widehat\Dfrak(0)u, \quad \lambda\in\C_{-1},
\end{aligned}
$$
where we in the last step used that $(-A)^{\frac12}$ commutes with $(\lambda-A)^{-1}$ and $(-A_{-1})^{\frac12}$ commutes with $A_{-1}^{-1}$; it is easy to check directly that the $m$-accretive operator $-A$ commutes with the bounded operator $(\lambda-A)^{-1}$; see \cite[Theorem 3.35 on p.\ 281]{Kato}. 

Taking for instance $\widehat\Dfrak(0):=0$, we get from \cite[Corollary 3.4.5]{StafBook} that $\widehat\Dfrak$ is a Schur function, but letting $\lambda\to\infty$ along the positive real line, we get from \cite[Theorem 3.2.9(iii)]{StafBook} that $\widehat\Dfrak(\lambda)u=u$ for all $u\in U$, and so $\widehat\Dfrak$ is \emph{not a strict} Schur function.

However, if we instead set $\widehat\Dfrak(0):=\frac{1}2 1_U$, then we get
\begin{equation}\label{eq:ExDstrict}
\widehat\Dfrak(\lambda)=-
\lambda\,(\lambda-A)^{-1}+\frac12=-
\frac12(\lambda+A)(\lambda-A)^{-1},\quad\lambda\in\C_{-1},
\end{equation}
which satisfies $	\|\widehat\Dfrak(\lambda)\|\leq\frac12$ for $\lambda\in\cplus$, i.e. this \emph{is} a strict Schur function. In Example \ref{ex:HaHr} below, we continue this example, in order to get two extremal solutions to the bounded KYP inequality \eqref{eq:KYPbdd} which are bounded both above and below.

Finally, we observe that, in both of the above cases, $\widehat\Dfrak\in H^\infty(\cplus;\cB(X))$, and then \cite[Theorem 10.3.6(iv)]{StafBook} gives that the system node $\SmallSysNode$ is
well-posed, but it is not passive, as we already saw. We may, however, apply Theorem \ref{thm:stdlemmaL2reg} to get that $\SmallSysNode$ is similar to a
passive system.
\end{example}

As the preceding example shows, $L^2$-minimality may be an exotic property. We further add to this conclusion by observing that, in general,
unless the point spectrum of $A$ is confined to a vertical strip, then the pair$(A,B)$ is not $L^2$-observable for any bounded operator $B:U\to X$. Dually,
no bounded $C:X\to Y$ makes $(C,A)$ an $L^2$-observable pair; indeed $\C^+_{\omega_\Afrak}\subset\rho(A)$, and so if $\sigma_p(A)$ is not contained in
a vertical strip, then there exists eigenpairs $(\lambda_n,\phi_n)$ of $A$, such that $\|\phi_n\|=1$ and $\re \lambda_n\to-\infty$ as $n\to\infty$.
Since $\phi_n\in\dom{A}$, we have for bounded $C$ and $\re\lambda_n<0$ that
$$
	\|\Cfrak\phi_n\|^2_{L^{2+}_Y}=
	\int_0^\infty \|C\Afrak^t\phi_n\|^2\ud t\leq
	\|C\phi_n\|^2\int_0^\infty e^{2\re\lambda_n t}\ud t \leq
	\frac{\|C\|^2}{-2\re\lambda_n}\,;
$$
here we used the extension of \eqref{eq:diagsemigr} to an arbitrary eigenpair. Thus $\phi_n\in\dom{\bW_o}$, and by letting $n\to\infty$, we get from
$\bW_o\phi_n=\Cfrak\phi_n$ that $\|\bW_o\phi_n\|\to 0$ with $\|\phi_n\|=1$. This proves that $(C,A)$ is not $L^2$-observable. The statement on controllability can be obtained by duality. Compare this to \eqref{eq:strip} and Remark \ref{R:drawbacks}.

We end this section by pointing out that observability can be strengthened into $L^2$-observability by weakening the norm in the state space and growing it,
while strengthening controllability to $L^2$-controllability can be achieved by shrinking the state space and strengthening the norm to make the $L^2$-reachable
state space Hilbert; see \cite[Theorem 9.4.7 and Proposition 9.4.9]{StafBook}. Note in particular the close relation between $L^2$-controllability/observability and
the concepts ``exact controllability/observability in infinite time (with bound $\omega=0$)'' used by Staffans; see \cite[Definitions 9.4.1--2]{StafBook}.
A difference in the approach is that we here force $\omega=0$ and accept that $\bW_c$ and/or $\bW_o$ may be unbounded, whereas in
\cite{StafBook}, Staffans is flexible about $\omega$ in order to get $\widetilde \Bfrak$ and $\widetilde \Cfrak$ in \eqref{eq:FrakTildeDef} bounded.


\section{The available storage and the required supply}\label{sec:storage}

In this section we return to the notion of storage functions associated with a well-posed system as in Definition \ref{def:storage}, which we recall here for the readers convenience:\ A function $S:X\to [0,\infty]$ is called a \emph{storage function} for the well-posed system $\Sigma$ in \eqref{def:WPsys} if $S(0)=0$ and for all trajectories $(\bu,\bx,\by)$ of $\Sigma$ on $\R^+$ it holds that
\begin{equation}\label{eq:storfndef3a}
	S\left(\bx(t)\right) + \|\pi_{[0,t]}\by\|_{L^{2+}_Y}^2 \leq
	 S\left(\bx(0)\right)+ \|\pi_{[0,t]}\bu\|_{L^{2+}_U}^2,
	\quad t>0.
\end{equation}
For systems $\Sigma$ that have densely defined $\dW_c^\bigstar$, \emph{$L^2$-regular} storage functions are defined as those storage functions that are finite-valued on $\range{\dW_c}$. A storage function $S$ is called \emph{quadratic} if there exists a positive semidefinite operator $H$ on $X$, such that
\begin{equation}\label{eq:quadstorefun}
	S(x)=S_H(x):=\left\{\begin{aligned}
		\| H^{\frac{1}{2}}x \|^2,&\quad x\in\dom{H^{\frac{1}{2}}}, \\
		\infty,&\quad x\not\in\dom{H^{\frac{1}{2}}}.
		\end{aligned}\right.
\end{equation}
Quadratic storage functions are of particular interest since they provide spatial solutions to the spatial KYP inequality.

\begin{proposition}\label{prop:storageimpliesschur}
If the well-posed system $\Sigma$ has a storage function $S$, then the transfer function $\widehat\Dfrak$ of $\Sigma$ has a unique analytic continuation to a Schur function on $\cplus$.
\end{proposition}

\begin{proof}
From \eqref{eq:storfndef3a} it is immediate that every trajectory of $\Sigma$ on $\rplus$ with $\bu\in L_U^{2+}$ and $x(0)=0$ satisfies
\begin{equation}\label{eq:external}
	0\leq S\left(\bx(t)\right) \leq
	\|\pi_{[0,t]}\bu\|_{L^{2+}_U}^2-\|\pi_{[0,t]}\by\|_{L^{2+}_Y}^2,\quad t>0.
\end{equation}
Letting $t\to\infty$ in \eqref{eq:external}, we see that $\by\in L_Y^{2+}$, and we get from \eqref{eq:trajonRplus} that
$$
	\|\Dfrak \bu\|_{L^{2+}_Y}^2=\|\by\|_{L^{2+}_Y}^2  \leq \|\bu\|_{L^{2+}_U}^2.
$$
From $\pi_-\Dfrak\pi_+=0$ and $\tau^s\Dfrak=\Dfrak\tau^s$ for all $s\in \R$, we get
$$
	\|\Dfrak\tau^s\bu\|^2_{L^2_Y}=\|\tau^s\Dfrak \bu\|^2_{L^2_Y}
	=\|\Dfrak \bu\|^2_{L^2_Y}=\|\Dfrak \bu\|^2_{L^{2+}_Y}
	\leq \|\bu\|^2_{L^{2+}_U}
	=\|\tau^s \bu\|^2_{L^2_U}.
$$
By letting $s$ run over $\R$, we obtain that $\Dfrak$ restricted to $L^2_{\ell,U}$ has a unique extension to a time-invariant, causal operator $L$ from $L^2_U$ into $L^2_Y$ with norm at most 1. This implies that $\cL L\cL^*:L^2(i\R;U)\to L^2(i\R;Y)$ coincides with a multiplication operator $M_F$ with symbol $F\in H^\infty(\cplus;\cB(U,Y))$ satisfying $\|F\|_\infty = \|\cL L\cL^*\|=\|L \|\leq 1$. Hence, $F\in\cS_{U,Y}$. Moreover, $F$ is an extension of $\widehat\Dfrak$, because for $u\in L^{2+}_U$, by \cite[Corollary 4.6.10(iii)]{StafBook} (see the last part of
Proposition \ref{prop:Transfer}) we have
$$
	F(\lambda)(\cL \bu)(\lambda) =
	(\cL L \bu)(\lambda) =
	(\cL \Dfrak \bu)(\lambda) =
	\widehat\Dfrak(\lambda)(\cL \bu)(\lambda),
	\quad \lambda\in\C_{\omega_0},
$$
where $\omega_0:=\max\set{\omega_\Afrak,0}$. From $\cL L^{2+}_U=H^{2+}_U$, we now get $\widehat\Dfrak\big|_{\C_{\omega_0}}=F\big|_{\C_{\omega_0}}$. The continuation $F$ of $\widehat\Dfrak$ to the open connected set $\cplus$ is unique since $\C_{\omega_0}$ has an interior cluster point.
\end{proof}

\begin{proposition}\label{P:StorageKYP}
Assume that $S=S_H$ is of the form \eqref{eq:quadstorefun} with $H$ on $X$ positive semidefinite. Then $S_H$ is a storage function for $\Sigma$ if and only if $H$ is a spatial solution to the KYP inequality \eqref{eq:HhalfDomCond}--\eqref{eq:KYP}.
\end{proposition}

\begin{proof}
Let $S$ be quadratic, i.e., $S=S_H$ as in \eqref{eq:quadstorefun} for some positive semidefinite operator $H$ on $X$. First assume that $S$ is a storage function for $\Sigma$, so that \eqref{eq:storfndef3a} holds for all trajectories $(\bu,\bx,\by)$ on $\rplus$ of $\Sigma$. Pick $t>0$, $x_0\in \dom{H^{\frac{1}{2}}}$ and $\bu\in L^{2}([0,t];U)$ arbitrarily. By \eqref{BCD^t} and \eqref{eq:trajonRplus},
\[
	\bx(t):=\Afrak^t x_0+\Bfrak^t \bu 
	\quad\mbox{and}\quad
	\pi_{[0,t]}\by:= \Cfrak^t x_0+\Dfrak^t \bu,\quad t>0,
\]
define a trajectory $(\bu,\bx,\by)$ on $[0,t]$ of $\Sigma$ with $\bx(0)=x_0$. Now \eqref{eq:storfndef3a} and $S(x_0)<\infty$ imply that $S(\bx(t))<\infty$, and hence that $\Afrak^t x_0+\Bfrak^t \bu=\bx(t)\in \dom{H^{\frac{1}{2}}}$.  Taking first $\bu=0$ and then $x_0=0$, we get \eqref{eq:HhalfDomCond}.

Since $S=S_H$, we obtain that
\[
\begin{aligned}
S\left(\bx(t)\right)+\|\pi_{[0,t]}\by\|^2
=\left\|\bbm{H^{\frac{1}{2}}&0\\0&I}\bbm{\bx(t)\\\pi_{[0,t]}\by}\right\|^2
&=\left\|\bbm{H^{\frac{1}{2}}&0\\0&I}
		\bbm{\Afrak^t&\Bfrak^t\\\Cfrak^t&\Dfrak^t}\bbm{x_0\\\bu}\right\|^2\\
\text{and}\qquad S\left(x_0\right)+\|\bu\|^2&
=\left\|\bbm{H^{\frac{1}{2}}&0\\0&I}\bbm{x_0\\\bu}\right\|^2.
\end{aligned}
\]
Hence \eqref{eq:storfndef3a} is equivalent to
\begin{align}
\left\|\bbm{H^{\frac{1}{2}}&0\\0&I}
		\bbm{\Afrak^t&\Bfrak^t\\\Cfrak^t&\Dfrak^t}\bbm{x_0\\ \bu}\right\|^2
\leq \left\|\bbm{H^{\frac{1}{2}}&0\\0&I}\bbm{x_0\\ \bu}\right\|^2.\label{eq:KYPinproof}
\end{align}
Since $x_0 \in \dom{H^{\frac{1}{2}}}$, $t>0$ and
$\bu\in L^2([0,t];U)$ were chosen arbitrarily, we obtain \eqref{eq:KYP}. Conversely, it is clear that \eqref{eq:KYP} implies \eqref{eq:KYPinproof} and
hence that \eqref{eq:storfndef3a} holds.
\end{proof}

Next we explain how solutions to the spatial KYP-inequality for a well-posed system relate to the solutions to the spatial KYP-inequality of the dual system.

\begin{proposition}\label{P:DualKYPsols}
Let $\Sigma$ be a well-posed system with causal dual $\Sigma^d$. A positive definite operator $H$ on $X$ is a spatial solution to the KYP-inequality for $\Sigma$ if and only if $H^{-1}$ is a spatial solution to the KYP-inequality for $\Sigma^d$: For all $t>0$ it holds that 
$$
	\Afrak^{t*}\dom{H^{-\frac12}}\subset\dom{H^{-\frac12}}\,,\quad
	\Cfrak^{t*}L^2([0,t];Y)\subset\dom{H^{-\frac12}}
$$ 
and
\begin{equation}   \label{KYP-d}
 \left\| \begin{bmatrix} H^{-\frac{1}{2}} & 0 \\ 0 & 1 \end{bmatrix} \begin{bmatrix} \Afrak^t & \Bfrak^t \\ \Cfrak^t & \Dfrak^t \end{bmatrix}^* \begin{bmatrix} x \\ \by \end{bmatrix}
\right\| \le \left\|\begin{bmatrix} H^{-\frac{1}{2}} & 0 \\ 0 & 1 \end{bmatrix} \begin{bmatrix} x \\ \by \end{bmatrix} \right\|,   \quad
 \begin{bmatrix} x \\ \by \end{bmatrix} \in \begin{bmatrix} \dom{H^{-\frac{1}{2}}}\\ L^2([0,t], Y) \end{bmatrix}.
\end{equation}
\end{proposition}

The proof could be carried out by mechanically imitating the proof of \cite[Proposition 5.3]{BGtH18b}, replacing $\sbm{\dA&\dB\\\dC&\dD}$ by $\sbm{\Afrak^t&\Bfrak^t\\\Cfrak^t&\Dfrak^t}$. However, as Proposition \ref{P:DualKYPsols} is not a core result of our theory, we illustrate how some continuous-time results can be imported from the discrete-time case by discretization using lifting of the input and output signals, combined with sampling of the state, as described in \cite[\S2.4]{StafBook}.

\begin{proof}[Proof of Proposition \ref{P:DualKYPsols}]

That \eqref{KYP-d} is a correct statement of the spatial KYP inequality for $\Sigma^d$, with solution denoted by $H^{-\frac12}$ instead of by $H^{\frac12}$, follows from Lemma \ref{L:dualadj}, the unitarity of $\sbm{1_X&0\\0&\Lambda^t_K}$ and the fact that $\sbm{H^{\frac12}&0\\0&1}$ commutes with $\sbm{1&0\\0&(\Lambda^t_K)^*}$. 

Now let $H$ be a solution to the spatial KYP equality in the sense of Theorem \ref{thm:stdlemma} and fix $t>0$ arbitrarily. Then $H$ is also a solution to the spatial KYP inequality for the discrete time system $\sbm{\dA&\dB\\\dC&\dD}:=\sbm{\Afrak^t&\Bfrak^t\\\Cfrak^t&\Dfrak^t}$ with input space $L^2([0,T];U)$, state space $X$ and output space $L^2([0,T];Y)$, in the sense of \cite[Theorem 1.3]{BGtH18b}. Then \cite[Proposition 5.3]{BGtH18b} gives that $H^{-1}$ is a solution to the spatial KYP inequality for the discrete-time system $\sbm{\dA&\dB\\\dC&\dD}^*:=\sbm{\Afrak^{t*}&\Cfrak^{t*}\\\Bfrak^{t*}&\Dfrak^{t*}}$, so that \eqref{KYP-d} holds. Since $t>0$ was arbitrary, we obtain the result.
\end{proof}

In Proposition \ref{prop:storageimpliesschur} we proved that the existence of a storage function implies that the transfer function coincides with a Schur function on some right half-plane. In order to prove the converse implication, we now introduce the \emph{available storage} 
\begin{equation}\label{eq:Sa}
	S_a(x_0):=\sup_{\bv\in L^{2+}_{loc,U},\, t>0}
		\left(\| \pi_{[0,t]}\by\|^2_{L^{2+}_Y}
		- \|  \pi_{[0,t]}\bv \|^2_{L^{2+}_U} \right),\quad x_0\in X,
\end{equation}
where in the supremum, $\by$ is the output signal of the trajectory on $\rplus$ of $\Sigma$, with input $\bv$ and initial state $x_0$, as well as the \emph{required supply} 
\begin{equation}\label{eq:Sr}
	S_r(x_0):=\inf_{(\bv,\by,t)\in\Vfrak_{x_0}}
		\left(\| \pi_{[t,0]} \bv \|^2_{L^{2-}_U} - \| \pi_{[t,0]} \by \|^2_{L^{2-}_Y}\right),\quad x_0\in X,
\end{equation}
where
$$
	\Vfrak_{x_0}:=\left\{(\bv,\by,t)\in L^2_{\ell,loc,U\times Y}\times \rminus \biggmid \begin{array}{l} (\bv,\bx,\by)\text{ is a trajectory of $\Sigma$ on $\R$,}\\ \bx(0)=x_0,~\supp {\pi_-\bv}\subset[t,0] \end{array}\right\}.
$$

We need the following lemma in order to prove that $S_a$ and $S_r$ are storage functions if $\widehat\Dfrak\in \cS_{U,Y}$.

\begin{lemma}\label{lem:tcontr}
Let $(\bu,\bx,\by)$ be a trajectory on $\rplus$ with $\bx(0)=0$, of a system $\Sigma$ whose transfer function is in $\cS_{U,Y}$. Then
$$
	\|  \pi_{[0,t]}\by \|^2_{L^{2+}_Y} \leq
	\|  \pi_{[0,t]}\bu \|^2_{L^{2+}_U},\quad t>0.
$$
\end{lemma}
\begin{proof}
By Theorem \ref{thm:hankel}, the operator $L_\Sigma$ in \eqref{L2inout} is a contraction from $L^2_U$ into $L^2_Y$, such that $L_\Sigma \bu=\Dfrak \bu$ for all $\bu\in L^{2+}_{U}$. By \eqref{eq:trajonRplus}, $\by=\Dfrak \bu$, so that item (4) of Definition \ref{def:WPsys} gives
$$
\begin{aligned}
	\pi_{[0,t]}\by&=\pi_{[0,t]}\Dfrak \pi_{[0,t]}\bu+\pi_{[0,t]}\Dfrak \pi_{(t,\infty)}\bu
	=\pi_{[0,t]}L_\Sigma \pi_{[0,t]}\bu+\pi_{[0,t]}\tau^{-t}\Dfrak\tau^t \pi_{(t,\infty)}\bu \\
	&=\pi_{[0,t]}L_\Sigma \pi_{[0,t]}\bu+\tau^{-t}\pi_{[-t,0]}\Dfrak \pi_+\tau^t \bu
	=\pi_{[0,t]}L_\Sigma \pi_{[0,t]}\bu,
\end{aligned}
$$
and then $\|\pi_{[0,t]}\by\|=\|\pi_{[0,t]} L_\Sigma \pi_{[0,t]}\bu\|\leq \| \pi_{[0,t]}\bu\|$.
\end{proof}

In the next result, we do not assume minimality, in contrast to many similar results in the literature.

\begin{theorem}\label{thm:schurimpliesstorage}
Assume that the well-posed system $\Sigma$ has transfer function in $\cS_{U,Y}$. Then $S_a$ and $S_r$ are storage functions for $\Sigma$, which are extremal in the sense that every other storage function $S$ for $\Sigma$ satisfies
\begin{equation}\label{eq:StorFunOrder}
	S_a(x_0)\leq S(x_0)\leq S_r(x_0),\quad x_0\in X.
\end{equation}
\end{theorem}
\begin{proof}
\emph{Step 1: $S_a$ is a storage function for $\Sigma$.} Choose $\bv=0$ in \eqref{eq:Sa} to obtain that $S_a(x_0)\geq0$ for all $x_0\in X$. On the other hand, by Lemma \ref{lem:tcontr}, $\| \pi_{[0,t]}\by\|-\|\pi_{[0,t]}\bv\|\leq0$ for all trajectories $(\bv,\bx,\by)$ on $\R^+$ with $\bv\in L^{2+}_{loc,U}$ and $\bx(0)=0$, and all $t>0$. Thus $S_a(0)=0$.

Let $(\bu,\bx,\by)$ be a system trajectory of $\Sigma$ over $\R^+$ and fix $t>0$. Let $\bv\in L^{2+}_{loc,U}$ and write $\bx_\bv$ and $\by_\bv$ for the state and output trajectory on $\rplus$ of $\Sigma$ corresponding to the input $\bv$ and initial state $\bx_\bv(0)=\bx(t)$. Define
\[
(\widetilde{\bv},\widetilde{\bx},\widetilde{\by}):= \pi_{[0,t)}(\bu,\bx,\by) +\tau^{-t} (\bv,\bx_\bv,\by_\bv).
\]
Since $\bx_\bv(0)=\bx(t)$, trajectory property (4) listed after Definition \ref{def:WPtraj} gives that $(\widetilde{\bv},\widetilde{\bx},\widetilde{\by})$ is also a trajectory of $\Sigma$ over $\R^+$ with $\widetilde{\bx}(0)=\bx(0)$. For every $s>0$, using \eqref{eq:Sa}, we now have
\begin{align*}
&\|\pi_{[0,s]}\by_\bv\|^2_{L^{2+}_Y} - \|\pi_{[0,s]} \bv \|^2_{L^{2+}_U}  = \|\pi_{[t,t+s]} \tau^{-t} \by_\bv\|^2_{L^{2+}_Y} - \|\pi_{[t,t+s]} \tau^{-t} \bv \|^2_{L^{2+}_U}\\
&\qquad \qquad = \|\pi_{[0,t+s]} \widetilde{\by}\|^2_{L^{2+}_Y} - \|\pi_{[0,t+s]} \widetilde{\bv} \|^2_{L^{2+}_U} - \|\pi_{[0,t]}  \by\|^2_{L^{2+}_Y} + \|\pi_{[0,t]} \bu \|^2_{L^{2+}_U}\\
& \qquad \qquad \leq S_a(\bx(0)) + \int_0^t\|\bu(\tau)\|_U^2\ud \tau -\int_0^t\|\by(\tau)\|_Y^2\ud s.
\end{align*}
Taking supremum over $\bv\in L^{2+}_{loc,U}$ and $s>0$ it follows that $S_a$ satisfies \eqref{eq:storfndef3a}.

\smallskip

\noindent
\emph{Step 2: $S_r$ is a storage function for $\Sigma$.} For $x_0\not\in \range\Bfrak$ it follows from \eqref{eq:trajonR} that $\Vfrak_{x_0}=\emptyset$, so that $S_r(x_0)=\inf \emptyset=\infty\geq 0$. Now assume that $x_0\in\range\Bfrak$ and choose $\bv\in L^{2}_{\ell, loc,U}$ with $\Bfrak \pi_- \bv=x_0$ and $t<0$ with $\supp {\pi_-\bv}\subset[t,0]$ arbitrarily.
Let $(\bv,\bx_\bv,\by_\bv)$ be the associated trajectory of $\Sigma$ on $\R$, so that $\bx_\bv(t)=\Bfrak \pi_- \tau^t \bv =0$ and $\bx_\bv(0)=\Bfrak \pi_- \bv =x_0$.
By trajectory property (2),
$\tau^t(\bv, \bx_\bv,\by_\bv)$ is a trajectory of $\Sigma$ on $\rplus$ with $(\tau^t  \bx_\bv)(0)=\bx_\bv(t)=0$.  Then Lemma \ref{lem:tcontr} gives that
$$
\|\pi_{[t,0]} \by_\bv\|_{L^{2-}_Y}
=\|\pi_{[0,-t]} \tau^t \by_\bv\|_{L^{2+}_Y}
\leq \|\pi_{[0,-t]}\tau^t \bv\|_{L^{2+}_U}
=\|\pi_{[t,0]}\bv\|_{L^{2-}_U},
$$
that is,
\[
\|\pi_{[t,0]}\bv\|_{L^{2-}_U}-\|\pi_{[t,0]} \by_\bv\|_{L^{2-}_Y}\geq 0.
\]
Taking the infimum over all pairs $(\bv,t)\in L^2_{\ell,loc,U}\times\rminus$ with $\Bfrak\pi_-\bv=x_0$ and $\supp {\pi_-\bv}\subset[t,0]$, we conclude that $S_r(x_0)\geq0$. For $x_0=0$, we may make the particular choice $\bv=0$ in \eqref{eq:Sr}, in order to get $S_r(0)\leq0-0=0$.

To see that $S_r$ satisfies \eqref{eq:storfndef3a}, we give a similar argument as in Step 1. Let $(\bu,\bx,\by)$ be a system trajectory of $\Sigma$ over $\R^+$
and fix $t>0$. If $\bx(0)\not\in \range\Bfrak$, then $S_r(\bx(0))=\inf\emptyset =\infty$, and hence \eqref{eq:storfndef3a} is satisfied. Now assume that
$\bx(0)\in \range\Bfrak$, say with $\bx(0)=\Bfrak \bv_0$. Then $\supp {\bv_0}\subset[s,0]$ for some $s<0$ and we let $(\bv,\bx_\bv,\by_\bv)$ be an
arbitrary trajectory of $\Sigma$ over $\R$ with $\pi_- \bv=\bv_0$; then also $\bx_\bv(0)=\Bfrak\pi_-\bv=\bx(0)$. Define
\[
(\widetilde{\bv},\widetilde{\bx},\widetilde{\by}):=\tau^t \pi_- (\bv,\bx_\bv,\by_\bv) + \tau^{t} (\bu,\bx,\by).
\]
Using that $\bx_\bv(0)=\bx(0)$, we obtain from trajectory properties (5) and (3) that $(\widetilde{\bv},\widetilde{\bx},\widetilde{\by})$ is a trajectory of
$\Sigma$ over $\R$ with $\supp{\pi_-\widetilde \bv}\subset[s-t,0]$ and $\widetilde{\bx}(0)=\bx(t)$.
Then we have from \eqref{eq:Sr} that
\begin{align*}
&\|\pi_{[s,0]} \bv \|^2_{L^{2-}_U}  - \|\pi_{[s,0]}\by_\bv\|^2_{L^{2-}_Y}
 = \|\pi_{[s-t,-t]}\tau^t \bv \|^2_{L^{2-}_U}  - \|\pi_{[s-t,-t]} \tau^t \by_\bv\|^2_{L^{2-}_Y}\\
& \qquad \qquad = \|\pi_{[s-t,0]}\widetilde{\bv} \|^2_{L^{2-}_U}  - \|\pi_{[s-t,0]}\widetilde{\by}\|^2_{L^{2-}_Y}
-\|\pi_{[-t,0]}\tau^t \bu \|^2_{L^{2-}_U}  + \|\pi_{[-t,0]} \tau^t \by\|^2_{L^{2-}_Y}\\
& \qquad \qquad  \geq
S_r(\bx(t)) -\int_0^t\|\bu(\tau)\|_U^2\ud \tau+\int_0^t\|\by(\tau)\|_Y^2\ud \tau.
\end{align*}
Taking the infimum over all $(\bv,\by_\bv,s)\in\Vfrak_{x_0}$, we obtain that \eqref{eq:storfndef3a} holds for $S=S_r$. Hence $S_r$ is a storage function.

\smallskip

\noindent
\emph{Step 3: Every storage function $S$ for $\Sigma$ satisfies $S_a\leq S\leq S_r$.} Let $S$ be an arbitrary storage function for $\Sigma$ and choose $x_0\in X$. If $S(x_0)=\infty$, then certainly $S_a(x_0)\leq S(x_0)$. Hence assume $S(x_0)<\infty$. Now let $(\bu,\bx,\by)$ be an arbitrary trajectory of $\Sigma$ on $\R^+$ with $\bx(0)=x_0$ and fix a $t>0$. Since $S(\bx(0))=S(x_0)<\infty$, by \eqref{eq:storfndef3a} we obtain that $S(\bx(t))<\infty$. Reordering \eqref{eq:storfndef3a}, we obtain that
\[
\|\pi_{[0,t]}\by\|_{L^2_Y}^2 - \|\pi_{[0,t]}\bu\|_{L^2_U}^2 \leq S(\bx(0))-S(\bx(t))\leq S(x_0).
\]
Taking the supremum over all trajectories $(\bu,\bx,\by)$ of $\Sigma$ on $\R^+$ with $\bx(0)=x_0$ and all $t>0$, we obtain that $S_a(x_0)\leq S(x_0)$. Hence $S_a(x_0)\leq S(x_0)$ for all $x_0\in X$.

Now we turn to the inequality for $S_r$. If $x_0\not\in \range{\Bfrak}$, then $S_r(x_0)=\infty$, and we clearly have $S(x_0)\leq S_r(x_0)$. Hence, assume that $x_0\in \range{\Bfrak}$ and let $\bu\in L^{2-}_{\ell,U}$ be such that $x_0=\Bfrak \bu$. Let $(\bu,\bx,\by)$ be the uniquely determined trajectory for $\Sigma$ over $\R$ with input $\bu$, and fix $t<0$ such that $\supp \bu\subset [t,0]$. Since $\bx(t)=\Bfrak\pi_-\tau^t\bu=\Bfrak0=0$, trajectory properties (1) and (2) give that
\[
(\widetilde{\bu},\widetilde{\bx},\widetilde{\by}):=\pi_+ \tau^t(\bu,\bx,\by)
\]
is a trajectory of $\Sigma$ over $\R^+$, with $\widetilde{\bx}(0)=0$ and $\widetilde{\bx}(-t)=x(0)=x_0$. Hence $S(\widetilde{\bx}(0))=0$. By \eqref{eq:storfndef3a}, we then have
\begin{align*}
\| \pi_{[t,0]} \bu \|^2_{L^{2-}_U} - \| \pi_{[t,0]} \by \|^2_{L^{2-}_Y} &
= \| \pi_{[0,-t]} \tau^t \bu \|^2_{L^{2+}_U} - \| \pi_{[0,-t]} \tau^t \by \|^2_{L^{2+}_Y}\\
&= \| \pi_{[0,-t]} \widetilde{\bu} \|^2_{L^{2+}_U} - \| \pi_{[0,-t]} \widetilde{\by} \|^2_{L^{2+}_Y}\\
& \geq S(\widetilde{\bx}(-t)) = S(x_0).
\end{align*}
Now, in the left hand side of the inequality, take the infimum over all trajectories $(\bu,\bx,\by)$ of $\Sigma$ on $\R$ such that $\bx(0)=x_0$, and all $t$ such that $\supp {\pi_-\bu}\subset[t,0]$. It then follows that $S_r(x_0)\geq S(x_0)$.
\end{proof}

Combining Proposition \ref{prop:storageimpliesschur} and Theorem \ref{thm:schurimpliesstorage}, we get the following corollary.

\begin{corollary}
The transfer function of a well-posed system $\Sigma$ has an analytic continuation in the Schur class if and only if $\Sigma$ has a storage function.
\end{corollary}

Next we derive more explicit formulas for $S_a$ and $S_r$, in terms of the operators constituting $\Sigma$, and we determine quadratic storage functions for
$\Sigma$, leading to, in general unbounded, solutions to the KYP inequality for $\Sigma$. For this purpose, assume
$\widehat{\Dfrak}|_{\mathbb{C}^+\bigcap \dom{ \widehat{\Dfrak}}}$  has an analytic continuation to a function in $\cS_{U,Y}$. By item (1) of
Theorem \ref{thm:hankel}, the operator $L_\Sigma$ in \eqref{L2inout} decomposes as
\begin{equation}\label{LfrakDec}
L_\Sigma=\bbm{\widetilde{\Tfrak}_\Sigma & 0\\ \Hfrak_\Sigma & \Tfrak_\Sigma}:\bbm{L^{2-}_U\\ L^{2+}_U} \to \bbm{L^{2-}_Y\\ L^{2+}_Y},
\end{equation}
with $\Hfrak_\Sigma$ the $L^2$-Hankel operator of \eqref{HankelBounded}. Since $\widehat{\Dfrak}\in\cS_{U,Y}$, we have
$\|L_\Sigma\|=\|M_{\widehat{\Dfrak}}\|=\|\widehat{\Dfrak}\|_\infty\leq 1$. Hence, also $\widetilde{\Tfrak}_\Sigma$, $\Hfrak_\Sigma$ and $\Tfrak_\Sigma$
are contractions.
In the statement of the lemma, the reader should recall the notation $D_T: = (I - T^* T)^{\frac{1}{2}}$ used to denote the {\em defect operator} of a Hilbert-space contraction operator $T$, as defined at the end of \S\ref{sec:intro}.

\begin{lemma}\label{L:SaSrOp}
Let $\Sigma=\sbm{\Afrak&\Bfrak\\ \Cfrak & \Dfrak}$ be a well-posed system, such that $\widehat{\Dfrak}\in\cS_{U,Y}$. Define $\dW_o$ as in
\S\ref{sec:L2maps} and decompose $L_\Sigma$ in \eqref{L2inout} as in \eqref{LfrakDec}. Then
\begin{align}
S_a(x_0) & = \sup_{\bu\in L^{2+}_U} \|\dW_o x_0 + \Tfrak_\Sigma \bu \|^2_{L^{2+}_Y} - \|\bu\|^2_{L^{2+}_U},\quad x_0\in\dom{\dW_o}, \label{Sa}\\
S_r(x_0) & =  \inf_{\bu\in L^{2-}_{\ell,U}, x_0=\Bfrak \bu}  \| D_{\widetilde\Tfrak_\Sigma} \bu \|^2_{L^{2-}_{U}}, \quad x_0\in X  \label{Sr},
\end{align}
and $S_a(x_0)=\infty$ in case $x_0\not\in\dom{\dW_o}$.   Finally, $S_r(x_0)<\infty$ if and only if $x_0\in \textup{Rea}\,(\Sigma)=\range\Bfrak$.
\end{lemma}

Note that for each $x_0\in X$, formula \eqref{Sa} exhibits $S_a(x_0)$ as the norm squared of $\dW_ox_0$ in the Brangesian complement of the
space $\range{\Tfrak_\Sigma}$; see the notes to Chapter I of \cite{SaraBook}, or \cite[\S3]{ArStKrein}.

\begin{proof}[Proof of Lemma \ref{L:SaSrOp}]
We start with $S_a$. Using \eqref{eq:Sa}, \eqref{eq:trajonRplus} and \eqref{BCD^t}, it follows that
\[
	S_a(x_0)=\sup_{\bv\in L^{2+}_{loc,U},\, t>0}
		\left(\| \Cfrak^t x_0+\Dfrak^t \bv \|^2_{L^{2+}_Y}
		- \|  \pi_{[0,t]}\bv \|^2_{L^{2+}_U} \right),\quad x_0\in X.
\]
In case $x_0\not\in \dom{\dW_o}$, we have $\Cfrak x_0\not\in L^{2+}_Y$, and fixing $\bv=0$ in the preceding supremum, we see that
\[
S_a(x_0)\geq \sup_{t>0}\|\Cfrak^t x_0\|^2_{L^{2+}_Y}=\sup_{t>0}\|\pi_{[0,t]}\Cfrak x_0\|^2_{L^{2+}_Y}=\infty.
\]
Now take $x_0\in\dom{\dW_o}$. Then $\Cfrak^t x_0 =\pi_{[0,t]}\dW_o x_0$. For now, fix $t>0$ and $\bv\in L^{2+}_{loc,U}$. Combining the causality and time-invariance of $\Dfrak$, see item (4) of Definition \ref{def:WPsys}, it follows that $\pi_{[0,t]}\Dfrak =\pi_{[0,t]}\Dfrak \pi_{(-\infty,t]}$. By Theorem \ref{thm:hankel} and because $\supp{\bv}\subset [0,\infty)$, we have $\Dfrak^t \bv=\pi_{[0,t]} \Dfrak \pi_{[0,t]}\bv=\pi_{[0,t]} L_\Sigma \pi_{[0,t]} \bv = \pi_{[0,t]} \Tfrak_\Sigma \pi_{[0,t]} \bv$. Thus $S_a$ can be written as
\[
	S_a(x_0)=\sup_{\bv\in L^{2+}_{loc,U},\, t>0}
		\left(\|\pi_{[0,t]} (\dW_o x_0+ \Tfrak_\Sigma \pi_{[0,t]}\bv) \|^2_{L^{2+}_Y}
		- \|  \pi_{[0,t]}\bv \|^2_{L^{2+}_U} \right).
\]
Next we show that $\pi_{[0,t]}$ can be removed everywhere in the right hand side. Set $\bw:=\pi_{[0,t]}\bv\in L^{2+}_U$, so that
\begin{align*}
&  \|\pi_{[0,t]} (\dW_o x_0+ \Tfrak_\Sigma \pi_{[0,t]}\bv) \|^2_{L^{2+}_Y} - \|  \pi_{[0,t]}\bv \|^2_{L^{2+}_U}=\\
 &\qquad\qquad\qquad =\|\pi_{[0,t]} (\dW_o x_0+ \Tfrak_\Sigma \bw) \|^2_{L^{2+}_Y} - \| \bw \|^2_{L^{2+}_U} \\
    &\qquad\qquad\qquad \leq\| \dW_o x_0+ \Tfrak_\Sigma \bw \|^2_{L^{2+}_Y} - \| \bw \|^2_{L^{2+}_U}.
\end{align*}
It follows that $S_a(x_0)$ is dominated by the right-hand side of \eqref{Sa}, and equality is approached as $t\to\infty$. Thus \eqref{Sa} holds.

Now we turn to the proof of the formula for $S_r$.
 If $x_0 \notin \operatorname{Rea}(\Sigma) = \range{\Bfrak}$, then $\Vfrak_{x_0} = \emptyset$
$S_r(x_0) = \infty$ in
\eqref{eq:Sr} as in Step 3 in the proof of Theorem \ref{thm:schurimpliesstorage} and in this case \eqref{Sr} is correct.
Next suppose that $x_0 \in \operatorname{Rea}(\Sigma) = \range{\Bfrak}$  so the set $\Vfrak_{x_0} \neq \emptyset$.
Let $(\bv, \by,t)$ be an arbitrary element of $\Vfrak_{x_0}$.  Thus $\supp{\pi_- \bv} \subset [t, 0]$, $(\bv, \by)$ embeds into a system trajectory
$(\bv,\bx,\by)$ of $\Sigma$ on ${\mathbb R}$ such that $\bx(0) = x_0$.

By \eqref{eq:trajonR}, combined with the causality and time-invariance of $\Dfrak$, we have
$$
	\pi_-\by=\pi_- \Dfrak \bv =\pi_- \Dfrak \pi_- \bv=\pi_- \Dfrak \pi_{[t,0]} \bv
		=\pi_{[t,0]}\Dfrak\pi_{[t,0]}\bv=\pi_{[t,0]}\by.
$$
In particular, the value of $\|\pi_- \bv\|^2- \| \pi_- \by\|^2=\|\pi_{[t,0]}\bv\|^2- \|\pi_{[t,0]}\by\|^2$ only depends on $\bu:=\pi_- \bv\in L^{2-}_{\ell,U}$, and
 thus we may assume without loss of generality that $\bv\in L^{2-}_{\ell,U}$. In that case, Theorem \ref{thm:hankel} shows that $\by=\Dfrak \bu= L_\Sigma \bu$
 and by \eqref{LfrakDec} we have $\pi_- \by = \pi_- \Dfrak \pi_- \bu = \widetilde{\Tfrak}_\Sigma \pi_-\bv$. Thus
\begin{equation}\label{SrId}
\| \pi_{[t,0]} \bv \|^2_{L^{2-}_U} - \| \pi_{[t,0]} \by \|^2_{L^{2-}_Y}
=\| \pi_- \bv \|^2_{L^{2-}_U} - \| \widetilde{\Tfrak}_\Sigma \pi_-\bv \|^2_{L^{2-}_Y}
=\|D_{\widetilde{\Tfrak}_\Sigma} \pi_- \bv \|^2_{L^{2-}_U}.
\end{equation}
As $(\bv, \by, t)$ was chosen to be an arbitrary element of $\Vfrak_{x_0}$ and $\bv\in L^{2-}_{\ell,U}$ satisfies $x_0=\Bfrak\bu$, we conclude that $S_r(x_0)$ (as defined by \eqref{eq:Sr})
is greater than or equal to the right-hand side of \eqref{Sr}. 

To conclude that in fact equality holds, just note that starting from $\bu\in L^{2-}_{\ell,U}$ with $x_0=\Bfrak \bu$ one obtains a triple $(\bv,\by,t)$ in $\Vfrak_{x_0}$
by taking $\bv:=\bu$, letting $t<0$ be such that $\supp{\pi_-\bv}=\supp{\bu}\subset[t,0]$, and defining $\bx$ and $\by$ by \eqref{eq:trajonR}. Then \eqref{SrId} shows that $S_r(x_0)$ is dominated by the right-hand side of \eqref{Sr}, and hence the expressions for $S_r$ are equal, as claimed.
\end{proof}

By the preceding analysis, $S_r(x_0)=\infty$ precisely when $x_0\not\in \textup{Rea}\,(\Sigma)=\range \Bfrak$ which in general is a proper subset of $\range{\dW_c}$; hence it is not an $L^2$-regular storage function as defined at the beginning of \S\ref{sec:storage}.
However, assuming that $\dom{\dW_c^\bigstar}$ is dense,  
we can define the following version of $S_r$:
\begin{equation}\label{SrMod}
\underline{S}_r(x_0) :=  \inf_{\bu\in \dW_c^{-1}(\{x_0\})}  \| D_{\widetilde\Tfrak_\Sigma} \bu \|^2_{L^{2-}_{U}}, \quad x_0\in X,
\end{equation}
where
$$
	\dW_c^{-1}(\{x_0\}):=\{\bu\in \dom{\dW_c} \mid \dW_c \bu=x_0\}.
$$

\begin{proposition}\label{P:L2reg}
Assume that the well-posed system $\Sigma$ has transfer function in $\cS_{U,Y}$ and that $\dW_c^\bigstar$ is densely defined. Then $S_a$ and $\underline{S}_r$ are $L^2$-regular storage functions.
\end{proposition}

\begin{proof}
We first prove that $\underline{S}_r$ is an $L^2$-regular storage function. Clearly $\underline{S}_r(x_0)\geq 0$ for all $x_0\in X$. Also, for $x_0=0$ we can select $\bu:=0\in \dW_c^{-1}(\{0\})$, obtaining that $\underline{S}_r(0)\leq \| D_{\widetilde\Tfrak_\Sigma} 0\|^2=0$. Hence $\underline{S}_r(0)=0$.

Next we prove that $\underline{S}_r$ satisfies the energy inequality \eqref{eq:storfndef3a}. To this end, fix a system trajectory
$(\widetilde{\bu},\widetilde{\bx},\widetilde{\by})$ of $\Sigma$ over $\R^+$ and a $t>0$. If $\widetilde \bx(0)\not\in\range{\dW_c}$ then
$\underline S_r(\widetilde \bx(0))=\inf\emptyset=\infty$ and \eqref{eq:storfndef3a} holds; otherwise let $\bu\in \dW_c^{-1}(\{\widetilde \bx(0)\})\subset L^{2-}_U$.
Then define
\begin{equation}\label{eq:ucircdef}
\bu^\circ:=\pi_-\tau^t(\bu+\widetilde{\bu})= \tau^t(\bu + \pi_{[0,t]}\widetilde{\bu})\in L^{2-}_U,
\end{equation}
and note that
\begin{equation}\label{eq:ucircnormid}
\|\bu^\circ\|^2_{L^{2-}_U}=\|\bu\|^2_{L^{2-}_U} + \|\pi_{[0,t]}\widetilde{\bu}\|^2_{L^{2+}_U}.
\end{equation}
We claim that
\begin{equation}\label{eq:ucirctoprove}
\mbox{(1) }\  \bu^\circ \in \dW_c^{-1}(\{\widetilde{\bx}(t)\}) \qquad \mbox{and} \qquad \mbox{(2) }\
\widetilde\Tfrak_\Sigma \bu^\circ =\tau^t (\widetilde\Tfrak_\Sigma \bu + \pi_{[0,t]}\widetilde{\by}).
\end{equation}
For claim (1), note that item (3) of Proposition \ref{P:Wc} implies that $\tau^t \pi_{[0,t]}\widetilde{\bu}\in L^{2-}_{\ell,U}$ is in $\dom{\dW_c}$ and
\[
\dW_c \tau^t \pi_{[0,t]}\widetilde{\bu}= \Bfrak \tau^t \pi_{[0,t]}\widetilde{\bu} = \Bfrak \pi_- \tau^t \widetilde{\bu}= \Bfrak^t \widetilde{\bu}.
\]
Also, item (4) of Proposition \ref{P:Wc} yields that $\tau^t \bu$ is in $\dom{\dW_c}$ and $\dW_c \tau^t \bu= \Afrak^t \dW_c \bu= \Afrak^t \widetilde{\bx}(0)$.
Therefore we have that $\bu^\circ \in \dom{\dW_c}$ and
\[
\dW_c \bu^\circ = \dW_c \tau^t \bu + \dW_c \tau^t \pi_{[0,t]}\widetilde{\bu}  = \Afrak^t \widetilde{\bx}(0) + \Bfrak^t \widetilde{\bu} = \widetilde{\bx}(t),
\]
using \eqref{eq:trajonRplus} in the last identity. Next we prove claim (2). By item (1) of Theorem \ref{thm:hankel} and \eqref{LfrakDec},
\begin{align*}
\pi_- L_\Sigma\tau^t
&=  \pi_- \tau^t L_\Sigma = \tau^t \pi_{(-\infty,t]}  L_\Sigma = \tau^t (\widetilde{\Tfrak}_\Sigma \pi_- + \pi_{[0,t]} \Hfrak_\Sigma \pi_- + \pi_{[0,t]} \Tfrak_\Sigma \pi_+).
\end{align*}
Therefore, from \eqref{eq:ucircdef}, we get
\begin{align*}
\widetilde\Tfrak_\Sigma \bu^\circ &
= \pi_- L_\Sigma \bu^\circ
= \pi_- L_\Sigma \tau^t (\bu + \pi_{[0,t]}\widetilde{\bu})
=  \tau^t (\widetilde{\Tfrak}_\Sigma \bu +  \pi_{[0,t]} \Hfrak_\Sigma \bu +  \pi_{[0,t]} \Tfrak_\Sigma \pi_{[0,t]}\widetilde{\bu}),
\end{align*}
and furthermore, by \eqref{Hfact2},
\[
\pi_{[0,t]} \Hfrak_\Sigma \bu = \pi_{[0,t]} \dW_o \dW_c  \bu  =\pi_{[0,t]} \dW_o \widetilde{\bx}(0) = \pi_{[0,t]} \Cfrak \widetilde{\bx}(0) = \Cfrak^t \widetilde{\bx}(0).
\]
On the other hand, using item (1) of Theorem \ref{thm:hankel} and causality, we obtain
\[
 \pi_{[0,t]} \Tfrak_\Sigma \pi_{[0,t]}\widetilde{\bu}=  \pi_{[0,t]} \Dfrak \pi_{[0,t]}\widetilde{\bu} = \Dfrak^t \widetilde{\bu}.
\]
Combining the above computations we find that
\begin{align*}
\widetilde\Tfrak_\Sigma \bu^\circ
&= \tau^t (\widetilde{\Tfrak}_\Sigma \bu + \Cfrak^t \widetilde{\bx}(0) + \Dfrak^t \widetilde{\bu})
=\tau^t (\widetilde{\Tfrak}_\Sigma \bu + \pi_{[0,t]} \widetilde{\by}),
\end{align*}
again using \eqref{eq:trajonRplus} in the last step. This proves claim (2).

Claim (2) implies that $\|\widetilde\Tfrak_\Sigma \bu^\circ\|_{L^{2-}_Y}^2 =\|\widetilde\Tfrak_\Sigma \bu\|_{L^{2-}_Y}^2
+\| \pi_{[0,t]}\widetilde{\by}\|^2_{L^{2+}_Y}$. Combining this with \eqref{eq:ucircnormid}, we find that
\[
\|\bu^\circ\|_{L^{2-}_U}^2-\|\widetilde\Tfrak_\Sigma \bu^\circ\|_{L^{2-}_Y}^2=\|\bu\|_{L^{2-}_U}^2 - \|\widetilde\Tfrak_\Sigma \bu\|_{L^{2-}_Y}^2 + \|\pi_{[0,t]}\widetilde{\bu}\|_{L^{2+}_U}^2 - \| \pi_{[0,t]}\widetilde{\by}\|^2_{L^{2+}_Y}.
\]
By claim (1) in \eqref{eq:ucirctoprove}, $\pi_-\tau^t\big(\dW_c^{-1}(\widetilde \bx(0))+\widetilde \bu)\subset \dW_c^{-1}(\widetilde \bx(t))$, and so we get that
\begin{align*}
&\inf_{\bu^\circ \in \dW_c^{-1}(\{\widetilde \bx(t)\})} \|\bu^\circ\|_{L^{2-}_U}^2-\|\widetilde\Tfrak_\Sigma \bu^\circ\|_{L^{2-}_Y}^2\\
&\qquad\qquad \leq \inf_{\bu \in \dW_c^{-1}(\{\widetilde \bx(0)\})}  \|\bu\|_{L^{2-}_U}^2 - \|\widetilde\Tfrak_\Sigma \bu\|_{L^{2-}_Y}^2
+ \|\pi_{[0,t]}\widetilde{\bu}\|_{L^{2+}_U}^2 - \| \pi_{[0,t]}\widetilde{\by}\|^2_{L^{2+}_Y}.
\end{align*}
This shows that $\underline{S}_r$ satisfies the energy inequality \eqref{eq:storfndef3a}, and hence it is a storage function. We already established that $S_a$ is a storage function. 

The boundedness of $\underline S_r$ on $\range{\bW_c}$ follows from Corollary \ref{C:bounds} below, and then $S_a$ is finite on $\range{\dW_c}$, since \eqref{eq:StorFunOrder} holds with $S=\underline S_r$. This completes the proof that $\underline S_r$ is $L^2$-regular.
\end{proof}

\begin{corollary}\label{C:bounds}
Assume that the well-posed system $\Sigma$ has a transfer function $\widehat\Dfrak\in\cS_{U,Y}$ and that $\dW_c^\star$ is densely defined. Then for all $x_0 \in X$ we have
\[
\| \dW_o x_0\|_{L^{2+}_Y}^2 \leq S_a(x_0)\leq \underline{S}_r(x_0)\leq \inf_{\bu\in \dW_c^{-1}(\{x_0\})}\|\bu\|^2_{L^{2-}_U},
\]
with $\| \dW_o x_0\|_{L^{2+}_Y}^2$ to be interpreted as $\infty$ in case $x_0\not\in\dom{\dW_o}$. Moreover, $\underline{S}_r(x_0)<\infty$ precisely when $x_0\in \range{\dW_c}$.
\end{corollary}

\begin{proof}
The first inequality is obtained by selecting $\bu=0$ for the input signals in the supremum in \eqref{Sa}. The second inequality follows from \eqref{eq:StorFunOrder}, using that $\underline{S}_r$ is a storage function for $\Sigma$ by Proposition \ref{P:L2reg}. The final inequality follows from the definition of $\underline{S}_r$ in \eqref{SrMod} and the fact that $D_{\widetilde\Tfrak_\Sigma}$ is contractive. If $x_0\not\in \range{\dW_c}$, then the infimum in \eqref{SrMod} is taken over an empty set, leading to $\underline{S}_r(x_0)=\infty$.
\end{proof}

We next establish that the storage functions $S_a$ and $\underline S_r$ are in fact quadratic.

\section{Quadratic descriptions of $S_a$ and $\underline S_r$}\label{sec:QuadStorage}

In the sequel, we will need the concept of a \emph{core} for a closed operator, which we recall here from \cite[p.\ 256]{RS80}: the set $D\subset\dom T$ is a core for the closed operator $T$ if the operator closure of $T|_D=T$ equals $T$, or in words, a closed operator is uniquely determined by its restriction to a core.

In case $\Sigma$ is a well-posed system whose transfer function $\widehat\Dfrak\in\cS_{U,Y}$, then the $L^2$-transfer map $L_\Sigma$ in \eqref{L2inout} is contractive. Hence, with respect to the decomposition in \eqref{LfrakDec}, we have
\begin{equation}\label{DefecLSi}
\begin{aligned}
I-L_\Sigma L_\Sigma^* &= \begin{bmatrix}
D_{\widetilde
\Tfrak_\Sigma^{*}}^{2} & - \widetilde \Tfrak_\Sigma
\Hfrak_\Sigma^{*} \\
 - \Hfrak_\Sigma \widetilde \Tfrak_\Sigma^{*} &
 D_{\Tfrak_\Sigma^{*}}^{2} - \Hfrak_\Sigma \Hfrak_\Sigma^{*}
\end{bmatrix}\succeq 0;\\
I-L_\Sigma^* L_\Sigma
&= \begin{bmatrix}
D_{\widetilde
  \Tfrak_\Sigma}^{2} - \Hfrak_\Sigma^{*} \Hfrak_\Sigma &
-\Hfrak_\Sigma^*\Tfrak_\Sigma\\
  -\Tfrak_\Sigma^* \Hfrak_\Sigma &  D_{\Tfrak_\Sigma}^{2}
\end{bmatrix}\succeq 0.
\end{aligned}
\end{equation}
Since $L_\Sigma$ is a contraction, so are $\Tfrak_\Sigma$, $\Tfrak_\Sigma^{*}$, $\widetilde \Tfrak_\Sigma$ and $\widetilde\Tfrak_\Sigma^{*}$, and hence their defect operators $D_{\Tfrak_\Sigma}$, $D_{\Tfrak_\Sigma^{*}}$, $D_{\widetilde \Tfrak_\Sigma}$ and $D_{\widetilde\Tfrak_\Sigma^{*}}$ are well defined. The inequalities in \eqref{DefecLSi} imply in particular that
$$
D_{\Tfrak_\Sigma^{*}}^{2} \succeq \Hfrak_\Sigma \Hfrak_\Sigma^{*} \quad \mbox{and}\quad
D_{\widetilde  \Tfrak_\Sigma}^{2} \succeq \Hfrak_\Sigma^{*} \Hfrak_\Sigma.
$$
Assuming, in addition, that $\Sigma$ is minimal, $\range{\dW_c}$ and $\range{\dW_o^*}$ are dense in $X$, by Corollary \ref{C:MinExp} and items (3)
of Propositions \ref{P:Wc} and \ref{P:Wo}, respectively, so that the factorizations of item (4) in Theorem \ref{thm:hankel} apply:
$$
\Hfrak_\Sigma\big|_{\dom{\dW_c }}=\dW_o \dW_c
\quad \mbox{and}\quad
\Hfrak_\Sigma^*\big|_{\dom{\dW_o^* }}=\dW_c^\bigstar \dW_o^*.
$$
The following lemma follows from Lemma \ref{L:X1X2} in Appendix \ref{sec:OpOpt} below, combined with \eqref{LfrakDec}, \eqref{Lop}, \eqref{Hfact2} and \eqref{HfactA}:

\begin{lemma}\label{L:XaXb}
Assume that the minimal well-posed system $\Sigma$ has transfer function in $\cS_{U,Y}$. Then:
\begin{enumerate}
\item[(1)] There exists a unique closable operator ${\bf X}_a$ with domain $\range{\dW_c}\subset X$, with range contained in $\Ker{D_{\Tfrak^*_{\Sigma}}}^\perp$,  and which satisfies the factorization
\begin{equation}\label{eq:WoFact}
\dW_o|_{\range{\dW_c}} = D_{\Tfrak^*_{\Sigma}} {\bf X}_a.
\end{equation}
Moreover, $\range{\dW_c}$ is a core for the closure $\overline {\bf X}_a$ of ${\bf X}_a$, and this closure is injective with range contained in $\Ker{D_{\Tfrak^*_{\Sigma}}}^\perp$.

\item[(2)] There exists a unique closable operator ${\bf X}_r$ with domain $\range{\dW_o^*}\subset X$, range contained in $\Ker{D_{\widetilde{\Tfrak}_{\Sigma}}}^\perp$, that satisfies the factorization
\begin{equation}\label{eq:WcFact}
\dW_c^*|_{\range{\dW_o^*}} = D_{\widetilde{\Tfrak}_{\Sigma}} {\bf X}_r.
\end{equation}
The range of $\dW_o^*$ is a core for the injective closure $\overline {\bf X}_r$ of ${\bf X}_r$ and $\range{\overline\bX_r}\perp\Ker{D_{\Tfrak^*_{\Sigma}}}$.

\end{enumerate}
\end{lemma}

Next we introduce operators $H_a$ and $H_r$, which give rise to the quadratic storage functions $S_{H_a}(x) =
\langle H_a x, x \rangle$ and $S_{H_r}(x) = \langle H_r x, x \rangle$ which are equal to the
available storage function $S_a(x)$ and the $L^2$-regularized required supply $\underline S_r(x)$ respectively, at least for $x\in\range{\dW_c}$.
Assume that $\Sigma$ is minimal and has transfer function in $\cS_{U,Y}$, so that ${\bf X}_a$ and ${\bf X}_r$ in Lemma \ref{L:XaXb} are densely defined, closable operators
with injective closures $\overline{\bf X}_a$ and $\overline{\bf X}_r$, respectively. Then, $\overline{\bf X}_a^*\overline{\bf X}_a$ is selfadjoint with unique positive,  selfadjoint,
injective square root $|\overline{\bf X}_a|=(\overline{\bf X}_a^*\overline{\bf X}_a)^{\frac{1}{2}}$ satisfying $\dom{|\overline{\bf X}_a|}=\dom{\overline{\bf X}_a}$;
see for instance \cite[\S VIII.9]{RS80}.
We now set $H_a=\overline{\bf X}_a^*\overline{\bf X}_a$ so that $H_a^{\frac{1}{2}}=|\overline{\bf X}_a|$. Analogously, set
$|\overline{\bf X}_r|:=(\overline{\bf X}_r^*\overline{\bf X}_r)^{\frac{1}{2}}$
and $H_r:=(\overline{\bf X}_r^*\overline{\bf X}_r)^{-1}$, so that $H_r^{\frac{1}{2}}=|\overline{\bf X}_r|^{-1}$, with $\dom{H_r^{\frac{1}{2}}}=\range{|\overline {\bf X}_r|}$. Note that the operators $H_a^{\frac{1}{2}}$, $H_a^{-\frac{1}{2}}$, $H_r^{\frac{1}{2}}$ and $H_r^{-\frac{1}{2}}$ are all closed. The following theorem follows directly from Theorem \ref{T:S-S+quad} in Appendix \ref{sec:OpOpt}.

\begin{theorem}\label{T:SaUnSrQuad}
Let $\Sigma$ be a minimal well-posed system which has transfer function in $\cS_{U,Y}$. Define ${\bf X}_a$, $\overline{\bf X}_a$, ${\bf X}_r$, $\overline{\bf X}_r$ as in Lemma \ref{L:XaXb} and $H_a$ and $H_r$ as in the preceding paragraph. Then the dense subspace $\range{\dW_c}$ of $X$ is contained in the domains of $H_a^{\frac{1}{2}}$ and $H_r^{\frac{1}{2}}$, and $S_a$ and $\underline{S}_r$ satisfy
\begin{equation}\label{SaUnSrQuad}
\begin{aligned}
S_a(x_0) & = \||\overline{\bf X}_a|x_0\|^2=\|H_a^{\frac{1}{2}}x_0\|^2,\quad x_0\in \range{\dW_c},  \\
\underline{S}_r(x_0) & = \||\overline{\bf X}_r|^{-1} x_0\|^2=\|H_r^{\frac{1}{2}}x_0\|^2,\quad x_0\in \range{\dW_c}.  
\end{aligned}
\end{equation}
Moreover, $\range{\dW_c}$ is a core for $H_a^{\frac{1}{2}}$  and $\range{\dW_o^*}$ is a core for $H_r^{-\frac{1}{2}}$.
\end{theorem}

Note that Theorem \ref{T:SaUnSrQuad} is not strong enough to justify the conclusion that $S_a$ and $\underline{S}_r$ are quadratic storage functions, since the identities in \eqref{SaUnSrQuad} only hold on $\range{\dW_c}$ which might be strictly contained in the domains of $H_a^{\frac{1}{2}}$ and $H_r^{\frac{1}{2}}$, respectively. Later on, in Theorem \ref{T:HaHr-KYPsols} below, we will show that $H_a$ and $H_r$ are spatial solutions to the KYP inequality of $\Sigma$ under the assumptions of Theorem \ref{T:SaUnSrQuad}, so that $H_a$ and $H_r$ induce quadratic storage functions by Theorem \ref{thm:stdlemma}. These may differ from $S_a$ and $\underline S_r$ outside $\range{\bW_c}$. However, if the initial state of a trajectory $(\bu,\bx,\by)$ of $\Sigma$ on $\rplus$ satisfies $\bx(0)\in\range{\bW_c}$, then $\bx(t)\in\range{\bW_c}$ for all $t\geq0$, by items (3) and (4) of Proposition \ref{P:Wc}. For such state trajectories, $S_a$ and $\underline S_r$ coincide with $S_{H_a}$ and $S_{H_r}$, respectively.

It is of interest to work out the corresponding results for the causal dual system $\Sigma^d$ explicitly in terms of objects related to the original system $\Sigma$. Using \eqref{eq:Lsig-d} and \eqref{LfrakDec}, one gets that the Laurent operator $L_{\Sigma^d}$ for $\Sigma^d$ is
\begin{align*}
 L_{\Sigma^d} &  = \begin{bmatrix} \widetilde  \Tfrak_{\Sigma^d}  & 0 \\ \Hfrak_{\Sigma^d} & \Tfrak_{\Sigma^d}  \end{bmatrix} :=
\begin{bmatrix} 0 & \ya \\ \ya & 0 \end{bmatrix}  \begin{bmatrix} \widetilde \Tfrak_\Sigma & 0 \\ \Hfrak_\Sigma & \Tfrak_\Sigma \end{bmatrix}^*
\begin{bmatrix} 0 & \ya \\ \ya & 0 \end{bmatrix}  \\
& = \begin{bmatrix} \ya \Tfrak^*_\Sigma \ya & 0 \\ \ya \Hfrak_\Sigma^* \ya & \ya \widetilde \Tfrak_\Sigma^* \ya \end{bmatrix}  \colon
\begin{bmatrix} L^{2-}_Y \\ L^{2+}_Y \end{bmatrix} \to \begin{bmatrix} L^{2-}_U \\ L^{2+}_U \end{bmatrix}.
\end{align*}
Furthermore, from \eqref{DualObsCon} we see that the dual $L^2$-output and dual $L^2$-input operators are given by
\begin{equation}   \label{DualObsCon'}
  \bW_o^d = \ya \bW^*_c,  \quad \bW_c^d = \bW_o^* \ya.
\end{equation}
Apply Lemma \ref{L:XaXb} with $\Sigma^d$ in place of $\Sigma$ to see that the operator $X_a^d$ obtained from item (1) is determined by
\begin{equation}\label{eq:WodX}
\bW^d_o|_{\range{ \bW_c^d}}
= D_{\Tfrak_{\Sigma^d}^*} \bX_a^d
=  D_{\scriptsize{\ya} \widetilde \Tfrak_\Sigma \scriptsize{\ya}} \bX_a^d
=  \ya D_{\widetilde \Tfrak_\Sigma} \ya \bX_a^d,
\end{equation}
where the last equality can be verified by simply squaring $\ya D_{\widetilde \Tfrak_\Sigma} \ya\geq0$.

On the other hand, by \eqref{DualObsCon'} and Lemma \ref{L:XaXb} applied to $\Sigma$ we have
$$
\bW^d_o|_{\range{ \bW_c^d}}
=  \ya \bW_c^*|_{\range{ \bW^*_o \scriptsize{\ya}}}
=  \ya \bW_c^*|_{\range{ \bW^*_o} }
= \ya D_{\widetilde \Tfrak_\Sigma}\bX_r.
$$
By combining these last two expressions we get that $\range{\ya\bX_a^d-\bX_r}\subset\Ker{D_{\widetilde\Tfrak_\Sigma}}$, and since $\range{\bX_r}$ is also perpendicular to this kernel, we may conclude that
\begin{equation*} 
\overline{\bX}_a^d = \ya \overline{\bX}_r
\end{equation*}
once we use \eqref{eq:WodX} to observe that
$$
	\range{\ya\bX_a^d}\subset
	\ya\Ker{D_{\Tfrak_{\Sigma^d}^*}}^\perp=
	\Ker{D_{\widetilde\Tfrak_\Sigma}}^\perp.
$$

By duality, we immediately get $\ya \overline{\bX}_r^d = \overline{\bX}_a$, and then the operators $H_a^d$ and $H_r^d$ associated with the dual system $\Sigma^d$, as in the paragraph preceding Theorem \ref{T:SaUnSrQuad}, are related to $H_a$ and $H_r$ via
\begin{equation}\label{HaHr-dual}
H_a^d=H_r^{-1} \quad\mbox{and}\quad H_r^d=H_a^{-1}.
\end{equation}
Therefore, Theorem \ref{T:SaUnSrQuad} applied to the causal dual system leads us to the following formulas for the available storage and $L^2$-regularized required supply for the causal dual system $\Sigma^d$.

 \begin{theorem} \label{T:SaUnSrQuad-dual}
 Let $\Sigma$ be a minimal well-posed system which has transfer function in $\cS_{U,Y}$. Define ${\bf X}_a$, $\overline{\bf X}_a$, ${\bf X}_r$, $\overline{\bf X}_r$ as in Lemma \ref{L:XaXb} and $H_a$ and $H_r$ as in Theorem \ref{T:SaUnSrQuad}. Then $\range{ \bW_o^* }$ is contained in the domains of $H_a^{-\frac{1}{2}}$ and $H_r^{-\frac{1}{2}}$, and the available storage $S^d_a$ and the $L^2$-regularized required supply $\underline{S}_r^d$
 for the causal dual system $\Sigma^d$ are given by
 \begin{align*}
 & S_a^d(x_0) = \| |\overline{\bX}_a^d| x_0 \|^2 = \| |\overline{\bX}_r|  x_0 \|^2 = \| H_r^{-\frac{1}{2}} x_0 \|^2 \text{ for } x_0 \in \range{ \bW_o^* }, \label{SaUnSrQuad-dual} \\
& \underline{S}_r^d(x_0) =  \| |\overline{\bX}_r^d|^{-1} x_0 \|^2 = \| |\overline{\bX}_a|^{-1}  x_0 \|^2 = \| H_a^{-\frac{1}{2}} x_0 \|^2 \text{ for } x_0 \in \range{ \bW_o^* }.
 \notag
\end{align*}
\end{theorem}

Using the above results, we will next show that the solutions $H_a$ and $H_r$  to the spatial KYP-inequality \eqref{eq:KYP} associated with $\Sigma$ are minimal and maximal spatial solutions respectively for certain subclasses of spatial solutions.

\begin{theorem}\label{T:HaHr-KYPsols}
Let $\Sigma$ be a minimal well-posed system which has transfer function in $\cS_{U,Y}$. Then the operators $H_a$ and $H_r$ defined above are spatial solutions to the KYP-inequality \eqref{eq:KYP}. Moreover, for all spatial solutions $H$ to \eqref{eq:KYP} the following hold:
\begin{enumerate}
\item[(1)] If $\range{\dW_c}$ is a core for $H^\frac{1}{2}$, then $H_a \preceq H$;

\item[(2)] If $\range{\dW_o^*}$ is a core for $H^{-\frac{1}{2}}$, then $H \preceq H_r$.

\end{enumerate}
\end{theorem}

\begin{proof}
We first prove the claims regarding $H_a$. By items (3) and (4) of Proposition \ref{P:Wc} it follows that
\[
\range{\Bfrak}\subset \range{\dW_c} \quad \mbox{and}  \quad \Afrak^t\, \range{\dW_c} \subset \range{\dW_c},\ \ t\in\R^+.
\]
In particular, Theorem \ref{T:SaUnSrQuad} yields $\range{\Bfrak}\subset \dom{H_a^\frac{1}{2}}$, implying $\Bfrak^t L^2([0,t];U)\subset \dom{H_a^{\frac{1}{2}}}$.
 Moreover, the fact that $S_a(x)=S_{H_a}(x)$ for a
\begin{equation}\label{eq:KYPrest}
	\left\|\bbm{H_a^{\frac{1}{2}}&0\\0&I}
		\bbm{\Afrak^t&\Bfrak^t\\\Cfrak^t&\Dfrak^t}\bbm{x\\\bu}\right\|
	\leq\left\|\bbm{H_a^{\frac{1}{2}}&0\\0&I}\bbm{x\\\bu}\right\|,\ \bbm{x\\\bu}\in\bbm{\range{\dW_c}\\L^2([0,t];U)}.
\end{equation}
Squaring on both sides and restricting to $\bu=0$, we get
\[
\| H_a^{\frac{1}{2}} \Afrak^t x\|^2 \leq \| H_a^{\frac{1}{2}} \Afrak^t x\|^2  + \|\Cfrak^t x\|^2 \leq \| H_a^{\frac{1}{2}} x\|^2, \quad x\in \range{\dW_c},
\]
hence
\begin{equation}\label{Ineq}
\| H_a^{\frac{1}{2}} \Afrak^t x\| \leq  \| H_a^{\frac{1}{2}} x\|, \quad x\in \range{\dW_c}.
\end{equation}

Now take $\widetilde{x}\in \dom{H_a^{\frac{1}{2}}}$ and fix $t\geq 0$. Since $\range{\dW_c}$ is a core for $H_a^{\frac{1}{2}}$ by Theorem \ref{T:SaUnSrQuad}, there exists a sequence
$x_n\in\range{\dW_c}$, $n\in\Z_+$, such that $x_n \to \widetilde{x}$ and $H_a^{\frac{1}{2}}x_n \to H_a^{\frac{1}{2}}\widetilde{x}$ in $X$. In particular, $H_a^{\frac{1}{2}}x_n$ is a Cauchy sequence. Applying \eqref{Ineq} with $x=x_n-x_m$, we obtain that
\[
\|H_a^{\frac{1}{2}} \Afrak^t x_n - H_a^{\frac{1}{2}} \Afrak^t x_m\|\leq \|H_a^{\frac{1}{2}} x_n - H_a^{\frac{1}{2}} x_m\| \to 0\quad \mbox{as $n,m\to 0$.}
\]
Hence $H_a^{\frac{1}{2}} \Afrak^t x_n$ is also a Cauchy sequence, thus convergent in $X$. Also, $\Afrak^t x_n$ converges to $\Afrak^t \widetilde{x}$, because $\Afrak^t$ is bounded. Since $H_a^{\frac{1}{2}}$ is closed, it follows that $\Afrak^t \widetilde{x}$ is in $\dom{H_a^{\frac{1}{2}}}$ and $H_a^{\frac{1}{2}}\Afrak^t \widetilde{x}=\lim_{n\to \infty} H_a^{\frac{1}{2}}\Afrak^t x_n$. In particular, we proved that $\Afrak^t\, \dom{H_a^{\frac{1}{2}}} \subset \dom{H_a^{\frac{1}{2}}}$. We have now proved that \eqref{eq:HhalfDomCond} holds. The fact that the spatial KYP inequality \eqref{eq:KYP} holds on $\dom{H_a^{\frac{1}{2}}} \oplus L^2([0,t];U)$ now also follows easily from \eqref{eq:KYPrest} and the fact that for $\widetilde{x}\in \dom{H_a^{\frac{1}{2}}}$ and $x_n\in \range{\dW_c}$ as above we have $H_a^{\frac{1}{2}} x_n \to H_a^{\frac{1}{2}} \widetilde{x}$, $H_a^{\frac{1}{2}} \Afrak^t x_n \to \Afrak^t H_a^{\frac{1}{2}} \widetilde{x}$ and $\Cfrak^tx_n\to\Cfrak^t\widetilde x$.

Assume next that $H$ is any solution to the spatial KYP-inequality \eqref{eq:KYP} with the property that $\range{\dW_c}$ is a core for $H^\frac{1}{2}$. By Proposition \ref{P:StorageKYP} and Theorem \ref{thm:schurimpliesstorage}, we have
\begin{equation}\label{Ineq2}
\|H_a^{\frac{1}{2}}x\|^2=S_a(x)\leq S_H(x)=\|H^{\frac{1}{2}}x\|^2,\quad x\in\range{\dW_c}.
\end{equation}
Take $\widetilde{x}\in \dom{H^{\frac{1}{2}}}$ arbitrarily, and let $x_n\in \range{\dW_c}$, $n\in\Z_+$, so that $x_n\to \widetilde{x}$ and $H^{\frac{1}{2}}x_n \to H^{\frac{1}{2}}\widetilde{x}$; such a sequence
exists since $\range{\dW_c}$ is a core for $H^{\frac{1}{2}}$, by assumption. Reasoning as above, the sequence $H^{\frac{1}{2}}x_n$, $n\in\Z_+$, is a Cauchy sequence, and the
inequality \eqref{Ineq2} implies that $H_a^{\frac{1}{2}}x_n$, $n\in\Z_+$, is a Cauchy sequence as well. The closedness of $H_a^{\frac{1}{2}}$ then implies that
$\widetilde{x}\in \dom{H_a^{\frac{1}{2}}}$ and $H_a^{\frac{1}{2}}x_n \to H_a^{\frac{1}{2}}\widetilde{x}$. Consequently, $\dom{H^{\frac{1}{2}}}\subset \dom{H_a^{\frac{1}{2}}}$ and the inequality \eqref{Ineq2}
extends to all $x \in \dom{H^{\frac{1}{2}}}$, which proves that $H_a\preceq H$, and the proof of statement (1) is complete.

The proof of statement (2)  requires drawing on results for the causal dual system $\Sigma^d$ as well as results for $\Sigma$ itself.
We note from \eqref{DualObsCon'} that $\range{ \bW_o^*} = \range{ \bW_c^d}$.  Note also by Proposition \ref{P:DualKYPsols} that $H$ is a solution of the
spatial KYP-inequality \eqref{eq:KYP} for $\Sigma$ if and only if $H^{-1}$ is a solution of the spatial KYP-inequality \eqref{KYP-d} for $\Sigma^d$.
Thus $\range{\bW_o^*}$ being a core for $H^{-\frac{1}{2}}$ where $H$ solves the KYP-inequality \eqref{eq:KYP} for $\Sigma$  is the same as $\range{\bW_c^d}$ being a core for
$(H^{-1})^{\frac{1}{2}}$  where $H^{-1}$ solves the KYP-inequality \eqref{KYP-d} for $\Sigma^d$.
We conclude that the hypothesis for statement (2) in the theorem is the same as the hypothesis for statement (1), but applied to $\Sigma^d$ rather than to $\Sigma$.
Hence, if we assume the hypothesis for statement (2), we can use the implication in statement (1) already proved to conclude that $H^d_a \preceq H^{-1}$,
where \eqref{HaHr-dual} gives $H_a^d = H_r^{-1}$, and thus we have
     $H_r^{-1} \preceq H^{-1}$.
Now \cite[Proposition 3.4]{AKP05}  
gives us the desired inequality $H \preceq H_r$.
\end{proof}

\begin{remark}\label{R:ASconnect}
Theorem \ref{T:HaHr-KYPsols} states that $H_a$ and $H_r$ are both positive definite spatial solutions to the KYP inequality \eqref{eq:KYP}, provided $\Sigma$ is a minimal well-posed system which has transfer function in $\cS_{U,Y}$, and they are the minimal and maximal spatial solutions at least within a certain subset of the collection of spatial solutions. To be precise, if $\cG\cK_\Sigma$ denotes the collection of all positive definite spatial solutions to \eqref{eq:KYP}, then $H_a$ is the minimal element in
\[
\widetilde{\cG\cK}_{\Sigma,\textup{core}}:=\{H\in \cG\cK_\Sigma \mid \mbox{$\range{\dW_c}$ is a core for $H^\frac{1}{2}$} \}
\]
while $H_r$ is the maximal element in
\[
\widehat{\cG\cK}_{\Sigma,\textup{core}}:=\{H\in \cG\cK_\Sigma \mid \mbox{$\range{\dW_o^*}$ is a core for $H^\frac{1}{2}$} \}.
\]
That we cannot claim that $H_a$ is the minimal element in $\cG\cK_\Sigma$, despite the fact that $S_a$ is the minimal storage function for $\Sigma$, is because in
general we only managed to prove that $S_a$ and $S_{H_a}$ coincide on $\range{\dW_c}$.

In \cite{ArSt07} another analysis of the spatial solutions to the KYP for well-posed linear systems is obtained, with somewhat different extremality results.
This may result from
the fact that the analysis conducted in \cite{ArSt07} is done at the level of system nodes, and that the requirements there are slightly different. More precisely, in \cite{ArSt07} it is not assumed that the well-posed system
$\Sigma$ is minimal, but rather,
for spatial solutions $H$ it is assumed, in addition, that the well-posed system $\Sigma_H$ obtained by applying $H^\frac{1}{2}$ as a pseudo-similarity
is minimal, and in that
case the minimal and maximal solutions are those that correspond to the so-called optimal and $*$-optimal solutions. Note that because of the applied
pseudo-similarity, the KYP-inequality for
$\Sigma_H$ always has a bounded and boundedly invertible solution, namely $1_X$. Why there are no core restrictions in \cite{ArSt07}, which correspond to those that we have in the present paper, is unclear to us at this stage.
\end{remark}

If in addition to the minimality and a Schur class transfer function we also have $L^2$-controllability or $L^2$-observability, more can be said about the operators $H_a$ and $H_r$. 

\begin{corollary}\label{C:L2minImplics}
Let $\Sigma$ be a minimal well-posed system which has transfer function in $\cS_{U,Y}$. Then the following holds:
\begin{enumerate}
  \item[(1)] If $\Sigma$ is $L^2$-controllable, then $H_a$ and $H_r$ are bounded.
  \item[(2)] If $\Sigma$ is $L^2$-observable, then $H_a^{-1}$ and $H_r^{-1}$ are bounded.
\end{enumerate}
\end{corollary}

\begin{proof}
Assume that $\Sigma$ is $L^2$-controllable, that is, $\dom{\bW_c^\bigstar}$ is dense and $\range{\bW_c}=X$. Since $X=\range{\bW_c}$ is contained in the domains of $H_a^\frac{1}{2}$ and $H_r^\frac{1}{2}$ by Theorem \ref{T:SaUnSrQuad}, it follows that $H_a^\frac{1}{2}$ and $H_r^\frac{1}{2}$ are bounded by the closed graph theorem; hence $H_a$ and $H_r$ are also bounded. Statement 2 follows by applying statement 1 to $\Sigma^d$.
\end{proof}

\section{Proofs of the bounded real lemmas}\label{sec:proofs}

In this section we prove the bounded real lemmas posed in the introduction. We start with a proof of Theorem \ref{thm:stdlemma}.

\begin{proof}[Proof of Theorem \ref{thm:stdlemma}]
The implication (5) $\Rightarrow$ (4) is trivial and many of the other implications have been proved in the preceding sections: that
(4) $\Rightarrow$ (1) follows from Proposition \ref{prop:storageimpliesschur}; the equivalence (2) $\Leftrightarrow$ (5)
 follows from Proposition \ref{P:StorageKYP}, together with the statement that the same $H$ works in both items; Theorem \ref{T:HaHr-KYPsols} shows that
(1) $\Rightarrow$ (2).  Hence it follows that (1) $\Leftrightarrow$ (2) $\Leftrightarrow$ (4) $\Leftrightarrow$ (5).

Next, we show that (3) $\Rightarrow$ (5).  Assume that item (3) holds, say that $\Gamma:X\supset\dom{\Gamma}\to X^\circ$ implements a
pseudo-similarity from $\Sigma=\sbm{\Afrak&\Bfrak\\\Cfrak&\Dfrak}$ to a passive well-posed system $\Sigma^\circ=\sbm{\Afrak^\circ&\Bfrak^\circ\\\Cfrak^\circ&\Dfrak^\circ}$ with state space $X^\circ$. In that case $H:=\Gamma^*\Gamma$ and its positive semidefinite square root are well-defined positive definite operators, and $\dom{H^{\frac{1}{2}}}=\dom{\Gamma}$ by \cite[\S VIII.9]{RS80}. We next prove that $S_H$ in \eqref{eq:quadstorefun} is a quadratic storage function for $\Sigma$. For this, pick $z_0\in\dom{H^{\frac{1}{2}}}$ arbitrarily and let $(\bu,\bz,\by)$ be a trajectory of $\Sigma$ on $\rplus$ with initial state $\bx(0)=z_0$. Setting $\bx(t):=\Gamma \bz(t)$, $t\geq0$, and $x_0:=\Gamma z_0$ we get that $(\bu,\bx,\by)$ is a trajectory of $\Sigma^\circ$ on $\rplus$ with initial state $x_0$, since
$$
	\bx(t)= \Gamma \bz(t)=\Gamma(\Afrak^tz_0+\Bfrak^t\bu)=
	\Afrak^{\circ t}\Gamma \bz(0)+\Bfrak^{\circ t}\bu,
$$
and
$$
	\Cfrak^\circ \Gamma \bz(0)+\Dfrak \bu=\Cfrak x_0+\Dfrak \bu=\by.
$$
By passivity, every trajectory $(\bu,\bx,\by)$ of $\Sigma^\circ$ on $\rplus$ satisfies \eqref{eq:storfndef3a} with $S(x_0)=\|x_0\|^2_{X^\circ}$, and by considering $\bx(t)=\Gamma \bz(t)$, we see that also \eqref{eq:storfndef3a} holds with $S$ replaced by $S_H$ and $\bx$ replaced by $\bz$. If $z_0\not\in\dom{H^{\frac{1}{2}}}$ then $S_H(z_0)=\infty$, and the modification of \eqref{eq:storfndef3a} is still true. We have proved that $S_H$ is a quadratic storage function for $\Sigma$, where $H=\Gamma^*\Gamma$.

Finally, we prove that (1) $\Rightarrow$ (3). Assume the transfer function $\widehat{\Dfrak}$ of $\Sigma$ is in $\cS_{U,Y}$, more precisely, that it has an analytic continuation to a function in $\cS_{U,Y}$. In that case $\widehat{\Dfrak}$ coincides with the transfer function of some minimal passive well-posed system on some right half-plane, by Theorem 11.8.14 in \cite{StafBook}. Hence we have two minimal well-posed systems whose transfer functions coincide on some right half-plane, of which one is passive. Then Theorem 9.2.4 in \cite{StafBook} (see also \cite[Theorem 4.11]{ArSt07}) implies that the two systems are pseudo-similar. In particular, $\Sigma$ is pseudo-similar to a passive well-posed system.
\end{proof}

Next we turn to the proof of Theorem \ref{thm:stdlemmaL2reg}.

\begin{proof}[Proof of Theorem \ref{thm:stdlemmaL2reg}]
By Corollary \ref{C:L2min}, the $L^2$-minimality of $\Sigma$ implies that $\Sigma$ is minimal. Assume item (3) holds, i.e., $\Sigma$ is similar to a passive system. Then, in particular, $\Sigma$ is pseudo-similar to a passive system, and since
$\Sigma$ is minimal we can conclude from the implication (3) $\Rightarrow $ (1) in Theorem \ref{thm:stdlemma} that the transfer function $\widehat\Dfrak$ is in
$\cS_{U,Y}$. Hence item (1) holds.

Next we show that (2) $\Rightarrow$ (3). Assume that the operator $H$ on $X$ is a bounded, strictly positive definite solution to the KYP inequality \eqref{eq:KYPbdd}.
In that case $\Gamma:=H^{\frac{1}{2}}$ can serve as a similarity to a passive system. Indeed, for each $t\geq 0$, set
\[
\bbm{\Afrak^{\circ t}&\Bfrak^{\circ t} \\ \Cfrak^{\circ t}& \Dfrak^{\circ t}}:=
\bbm{H^\frac{1}{2} &0\\0&I} \bbm{\Afrak^t & \Bfrak^t \\ \Cfrak^t & \Dfrak^t} \bbm{H^{-\frac{1}{2}} &0\\0&I}.
\]
Then we have
\begin{equation}\label{eq:introinter3}
H^{\frac{1}{2}}\Afrak^t=\Afrak^{\circ t} H^{\frac{1}{2}},\quad H^{\frac{1}{2}}\Bfrak^t=\Bfrak^{\circ t},\quad
\Cfrak^t=\Cfrak^{\circ t} H^{\frac{1}{2}},\quad \Dfrak^t=\Dfrak^{\circ t}.
\end{equation}
Furthermore,  \eqref{eq:KYPbdd} implies that $\sbm{\Afrak^{\circ t}&\Bfrak^{\circ t} \\ \Cfrak^{\circ t}& \Dfrak^{\circ t}}$ is contractive for each $t\geq 0$.
Clearly, the relation between $\Afrak^t$ and $\Afrak^{\circ t}$ in \eqref{eq:introinter3} with $H^{\frac{1}{2}}$ bounded and boundedly invertible implies that
$\Afrak^{\circ t}$ inherits the properties of a $C_0$-semigroup from $\Afrak^t$. Next, define $\Bfrak^\circ$, $\Cfrak^\circ$ and $\Dfrak^\circ$ via the limits
in \eqref{t-back}, adding $\circ$ where appropriate. It is then easy to check that \eqref{eq:introinter3} extends to
$$
H^{\frac{1}{2}}\Afrak^t=\Afrak^{\circ t} H^{\frac{1}{2}},\quad H^{\frac{1}{2}}\Bfrak=\Bfrak^{\circ},\quad
\Cfrak=\Cfrak^{\circ} H^{\frac{1}{2}},\quad \Dfrak=\Dfrak^{\circ},
$$
and via these relations it follows that the requirements on the $C_0$-semigroup $\Afrak^\circ$ and the operators $\Bfrak^\circ$, $\Cfrak^\circ$ and
$\Dfrak^\circ$ to form a well-posed system (Definition \ref{def:WPsys}) carry over from $\Afrak$, $\Bfrak$, $\Cfrak$ and $\Dfrak$. We have proved that
$\sbm{\Afrak^{\circ}&\Bfrak^{\circ} \\ \Cfrak^{\circ}& \Dfrak^{\circ}}$ is a passive system that is similar to $\sbm{\Afrak &\Bfrak  \\ \Cfrak & \Dfrak }$ via the
similarity $\Gamma=H^{\frac{1}{2}}$; hence item (3) holds.

To establish the mutual equivalence of all  three items, it remains to prove that (1) $\Rightarrow$ (2).  Hence assume that $\widehat\Dfrak\in\cS_{U,Y}$.
Since $\Sigma$ is minimal, Theorem \ref{T:HaHr-KYPsols} gives that $H_a$ and $H_r$ are spatial solutions to the KYP-inequality \eqref{eq:KYP}. However,
the $L^2$-minimality of $\Sigma$ implies that $H_a$ and $H_r$ are bounded and boundedly invertible, by Corollary \ref{C:L2minImplics}.
Thus $H_a$ and $H_r$ are bounded, positive definite operators on $X$ with bounded inverses, and hence both are bounded and strictly positive definite.
Since $H_a$ and $H_r$ are bounded
solutions to the spatial KYP inequality \eqref{eq:KYP}, it is immediate that $H_a$ and $H_r$ also satisfy the standard KYP inequality \eqref{eq:KYPbdd}.
Hence statement (2) holds.

Next we prove that $\cplus\subset\dom{\widehat\Dfrak}$ if there is some bounded and boundedly invertible $\Gamma$ that implements the similarity from $\Sigma$ to a passive system $\Sigma^\circ$. Assume this and recall that by Proposition \ref{prop:Transfer}, $	\dom{\widehat\Dfrak}=\C_{\omega_{\Afrak}}$. Since $\Afrak^\circ$ is a contraction semigroup, as implied by passivity, we get from \eqref{eq:omegaAdef} that
$$
	\omega_{\Afrak}=\lim_{t\to\infty}\frac{\ln\|\Afrak^{t}\|}{t} \leq
	\lim_{t\to\infty}\frac{\ln\|\Gamma^{-1}\|+\ln\|\Afrak^{\circ t}\|+\ln\|\Gamma\|}{t}=\omega_{\Afrak^\circ}\leq 0.
$$

We established above that every bounded, strictly positive definite solution $H$ to the KYP inequality provides a similarity via $H^{\frac{1}{2}}$. The converse implication follows from the final statement in Theorem \ref{thm:stdlemma}.

We already noted that $H_a$ and $H_r$ are both bounded and strictly positive definite, 
and that $\Sigma$ is approximately controllable, so that $\range\Bfrak$ is dense in $X$. 
By Theorem \ref{P:StorageKYP}, every solution $H$ to the spatial KYP inequality \eqref{eq:KYP} defines a storage function $S_H$, which by Theorem \ref{thm:schurimpliesstorage} is wedged between $S_a$ and $S_r$: $S_a(x)\leq S_H(x)\leq S_r(x)$ for all $x\in X$. Moreover, combining item (3) in Proposition \ref{P:Wc} with \eqref{Sr} and \eqref{SrMod}, we get that $S_r(x)=\underline S_r(x)$ for all $x\in\range\Bfrak\subset\range{\bW_c}$. Then \eqref{SaUnSrQuad} gives
$$
	\|H_a^{\frac{1}{2}} x\|\leq \|H^{\frac{1}{2}} x\|\leq \|H_r^{\frac{1}{2}} x\|,\quad x\in \range{\Bfrak}.
$$ 
Since $\range\Bfrak$ is dense in $X$, these inequalities in fact hold on all of $X$, and we get that $H$ inherits boundedness from $H_r$, while strict positive definiteness carries over to $H$ from $H_a$. Hence every generalized solution $H$ to the spatial KYP is also a bounded, strictly positive definite solution to the standard KYP inequality \eqref{eq:KYPbdd}, and $H_a \preceq H \preceq H_r$ holds.
\end{proof}

In case the transfer function is a strict Schur class function and $\Afrak$ is exponentially stable, to obtain a bounded, strictly positive definite solution $H$ to the standard KYP inequality \eqref{eq:KYPbdd}, it suffices to have only $L^2$-controllability or $L^2$-observability:

\begin{proposition}\label{P:BRLstrict+L2}
Let $\Sigma=\sbm{\Afrak&\Bfrak\\\Cfrak&\Dfrak}$  be a minimal, exponentially stable well-posed system with transfer function $\widehat\Dfrak$ in the strict Schur class $\cS_{U,Y}^0$. Then $H_a$ and $H_r^{-1}$ are bounded and are given by
\begin{equation}\label{eq:HaHrEasy}
H_a=\dW_o^* D_{\Tfrak_\Sigma^*}^{-2}  \dW_o\quad\mbox{and}\quad
H_r^{-1}=\dW_c D_{\widetilde\Tfrak_\Sigma}^{-2}  \dW_c^*.
\end{equation}
Furthermore, $H_a^{-1}$ is bounded if and only if $\Sigma$ is $L^2$-observable and $H_r$ is bounded if and only if $\Sigma$ is $L^2$-controllable.
\end{proposition}

\begin{proof}
By Lemma \ref{lem:PasvContr}, the exponential stability guarantees that the operators $\dW_o$ and $\dW_c$ are bounded. Moreover, because
$\widehat\Dfrak\in\cS_{U,Y}^0$, $D_{\Tfrak_\Sigma^*}$ and $D_{\widetilde\Tfrak_\Sigma}$ are boundedly invertible. It follows that the operators $\bX_a$ and $\bX_r$ in Lemma \ref{L:XaXb} are given by
\[
\bX_a= D_{\Tfrak_\Sigma^*}^{-1}  \dW_o|_{\range{\dW_c}}\quad\mbox{and}\quad
\bX_r= D_{\widetilde\Tfrak_\Sigma}^{-1}  \dW_c^*|_{\range{\dW_o^*}},
\]
and hence they extend uniquely to bounded operators $\overline{\bX}_a=D_{\Tfrak_\Sigma^*}^{-1}  \dW_o$ and $\overline{\bX}_r=D_{\widetilde\Tfrak_\Sigma}^{-1}  \dW_c^*$ from $X$ into $L^{2+}_Y$ and $L^{-2}_U$, respectively. The boundedness of and formulas for $H_a=\overline{\bX}_a^*\overline{\bX}_a$ and $H_r^{-1}=\overline{\bX}_r^*\overline{\bX}_r$ now follow directly. Moreover, given the boundedness of $\dW_o$ and $\dW_c$ we have that $L^2$-observability and $L^2$-controllability are equivalent to $\dW_o$ and $\dW_c^*$ being bounded below, respectively, from which the last claim follows.
\end{proof}

Using Proposition \ref{P:BRLstrict+L2}, we can obtain explicitly the extremal KYP solutions $H_a$ and $H_r$ arising from the minimal realization for the strict Schur-class transfer function \eqref{eq:ExDstrict} which was already discussed in Example \ref{ex:counter}, thereby illustrating Proposition \ref{P:BRLstrict+L2} and item (5) of Theorem \ref{thm:stdlemmastrict}.

\begin{example}\label{ex:HaHr}
In Example \ref{ex:counter}, we considered the diagonal system $\Sigma$ with operators
$$
	\Afrak^t\phi_n=e^{-(n+1)t}\phi_n,\quad
	B\phi_n=\frac{\sqrt{n+1}}2\phi_n,\quad
	C\phi_n=2\sqrt{n+1}\phi_n,\qquad n=0,1,\ldots,
$$
leading to $\bW_o$ determined by
\begin{equation}\label{eq:WoDiagEx}
	(\bW_o\phi_n)(t)=2\sqrt{n+1}e^{-(n+1)t}\phi_n,
	\quad t\geq0,
\end{equation}
being bounded from both below and above.

In order to apply the formula for $H_a$ in \eqref{eq:HaHrEasy}, we additionally need some information on the action of the adjoint of $\Tfrak_\Sigma$ in \eqref{LfrakDec}. Combining the latter with item (1) of Theorem \ref{thm:hankel}, we get $\Tfrak_\Sigma=\Dfrak\big|_{L^{2+}_U}$, and we next compute this operator using \eqref{Dfrak0}. Because of \eqref{BfrakCfrak}, and item (3) of Proposition \ref{P:Wc},
\begin{equation}\label{eq:WcDiagEx}
	\bW_c \big(f(\cdot)\phi_n\big)= \frac12\sqrt{n+1}\int_{-\infty}^0e^{(n+1)s}f(s)\ud s\,\phi_n,
	\quad f\in L^{2-},
\end{equation}
and $\Bfrak=\bW_c\big|_{L^{2-}_{\ell,U}}$. Combining the above with \eqref{eq:CDexPhi} and $\widehat\Dfrak(0)=\frac121_U$ gives for all $f\in L^{2}_{\ell,loc,\C}$ and $n=1,2,\ldots$ that
\begin{equation}\label{eq:DfrakEx}
	\Dfrak \big(f(\cdot)\phi_n\big) = t\mapsto (n+1)\,e^{-(n+1)t}\int_{-\infty}^te^{(n+1)s}f(s)\ud s\,\phi_n - \frac12f(t)\phi_n,
	\quad t\in\R.
\end{equation}
For all $n=0,1,\ldots$ and $u=\sum_{m=1}^\infty f_m\phi_m$, $f_m\in L^{2+}_\C$, it then holds that
$$
\begin{aligned}
	&\Ipdp{\Tfrak_\Sigma^*\pi_+e_{-(n+1)}\phi_n}{u}_{L^{2+}_U} =
	\Ipdp{\pi_+e_{-(n+1)}\phi_n}{\Tfrak_\Sigma \!\sum_{m=1}^\infty f_m(\cdot)  \phi_m}_{L^{2+}_U}\\
	&\quad=\int_0^\infty e^{-(n+1)t}\,\overline{\phi_n^*\big(\Dfrak f_n(\cdot)\phi_n\big)(t)}\ud t \\
	 &\quad=\int_0^\infty e^{-(n+1)t}\left(\overline{(n+1)e^{-(n+1)t}\int_0^te^{(n+1)s}f_n(s)\ud s -\frac12 f_n(t)}\right)\ud t \\
	&\quad=\int_0^\infty (n+1)e^{(n+1)s}\overline{f_n(s)}\int_s^\infty e^{-2(n+1)t}\ud t \ud s - \frac12\int_0^\infty e^{-(n+1)t} \overline{f_n(t)}\ud t=0.
\end{aligned}
$$
Hence, $\Tfrak_\Sigma^*\pi_+e_{-(n+1)}\phi_n=0$ for $n=0,1,\ldots$, which implies that 
$$
	D_{\Tfrak_\Sigma^*}^{2}(\pi_+e_{-(n+1)}\phi_n)=\pi_+e_{-(n+1)}\phi_n=
	D_{\Tfrak_\Sigma^*}^{-2}(\pi_+e_{-(n+1)}\phi_n).
$$
Using \eqref{eq:HaHrEasy} and \eqref{eq:WoDiagEx}, we then easily calculate
$$
	H_a\phi_n=\dW_o^*\dW_o\phi_n=4(n+1)\int_0^\infty e^{-2(n+1)t}\ud t\,\phi_n=2\,\phi_n,
$$
i.e., that $H_a=2\cdot 1_U$.

Now proceeding to $H_r$, we get from \eqref{eq:WcDiagEx} that
$$
	\bW_c^*\phi_n=\frac{\sqrt{n+1}}{2}\pi_-e_{n+1}\phi_n,
$$
and we need to evaluate $D_{\widetilde\Tfrak_\Sigma}^{-2}$ on this. By item (1) of Theorem \ref{thm:hankel} and \eqref{eq:DfrakEx},
$$
	(L_\Sigma\pi_-e_{n+1}\phi_n)(t)=
	(n+1)e^{-(n+1)t}\int_{-\infty}^te^{2(n+1)s}\ud s\,\phi_n 
	 -\frac{e^{(n+1)t}}2\phi_n=0,\quad t\leq0,
$$
so that $D^{-2}_{\widetilde\Tfrak_\Sigma}\pi_-e_{n+1}\phi_n=\pi_-e_{n+1}\phi_n$. Then \eqref{eq:HaHrEasy} and \eqref{eq:WcDiagEx} give
$$
	H_r\phi_n = (\bW_c\bW_c^*)^{-1}\phi_n =8\,\phi_n.
$$
Finally, by Theorem \ref{thm:stdlemmaL2reg}, all solutions $H$ to the spatial KYP inequality for $\Sigma$ are bounded and strictly positive definite; in fact they satisfy $2\cdot 1_X\preceq H\preceq 8\cdot 1_X$.

\end{example}

We now turn to the proof of the strict bounded real lemma, stated as Theorem \ref{thm:stdlemmastrict}.

\begin{proof}[Proof of   \rm{(2a)} $\Rightarrow$ \rm{(2b)}, \rm{(3a)} $\Rightarrow$ \rm{(3b)}, \rm{(4a)} $\Rightarrow$ \rm{(4b)}, \rm{(5a)} $\Rightarrow$ \rm{(5b)}]
Note that these are tautologies following from the definitions.
\end{proof}

\begin{proof}[Proof of \rm{(2a)} $\Leftrightarrow$ \rm{(3a)} $\Leftrightarrow$ \rm{(4a)} $\Rightarrow$ \rm{(5a)} ]
Let us assume (2a).  Thus there is a bounded, strictly positive definite $H$ on $X$ satisfying
\eqref{eq:strictKYP} for some $\delta > 0$.
As we saw in the proof of (2) $\Rightarrow$ (3) in Theorem \ref{thm:stdlemmaL2reg}, $\Gamma := H^{\frac{1}{2}}$ is an invertible change of state-space
coordinates $x^\circ := \Gamma x$ transforming the well-posed linear system $\Sigma = \sbm{ \Afrak & \Bfrak \\ \Cfrak & \Dfrak }$ to the system
$$
  \Sigma^\circ = \begin{bmatrix} \Afrak^\circ & \Bfrak^\circ \\ \Cfrak^\circ & \Dfrak^\circ \end{bmatrix} :=
  \begin{bmatrix} \Gamma \Afrak \Gamma^{-1}  & \Gamma \Bfrak \\ \Cfrak \Gamma^{-1}  & \Dfrak \end{bmatrix},
$$
and moreover, for each $t > 0$, the map
$$
 \Sigma^{\circ t} = \begin{bmatrix}  \Afrak^{\circ t} & \Bfrak^{\circ t} \\ \Cfrak ^{ \circ t}  & \Dfrak^{\circ t}   \end{bmatrix} \colon
 \begin{bmatrix} x^\circ(0) \\ u^\circ|_{[0,t]} \end{bmatrix} \mapsto \begin{bmatrix} x^\circ(t) \\ y^\circ|_{[0,t]} \end{bmatrix}
$$
has the same form when considered as a transformation of $\Sigma^t = \sbm{ \Afrak^t & \Bfrak^t \\ \Cfrak^t & \Dfrak^t }$:
$$
   \Sigma^{\circ t} = \begin{bmatrix} \Gamma \Afrak^t \Gamma^{-1}  & \Gamma \Bfrak^t \\ \Cfrak^t \Gamma^{-1} & \Dfrak^t \end{bmatrix}.
$$
Note that the inequality \eqref{eq:strictKYP} can be interpreted as the statement that the system trajectories $(\bu,\bx,\by)$ of $\Sigma$ satisfy
\begin{equation}\label{eq:strictpower}
\begin{aligned}
\| \Gamma \bx(t) \|^2 + \| \by|_{[0,t]} \|^2_{L^2([0,t];Y)} +\delta \| \bx|_{[0,t]} \|^2_{L^2([0,t];X)}
&   \\ \le \| \Gamma \bx(0) \|^2 + (1 - \delta) \| \bu|_{[0,t]} \|^2_{L^2([0,t]; U)},&\qquad t>0.
\end{aligned}
\end{equation}
Using that $(\bu,\bx,\by)$ is a system trajectory for $\Sigma$ if and only if $(\bu^\circ, \bx^\circ, \by^\circ) = (\bu,\Gamma \bx, \by)$ is a system trajectory for $\Sigma^\circ$ and the simple estimate $\|\Gamma x\|\leq\|\Gamma\|\cdot\|x\|$, we get from \eqref{eq:strictpower} that
\begin{equation}\label{eq:simbalance}
\begin{aligned}
\| \bx^\circ(t) \|^2 + \| \by|_{[0,t]} \|^2_{L^2([0,t];Y)} +\delta' \| \bx^\circ|_{[0,t]} \|^2_{L^2([0,t];X)}
&   \\ \le \| \bx^\circ(0) \|^2 + \left(1 - \delta\right) \| \bu|_{[0,t]} \|^2_{L^2([0,t]; U)}&\,,\qquad t>0,
\end{aligned}
\end{equation}
where $\delta':=\min(\delta,\delta/\|\Gamma\|^2)>0$. In \eqref{eq:simbalance}, we can still replace $\delta$ by $\delta'\leq \delta$, and the result then translates back to \eqref{eq:strictKYP} holding for the system $\Sigma^\circ$ with $H = 1_X$ and $\delta$ replaced by $\delta'>0$, and (3a) is established.

Conversely, assume (3a), so that $\Sigma$ is similar to a strictly passive system $\Sigma^\circ$ via an invertible
$\Gamma \colon X \mapsto X^\circ$, and let $(\bu,\bx,\by)$ be a system trajectory of $\Sigma$. Then $(\bu^\circ, \bx^\circ, \by^\circ) = (\bu, \Gamma \bx, \by)$
is a system trajectory of $\Sigma^\circ$ such that \eqref{eq:simbalance} holds for some $\delta=\delta'>0$. Setting $H = \Gamma^* \Gamma \succ 0$ and observing that $\|x\|/\|\Gamma^{-1}\|\leq \|\Gamma x\|$, we obtain from \eqref{eq:simbalance}, with $\delta'':=\min(\delta,\delta/\|\Gamma^{-1}\|^2)>0$, that
\begin{align*}
& \| H^{\frac{1}{2}}  \bx(t) \|^2    + \| \by|_{[0,t]} \|^2_{L^2([0,t]),Y)} +\delta'' \| \bx|_{[0,t]} \|^2_{L^2([0,t];X)}  \notag  \\
& \quad  \le \| H^{\frac{1}{2}} \bx(t) \|^2 + \left(1 - \delta''\right)   \| \bu|_{[0,t]} \|^2_{L^2([0,t],U)}.
\end{align*}
This in turn is equivalent to $H$ being a bounded, strictly positive-definite solution to \eqref{eq:strictKYP} with $\delta$ replaced by $\delta''>0$. Hence  (2a) $\Leftrightarrow$ (3a).

Next note that (2a) $\Leftrightarrow$ (4a) follows from the discussion in Remark \ref{R:strict-storages}. Finally (4a) $\Rightarrow$ (5a) is a tautology.
\end{proof}

\begin{proof}[Proof of \rm{(2b)} $\Leftrightarrow$ \rm{(3b)} $\Leftrightarrow$ \rm{(4b)} $\Rightarrow$ \rm{(5b)}.] (2b) $\Leftrightarrow$ (3b) is a simpler version of the above proof of
(2a) $\Leftrightarrow$ (3a), where one works with \eqref{eq:semi-strictKYP} in place of \eqref{eq:strictKYP} and the manipulations of $\delta$ associated to the now absent term $\| \bx|_{[0,t]} \|^2_{L^2([0,t];X)}$ are not needed. The equivalence of (2b) and (4b) is again a consequence
of the observations in Remark \ref{R:strict-storages}.  Finally, (4b) $\Rightarrow$ (5b) is a tautology.
\end{proof}

\begin{proof}[Proof of \rm{(5b)} $\Rightarrow$ \rm{(1)}]  Assume that $\Sigma$ satisfies condition (5b), so that $\Sigma$ has a semi-strict storage function
$S$ satisfying \eqref{eq:stordef-semi-strict}, repeated here (in the case $t_1 = 0$, $t_2 = t$) for the reader's convenience:  {\em There is a $\delta > 0$ such that
$$
S(\bx(t)) + \int_0^t \| \by(s) \|^2 \ud s \le S(\bx(0)) + (1 - \delta) \int_0^t \| \bu(s) \|^2 \ud s,\quad t\geq0,
$$
for all trajectories $(\bu, \bx, \by)$ of $\Sigma$ on $\rplus$.} As $S(x)$ (and hence $S(\bx(t))$) has values in $[0, \infty]$, we certainly then also have
$$
\int_0^t \| \by(s) \|^2 \ud s \le S(\bx(t)) + \int_0^t \| \by(s) \|^2 \ud s \le S(\bx(0)) + (1 - \delta) \int_0^t \| \bu(s) \|^2 \ud s
$$
for all such system trajectories $(\bu, \bx, \by)$ and $t\geq0$.  In particular, let us consider only those system trajectories initialized to satisfy $\bx(0) = 0$.
Then using that storage functions by definition satisfy $S(0) = 0$ and ignoring the middle in the preceding chain of inequalities, we see that
$$
\int_0^t \| \by(s) \|^2\ud s  \le (1 - \delta) \int_0^t \| \bu(s) \|^2 \ud s,\quad t>0.
$$
Letting $t$ tend to $+\infty$ then gives us
 $$
\| \by \|^2_{L^2({\mathbb R}^+, Y)}  \le (1 - \delta)  \| \bu \|^2_{L^2({\mathbb R}^+,U)}.
$$
Applying the Plancherel Theorem and taking Laplace transforms then gives us
$$
 \| \widehat \by \|^2_{H^2({\mathbb R}^+, Y)} \le (1 - \delta) \| \widehat \bu \|^2_{L^2({\mathbb R}^+, U)},
$$
where, as noted in \eqref{CTtransfunc}, $\widehat \by = M_{\widehat \Dfrak} \widehat \bu$; see also \eqref{MultOp}.   Hence $\| M_{\widehat \Dfrak} \| \le  \sqrt{1 - \delta}$
and therefore
$$
   \| \widehat \Dfrak\|_{H^\infty({\mathbb C}^+, \cB(U,Y))} = \| M_{\widehat \Dfrak} \|  \le \sqrt{1 - \delta} < 1,
$$
i.e., $\widehat \Dfrak$ is in the strict Schur class with $\cplus\subset\dom {\widehat\Dfrak}$, and we have arrived at statement (1) as wanted.
\end{proof}

Next we work towards a proof of the remainder of Theorem \ref{thm:stdlemmastrict}, namely that the implication (1) $\Rightarrow$ (2a) holds  under the additional hypothesis  that $\Afrak$
is exponentially stable and that at least one of the hypotheses (H1), (H2), (H3) holds.
The tool for this analysis is to dilate 
$\Sigma$ into a well-posed system $\Sigma_\varepsilon$ for which there exists a bounded
and boundedly invertible solution $H$ to the KYP-inequality for $\Sigma_\varepsilon$; 
this $H$ then turns out to be a bounded and boundedly invertible solution of the
strict KYP-inequality for the original well-posed system $\Sigma$. The details are as follows.

The first step is to embed the system node $\bS$ of $\Sigma$ into a larger system node $\bS_\varepsilon$  via a procedure which we call \emph{$\varepsilon$-regularization.} We extend the operators  $B \in \cB(U, X_{-1})$ and $C \in \cB(X_1, Y)$ to operators
$ B_\varepsilon  = \begin{bmatrix} B & \varepsilon 1_X \end{bmatrix} \in \cB(\sbm{ U \\ X}, X_{-1})$ and
$C_\varepsilon  = \sbm{ C \\ \varepsilon  1_X \\ 0} \in \cB(X_1, \sbm{Y \\ X \\ U})$. 
Using the operators $B_\varepsilon$ and $A$ we define $\bbm{\AB}_\varepsilon$ with domain
$$
\dom{\bbm{\AB}_\varepsilon}
\!:=\!\! \set{\bbm{x\\u\\u_1}\in\bbm{X\\U\\X}
\biggmid A_{-1}x+B_\varepsilon\bbm{u\\u_1}\in X}
\!=\!\emph{} \bbm{\dom{\AB}\\X},
$$
and action given by
$$
\bbm{\AB}_\varepsilon:=\bbm{A_{-1}&B_\varepsilon}\Big|_{\dom{\bbm{\AB}_\varepsilon}}=\bbm{\AB & \varepsilon 1_X}.
$$
Next we define $\bbm{\CD}_\varepsilon$ on $\dom{\bbm{\CD}_\varepsilon}=\dom{\bbm{\AB}_\varepsilon}$ by
\begin{equation}\label{CDeps}
\begin{aligned}
	\bbm{\CD}_\varepsilon\bbm{x\\u\\u_1}:=&\,
	C_\varepsilon\left(x-(\alpha-A_{-1})^{-1}B_\varepsilon\bbm{u\\u_1}\right)  \\
	&\quad+\bbm{\widehat\Dfrak(\alpha)&\varepsilon\,C(\alpha-A)^{-1} \\ \varepsilon\,
(\alpha-A_{-1})^{-1}B&\varepsilon^2(\alpha-A)^{-1}\\\varepsilon 1_U&0}\bbm{u\\u_1},   
\end{aligned}
\end{equation}
where $\alpha\in\rho(A)$ is the same number $\alpha$ as used in the definition  of $C \& D$ via formula \eqref{C&D} as part of the
definition of $\SmallSysNode$, and where $\widehat \Dfrak(\alpha)$ is the value at $\alpha$ of the transfer function $\widehat \Dfrak$ for the original well-posed
system $\Sigma$.  It is now an easy exercise to verify that $\bS_\varepsilon : = \sbm{ (A \& B)_\varepsilon  \\ (C \& D)_\varepsilon}$
is a system node in the sense of Definition \ref{def:sysnode}. 

Our next goal is to apply Theorem \ref{T:absnode} to show that $\bS_\varepsilon$ is the system node arising from a well-posed linear system $\Sigma_\varepsilon$.
Note that Theorem \ref{T:absnode} calls for a choice of $\omega \in {\mathbb R}$ with $\omega_\Afrak < \omega$.  Here we shall be assuming that $\Afrak$ is exponentially
stable, i.e., that $\omega_\Afrak < 0$.  Hence we have the option (which we shall use) of taking $\omega = 0$ in the application of Theorem \ref{T:absnode}.
For this case it is customary to simplify the terminology {\em $0$-bounded} (i.e., {\em $\rho$-bounded} for the case $\rho = 0$) to simply {\em bounded}.
Thus $\Bfrak$, $\Cfrak$, $\Dfrak$ being {\em bounded} means that the operators $\widetilde \Bfrak$, $\widetilde \Cfrak$, $\widetilde \Dfrak$ appearing in \eqref{eq:FrakTildeDef}
satisfy
$$
\widetilde \Bfrak \in \cB(L^{2-}_U, X), \quad \widetilde \Cfrak \in \cB(X, L^{2+}_Y), \quad \widetilde \Dfrak \in \cB(L^2_U, L^2_Y).
$$
The following lemma encodes the main properties of the $\varepsilon$-regularized system node $\bS_\varepsilon$.
In particular we see that we view the $\varepsilon$-regularization process as producing a dilation at three levels:
\begin{itemize}
\item at the system node level:  $\bS_\varepsilon$ can be seen as a dilation of $\bS$;

\item at the transfer-function level:  $\widehat \Dfrak_\varepsilon$ can be seen as a dilation of $\widehat \Dfrak$;

\item at the well-posed level:  $\sbm{ \Afrak_\epsilon^t & \Bfrak_\varepsilon^t \\ \Cfrak_\varepsilon^t & \Dfrak_\varepsilon^t }$ can be  seen
as a dilation of $\sbm{ \Afrak^t & \Bfrak^t \\ \Cfrak^t & \Dfrak^t }$.
\end{itemize}

\begin{lemma}\label{L:epsregsys}
Assume that $\Sigma=\sbm{\Afrak&\Bfrak\\ \Cfrak & \Dfrak}$ is an exponentially stable well-posed system with associated system node $\bS=\sbm{\AB\\ \CD}$ with a strict Schur class transfer function $\widehat\Dfrak\in\cS^0_{U,Y}$. Then, for all $\varepsilon>0$, the operator
\[
\bS_\varepsilon=\SmallSysNode_\varepsilon:=\sbm{\sbm{\AB}_\varepsilon\\\sbm{\CD}_\varepsilon}
\]
constructed above is the system node of a minimal,
exponentially stable, bounded, well-posed system $\Sigma_\varepsilon$ with transfer function $\widehat\Dfrak_\varepsilon$ given by
\begin{equation}\label{eq:transfepsilon}
\widehat\Dfrak_\varepsilon(\lambda)=\bbm{\widehat\Dfrak(\lambda)&
	\varepsilon\,C(\lambda-A)^{-1}\\
	\varepsilon\,(\lambda-A_{-1})^{-1}B&\varepsilon^2(\lambda-A)^{-1}\\\varepsilon 1_U&0},
	\qquad \lambda\in \rho(A).
\end{equation}
For $\varepsilon>0$ sufficiently small,
$\widehat\Dfrak_\varepsilon$ is also in the strict Schur class over $\cplus$.

For each $t\geq 0$ the $t$-dependent operators $\sbm{\Afrak_\varepsilon^t & \Bfrak_\varepsilon^t\\ \Cfrak_\varepsilon^t & \Dfrak_\varepsilon^t}$ for the
well-posed system $\Sigma_\varepsilon$ have the form
\begin{equation}   \label{Sigma-varepsilon-t}
\bbm{\Afrak_\varepsilon^t & \Bfrak_\varepsilon^t\\ \Cfrak_\varepsilon^t & \Dfrak_\varepsilon^t} =
\bbm{\Afrak^t & \Bfrak^t & \Bfrak_1^t\\ \Cfrak^t &\Dfrak^t& \Dfrak_1^t\\ \Cfrak_1^t &\Dfrak_2^t& \Dfrak_3^t\\ 0 &\varepsilon 1_{L^2([0,t],U)} & 0}:
\bbm{X\\ L^2([0,t],U) \\ L^2(([0,t],X)} \to \bbm{X\\ L^2([0,t],Y) \\ L^2([0,t],X) \\ L^2([0,t],U)},
\end{equation}
with $\Afrak^t$, $\Bfrak^t$, $\Cfrak^t$ and $\Dfrak^t$ equal to the $t$-dependent operators determined by the original system $\Sigma$ and $\Bfrak_1^t$, $\Cfrak_1^t$,
$\Dfrak_1^t$, $\Dfrak_2^t$ and $\Dfrak_3^t$ some operators acting between appropriate spaces.

If  $\Sigma$ is $L^2$-controllable ($L^2$-observable), then also $\Sigma_\varepsilon$ is $L^2$-controllable ($L^2$-observable).
\end{lemma}
\begin{proof}   We already left as an exercise for the reader to check that $\bS_\varepsilon$ is a system node. In order to prove that $\bS_\varepsilon$ is the system node of a well-posed system $\Sigma_\varepsilon$, we prove that conditions (1)--(3) of Theorem \ref{T:absnode} are satisfied.

First we verify that $B_\varepsilon$ is an admissible control operator for $A$.
For all $\sbm{\bu\\\bu_1}\in L^{2-}_{\ell,U\times X}$, the formula for $B_\varepsilon$ gives
\begin{equation}\label{eq:BfrakEps}
	\Bfrak_\varepsilon\bbm{\bu\\\bu_1}=
	\int_{-\infty}^0 \Afrak_{-1}^{-s}B\bu(s)\ud s+
	\varepsilon\int_{-\infty}^0 \Afrak^{-s}\bu_1(s)\ud s\in X.
\end{equation}
The first term lands in $X$ since $B$ is admissible for $A$.  The second term lands in $X$ by the compact support of $\bu_1$ and the uniform boundedness of $\Afrak$ on compact intervals.   Thus $B_\varepsilon$
 is an admissible control operator for $A$. We next observe that $C_\varepsilon$ is an admissible observation operator for $A$, i.e., that
$$
	x\mapsto \left(\bbm{C\\\varepsilon 1_X\\0}\Afrak^tx\right)_{t\geq0},\quad x\in\dom A,
$$
can be extended to a continuous linear operator from $X$ to $L^{2+}_{loc,Y\times X\times U}$; indeed, $C$ is admissible for $A$ and from $\omega_\Afrak<0$, we get
$$
	\int_0^T \|\varepsilon \Afrak^tx\|^2\ud t\leq
	-\frac{2M^2\varepsilon^2}{\omega_\Afrak}\|x\|^2.
$$
This completes the verification of conditions (1) and (2) in Theorem \ref{T:absnode}. 

In order to verify condition (3), we first prove formula \eqref{eq:transfepsilon} for the  transfer function $\widehat \Dfrak_\varepsilon$ of the system node $\bS_\varepsilon$. To this end, we use formulas \eqref{node-transfunc} and \eqref{CDeps} to compute: 
\begin{align*}
	\widehat\Dfrak_\varepsilon(\lambda)&=
\bbm{\CD}_\varepsilon\bbm{(\lambda-A_{-1})^{-1}B_\varepsilon\\\bbm{1_U&0\\0&1_X}}    \\   
 &=\begin{bmatrix}  C  \\ \varepsilon 1_X \\ 0 \end{bmatrix}   
 \big(  (\lambda - A_{-1})^{-1}  - (\alpha - A_{-1})^{-1}\big) \bbm{B&\varepsilon 1_X}  \\
&\qquad+\begin{bmatrix} \widehat \Dfrak(\alpha)   & \varepsilon C (\alpha  - A)^{-1} \\
    \varepsilon  (\alpha - A_{-1})^{-1} B & \varepsilon^2 (\alpha - A)^{-1}  \\
     \varepsilon 1_U  & 0 \end{bmatrix},\quad \lambda\in\rho(A),
\end{align*}
and observing that the $(1,1)$ entry equals $\CD\sbm{(\lambda-A_{-1})^{-1}B\\1_U}$, we get \eqref{eq:transfepsilon}. 
To verify condition (3) in Theorem \ref{T:absnode} applied to $\bS_\varepsilon$, we need to verify that each block entry appearing in the formula
\eqref{eq:transfepsilon} for $\widehat \Dfrak_\varepsilon$ is in $H^\infty(\cplus;\cB(K,L))$ for the relevant $K,L = X,U,Y$ as appropriate.  Since the original system $\Sigma$ is well-posed with $\omega_\Afrak < 0$, we can apply \cite[Lemma 10.3.3]{StafBook}
(with parameter $\omega$ taken to be $\omega = 0$)
to conclude that
\begin{equation*}
 \lambda \mapsto (\lambda - A)^{-1}
, \quad
\lambda \mapsto (\lambda - A_{-1})^{-1} B
,\quad \lambda \mapsto C (\lambda - A)^{-1}
,\quad\lambda\in\cplus,
\end{equation*}
are all in $H^\infty$ over $\cplus$ as wanted. With these observations in hand, it then becomes clear that choosing $\varepsilon > 0$ sufficiently small implies that $\widehat \Dfrak_\varepsilon$ is in the strict Schur class too. 
Moreover, it now follows from Theorem \ref{T:absnode} that $\bS_\varepsilon$ is the system node of a bounded, well-posed system $\Sigma_\varepsilon$, which is exponentially stable, since the $C_0$-semigroup is the same as that of the original system $\Sigma$. 
The formula \eqref{Sigma-varepsilon-t} for $\sbm{ \Afrak_\varepsilon^t & \Bfrak_\varepsilon^t  \\ \Cfrak_\varepsilon^t & \Dfrak_\varepsilon^t }$
is a straightforward consequence of the construction. 

We next discuss minimality. Fixing any $x\in X$ perpendicular to $\range{\Bfrak_\varepsilon}$, we get from \eqref{eq:BfrakEps} that for all $\bu_1\in L^{2-}_{\ell,X}$:
\begin{equation}\label{eq:epsmapadj}
	0=\Ipdp{x}{\varepsilon\int_{-\infty}^0 \Afrak^{-s}\bu_1(s)\ud s}_X
	=\varepsilon\Ipdp{s\mapsto \Afrak^{-s*}x}{\bu_1}_{L^{2-}_X}.
\end{equation}
By the density of $L^{2-}_{\ell,X}$ in $L^{2-}_{X}$, the continuous function $s\mapsto \Afrak^{-s*}x$ must vanish on $(-\infty,0)$, and letting $s\to0^-$,
we get that $x=0$, i.e., that $\Sigma_\varepsilon$ is (approximately) controllable. Since $C^*_\varepsilon=\bbm{C^*&\varepsilon 1_X&0}$,
\eqref{eq:epsmapadj} gives that $\Sigma_\varepsilon^*$ is controllable, i.e., $\Sigma_\varepsilon$
is observable; hence $\Sigma_\varepsilon$ is minimal.
As $\Afrak$ is exponentially stable by assumption, it follows that $\bW_{c,\varepsilon}$ is bounded by Lemma \ref{lem:PasvContr}, and hence it follows
from \eqref{eq:BfrakEps} that $\bW_{c,\varepsilon}=\bbm{\bW_c&\bW_\varepsilon}$ for some bounded operator $\bW_\varepsilon:L^{2-}_X\to X$;
now it is trivial from Definition \ref{def:L2min} that $\Sigma_\varepsilon$ inherits $L^2$-controllability from $\Sigma$. By \eqref{DualObsCon},
the bounded $L^2$-controllability map of $\Sigma^d_\varepsilon$ is $\bbm{\dW_o^*\ya&\dW_\varepsilon&0}$,
and so $\Sigma_\varepsilon^*$ is $L^2$-controllable, i.e., $\Sigma_\varepsilon$ is $L^2$-observable, whenever $\Sigma$ is
$L^2$-observable.
\end{proof}

Now we can prove the last part of Theorem \ref{thm:stdlemmastrict}.

\begin{proof}[Proof of {\rm (1) $\Rightarrow$ (2a)} in Theorem \ref{thm:stdlemmastrict}]
To complete the proof of Theorem \ref{thm:stdlemmastrict} it remains to show that (1) $\Rightarrow$ (2a) holds under the assumption that  $\Afrak$ is
exponentially stable and that at least one of the additional
conditions (H1), (H2) or (H3) holds. Assume $\widehat\Dfrak\in\cS^0_{U,Y}$. Let $\Sigma_\varepsilon$ be the $\varepsilon$-regularized system constructed above,
where we take $\varepsilon>0$ small enough, so that the transfer function $\widehat\Dfrak_\varepsilon$ of $\Sigma_\varepsilon$ is still a strict Schur class function.

We claim that each of the conditions (H1), (H2) and (H3) implies that the standard KYP-inequality for $\Sigma_\varepsilon$ has a bounded, strictly positive
definite solution $H$.
Assuming (H1), note that clearly the operators $B_\varepsilon$ and $C_\varepsilon$ satisfy the conditions of Proposition \ref{P:L2min-C0group}, so that item (3) of Proposition \ref{P:L2min-C0group} implies that $\Sigma_\varepsilon$ is $L^2$-minimal. Then the $L^2$-minimal standard bounded real lemma,
Theorem \ref{thm:stdlemmaL2reg}, shows that the standard KYP-inequality for $\Sigma_\varepsilon$ has a bounded, strictly positive definite solution
$H_\varepsilon$. In fact, both the operators $H_{\varepsilon,a}$ and $H_{\varepsilon,r}$ associated with the available storage and required supply of
$\Sigma_\varepsilon$ are bounded and strictly positive definite. 

For (H2) and (H3), note that $\Sigma_\varepsilon$ is minimal and
exponentially stable. Therefore, by Proposition \ref{P:BRLstrict+L2}, $H_{\varepsilon,a}$ and $H_{\varepsilon,r}^{-1}$
 are bounded and their inverses are bounded precisely when $\Sigma_\varepsilon$ is $L^2$-observable and $L^2$-controllable, respectively.
Since, by Lemma \ref{L:epsregsys}, $L^2$-observability of $\Sigma$  implies $L^2$-observability of $\Sigma_\varepsilon$, and likewise for $L^2$-controllability,
 it follows that $H_{\varepsilon,a}$ is a bounded, strictly positive definite solution to the KYP inequality for $\Sigma_\varepsilon$ whenever (H3) holds,
 while (H2) implies that $H_{\varepsilon,r}$ is a bounded, strictly positive definite solution to the KYP inequality for $\Sigma_\varepsilon$.

Hence, assuming (H1), (H2) or (H3) as well as the exponential stability,  we obtain a bounded, strictly positive definite solution $H$  to the standard KYP inequality
for $\Sigma_\varepsilon$. Our next goal is to show that this $H$ is also a solution to the strict KYP inequality \eqref{eq:strictKYP} for the original system $\Sigma$,
and thereby arrive at (2a) and complete the proof of (1) (and extra hypotheses) $\Rightarrow$ (2a).  We first need to probe a little deeper into the structure of
$\Sigma_\varepsilon$.

We shall have need for more explicit formulas for the operators
 $\Cfrak^t_1$ and $\Dfrak^t_2$ appearing in \eqref{Sigma-varepsilon-t}.
It is easy to see from the definition of $C_\varepsilon$ that $\Cfrak^t_1 \colon X \to L^2([0,t],X)$ is given by
$$
 \Cfrak^t_1 \colon x_0 \mapsto \bigg( s \mapsto  \varepsilon\,\Afrak^s x_0 \bigg)_{0 \le s \le t}.
$$
As for $\Dfrak^t_2$, what we know from \eqref{eq:transfepsilon} is that
$$
   \cL \Dfrak_2 \cL^{-1} = M_{\widehat \Dfrak_2}
$$
where $\cL$ is the bilateral Laplace transform, and where by \eqref{eq:transfepsilon} we know that
\begin{equation}\label{eq:Dhat2}
  \widehat \Dfrak_2(\lambda) = \varepsilon \,(\lambda - A_{-1})^{-1} B,\quad\lambda\in\rho(A).
\end{equation}
In general for a well-posed linear system $\Sigma = \sbm{ \Afrak & \Bfrak \\ \Cfrak & \Dfrak}$ it is difficult to compute the input-output map
$\Dfrak$ explicitly from the transfer function $\widehat \Dfrak(\lambda)$.  However  for the case here, where $\widehat \Dfrak_2$ is a simple expression
in terms of the resolvent of the semigroup generator $A$,  from experience with the reverse direction of computing  the frequency-domain transfer function from the
time-domain system equations, we conjecture that 
$$
 \Dfrak^t_2 \colon \bu|_{[0,t]} \mapsto \bigg( s \mapsto 
 \varepsilon\int_0^s \Afrak^{s - r}_{-1} B \bu(r) \ud r \bigg)_{0 \le s \le t};
$$
indeed this is correct, because it agrees with the observation that \eqref{eq:Dhat2} is the transfer function for the special case $\CD= \sbm{\varepsilon 1_X & 0 }$, followed by application of \eqref{Dfrak0} for this special $\CD$.

 We conclude that if $(\bu, \bx, \by)$ is a system trajectory on ${\mathbb R}^+$ with $\bx(0) = x_0$, then
 $$
   \begin{bmatrix} \Cfrak^t_1  & \Dfrak^t_2 \end{bmatrix} \colon \begin{bmatrix} x_0  \\ \bu|_{[0,t]} \end{bmatrix} \mapsto
   \bigg( s \mapsto \varepsilon\,\bx(s) \bigg)_{0 \le s \le t} =
   \varepsilon \begin{bmatrix} \Cfrak^t_{1_X,A}  & \Dfrak^t_{A,B} \end{bmatrix},
 $$
 where the right hand side is defined in \eqref{x(0)u-x}.

Let us now suppose that $H$ is bounded strictly positive-definite solution of the standard KYP-inequality associated with the $\varepsilon$-regularized
well-posed system $\Sigma_\varepsilon$.  Then $H$ satisfies
\[
\bbm{\Afrak_\varepsilon^t & \Bfrak_\varepsilon^t\\ \Cfrak_\varepsilon^t & \Dfrak_\varepsilon^t}^*
\bbm{H& 0\\ 0 & 1_{L^2([0,t], \sbm{Y \\ X \\ U})}}
\bbm{\Afrak_\varepsilon^t & \Bfrak_\varepsilon^t\\ \Cfrak_\varepsilon^t & \Dfrak_\varepsilon^t} \preceq
\bbm{H& 0\\ 0 & 1_{L^2([0,t], \sbm{U \\ X})}}.
\]
Compressing this inequality to $X \oplus L^2([0,t],U)$ and writing out $\sbm{\Afrak_\varepsilon^t & \Bfrak_\varepsilon^t\\ \Cfrak_\varepsilon^t &
\Dfrak_\varepsilon^t}$ yields
\begin{align*}
&\bbm{H& 0\\ 0 & 1_{L^2([0,t],U)}}  \\
 & \qquad \succeq
\bbm{\Afrak^t & \Bfrak^t \\ \Cfrak^t &\Dfrak^t\\ \varepsilon \Cfrak_{1_X,A}^t &\varepsilon \Dfrak_{A,B}^t\\ 0 &\varepsilon 1_{L^2([0,t],U)}}^*
\bbm{H&0\\0& 1_{L^2([0,t], \sbm{Y \\ X \\ U})}}
\bbm{\Afrak^t & \Bfrak^t \\ \Cfrak^t &\Dfrak^t\\ \varepsilon \Cfrak_{1_X,A}^t & \varepsilon \Dfrak_{A,B}^t\\ 0 &\varepsilon 1_{L^2([0,t],U)}}\\
&\qquad
= \bbm{\Afrak^t & \Bfrak^t \\ \Cfrak^t &\Dfrak^t}^* \bbm{H& 0\\ 0 & 1_{L^2([0,t],Y)}} \bbm{\Afrak^t & \Bfrak^t \\ \Cfrak^t &\Dfrak^t}
+  \varepsilon^2 \bbm{\Cfrak_{1_X,A}^{t *}\\ \Dfrak_{A,B}^{t*}} \bbm{\Cfrak_{1_X,A}^t &\Dfrak_{A,B}^t} + \\
& \qquad\qquad\qquad\qquad\qquad\qquad + \bbm{0&0\\0&\varepsilon^2 1_{L^2([0,t],U)}}.
\end{align*}
Subtracting $\sbm{0&0\\0&\varepsilon^2 1_{L^2([0,t],U)}}$ from both sides gives \eqref{eq:strictKYP} with $\delta=\varepsilon^2>0$ and this completes the proof.
\end{proof}

\paragraph{\bf Acknowledgments}
The authors thank the anonymous referees for their helpful remarks. This work is based on research supported in part by the National Research Foundation of South Africa (NRF) and the DSI-NRF Centre of Excellence in Mathematical and Statistical Sciences (CoE-MaSS). Any opinion, finding and conclusion or recommendation expressed in this material is that of the authors and the NRF and CoE-MaSS do not accept any liability in this regard.

\paragraph{\bf Data availability} 
Data sharing is not applicable to this article as no datasets were generated or analysed
during the current study.

\appendix

\section{An operator optimization problem}\label{sec:OpOpt}

In this section we consider a general operator optimization problem used in \S\ref{sec:QuadStorage}. Consider a contractive Hilbert space operator matrix:
\begin{equation}\label{Lop}
L=\bbm{T_1&0\\H & T_2}:\bbm{K_1\\ R_1}\to \bbm{K_2\\ R_2}.
\end{equation}
In particular, the operators $T_1$, $T_2$ and $H$ are contractive and hence bounded. Note that $H$ has a different meaning here in the appendix than in the main part of the paper. Further assume that $H$ admits a factorization
\begin{equation}\label{HfactA}
H|_{\dom{W_2}}=W_1W_2,
\end{equation}
where for some auxiliary Hilbert space $X$, the operators $W_1:\dom{W_1}\subset X \to R_2$ and $W_2:\dom{W_2}\subset K_1\to X$ are closed and densely defined. In particular, $W_1$ and $W_2$ then have closed, densely defined adjoints $W_1^*$ and $W_2^*$, respectively. Moreover,
$$
H^*|_{\dom{W_1^*}}=W_2^*W_1^*,
$$
since \eqref{HfactA} implies that $\range{W_2}\subset\dom{W_1}$, and then for all $x\in\dom{W_2}$ and $y\in\dom{W_1^*}$, it holds that $\Ipd{W_1W_2x}{y}=\Ipdp{W_2x}{W_1^*y}$. Then $W_1^*y\in\dom{W_2^*}$, and the boundedness of $H$ gives
$$
	\Ipdp{x}{W_2^*W_1^*y}=
	\Ipdp{W_1W_2x}{y}=\Ipdp{x}{H^*y}.
$$
Since $\dom{W_2}$ is dense, $W_2^*W_1^*y=H^*y$ for all $y\in\dom{W_1^*}$. In particular, also $\range{W_1^*}\subset \dom{W_2^*}$.

The objective of this appendix is to study the functions $S_-:X\to[0,\infty]$ and $S_+:X\to[0,\infty]$ determined by the general optimization problems
\begin{equation}\label{S-S+}
\begin{aligned}
S_-(x_0) & =
\left\{\begin{array}{ll}
\sup_{h\in R_1}\|W_1 x_0 + T_2 h\|^2 -\|h\|^2&\quad \mbox {if $x_0\in\dom{W_1}$}\\
\infty& \quad  \mbox {if $x_0\not\in\dom{W_1}$}
\end{array} \right.\\
S_+(x_0) & =
\left\{\begin{array}{ll}
\inf_{k\in W_2^{-1}(\{x_0\})}\|k\|^2 -\|T_1 k\|^2& \quad \mbox {if $x_0\in\range{W_2}$}\\
\infty &\quad  \mbox {if $x_0\not\in\range{W_2}$}.
\end{array} \right.
\end{aligned}
\end{equation}

In order to analyze these functions we define operators ${\bf X}_1$ and ${\bf X}_2$ on $X$ in the following lemma, which amounts to Lemma \ref{L:XaXb}, but formulated in a logically more optimal general context.

\begin{lemma}\label{L:X1X2}
Let $T_1$, $T_2$, $H$, $W_1$ and $W_2$ be as above. 
The following are true:
\begin{enumerate}
\item[(1)] Assume that $W_2$ has dense range. Then there exists a unique closable operator ${\bf X}_1$ from $X$ to $R_2$ with dense domain equal to $\range{W_2}$, $\range{{\bf X}_1}\perp \Ker{D_{T_2^*}}$ and
\begin{equation}\label{W1fact}
W_1|_{\range{W_2}}=D_{T_2^*}{\bf X}_1.
\end{equation}
Moreover, $\range{W_2}$ is a core for the closure $\overline{\bf X}_1$ of ${\bf X}_1$ and $\range{\overline{\bf X}_1}\perp \Ker{D_{T_2^*}}$. If additionally $W_1$ is injective, then $\overline{\bf X}_1$ is injective too. 

\item[(2)] If $W_1$ is injective, then there exists a unique closable operator ${\bf X}_2$ from $X$ to $K_1$ with dense domain $\range{W_1^*}$, $\range{{\bf X}_2}\perp \Ker{D_{T_1}}$, and
    \begin{equation}\label{W2*fact}
    W_2^*|_{\range{W_1^*}}=D_{T_1}{\bf X}_2.
    \end{equation}
Moreover, $\range{W_1^*}$ is a core for the closure $\overline{\bf X}_2$ of ${\bf X}_2$, whose range is still perpendicular to $\Ker{D_{T_1}}$. If $W_2$ has dense range, then $\overline{\bf X}_2$ is injective.
\end{enumerate}
\end{lemma}

The proof requires the use of the Moore-Penrose generalized inverse, which we reproduce from \cite[(4.31)]{BGtH18b}; see also \cite{NV74}. Let $W:X\to R$ be a closed, densely defined Hilbert-space operator. Define the operator $W^\dagger:R\supset\dom {W^\dagger}\to \dom W\subset X$ by $\dom{W^\dagger}:=\range W\oplus\range W^\perp$,
$$
	W^\dagger Wx=P_{\Ker W^\perp}x,\quad x\in\dom W,\quad
	W^\dagger\big|_{\range W^\perp}=0,
$$
where $P_{\Ker W^\perp}$ is the orthogonal projection in $X$ onto $\Ker W^\perp$.

\begin{proof}

Item (2) is obtained by applying item (1) to $\sbm{0&I\\I&0}L^*\sbm{0&I\\I&0}$ and hence we provide a detailed proof for item (1) only.

We start with the construction of ${\bf X}_1$. The fact that $L$ in \eqref{Lop} is contractive implies that $T_2 T_2^* + HH^*\preceq 1_{R_2}$, so that $D_{T_2^*}^2\succeq HH^*$. By Douglas' lemma, there exists a unique contraction ${\bf Y}_1$ from $K_1$ into $R_2$ with $D_{T_2^*}{\bf Y}_1=H$ and $\range{{\bf Y}_1}\perp \Ker{D_{T_2^*}}$. Next, write $W_2^\dagger$ for the Moore-Penrose generalized inverse of $W_2$. Then $W_2^\dagger$ has domain equal to $\range{W_2}$, since $W_2$ has dense range.

Now we define
\[
{\bf X}_1:={\bf Y}_1 W_2^\dagger.
\]
We claim that this operator ${\bf X}_1$ has the required properties. Clearly, ${\bf X}_1$ is a well-defined operator with dense domain $\dom{{\bf X}_1}=\range{W_2}$. Furthermore,
\[
D_{T_2^*} {\bf X}_1= D_{T_2^*}{\bf Y}_1 W_2^\dagger= H W_2^\dagger= W_1 W_2 W_2^\dagger=W_1|_{\range{W_2}}.
\]
We have $\range{{\bf X}_1}\subset \range{{\bf Y}_1}$ so that also $\range{{\bf X}_1}\perp \Ker{D_{T_2^*}}$. This establishes that ${\bf X}_1$ has the stated properties. If $\bX_1'$ also has these properties, then $\range{\bX_1-\bX_1'}\subset \Ker{D_{T_2^*}}\cap\Ker{D_{T_2^*}}^\perp$, so that $\bX_1'=\bX_1$, and uniqueness is also clear.

Next we prove that ${\bf X}_1$ is closable. Let $\{x_k\}_{k\geq 0}$ be a sequence in $\dom{{\bf X}_1}=\range{W_2}$ such that $x_k\to 0$. Assume that ${\bf X}_1 x_k \to y\in R_2$.
Then
\[
\lim_{k\to\infty} W_1 x_k = \lim_{k\to\infty} D_{T_2^*}{\bf X}_1 x_k = D_{T_2^*} y
\]
since $D_{T_2^*}$ is bounded and ${\bf X}_1 x_k \to y$. Since $W_1$ is closed and we have $x_k \to 0$ while
$W_1 x_k \to D_{T_2^*}y$,
we see that $0 = W_1 0 = D_{T_2^*} y$.
Since ${\bf X}_1 x_k  \perp \Ker{D_{T_2^*}}$, also $y \perp \Ker{D_{T_2^*}}$. But then $D_{T_2^*} y=0$ implies $y=0$, and hence ${\bf X}_1$ is closable.

Write $\overline{\bf X}_1$ for the closure of ${\bf X}_1$. Then ${\bf X}_1=\overline{\bf X}_1|_{\range{W_2}}$ and it follows by the definition of the closure of a closable operator that $\range{W_2}$ is a core of $\overline{\bf X}_1$. Moreover, $\range{\overline\bX_1}\subset\overline{\range{\bX_1}}\subset \Ker{D_{T_2^*}}^\perp$.

Let $x\in \dom{\overline{\bf X}_1}$ with $\overline{\bf X}_1 x=0$. Then there exists a sequence $\{x_k\}_{k\in\Z_+}$ in $\dom{{\bf X}_1}=\range{W_2}$ such that $x_k \to x$ in $X$ and ${\bf X}_1 x_k \to 0$ in $R_2$. Since $D_{T_2^*}$ is bounded, we have
\[
\lim_{k\to\infty}W_1 x_k= \lim_{k\to\infty}D_{T_2^*}{\bf X}_1 x_k= D_{T_2^*}0=0.
\]
Thus $x_k\to x$ and $W_1 x_k\to 0$. The fact that $W_1$ is a closed operator implies that $x\in\dom{W_1}$ and $W_1 x=0$.
If $W_1$ is injective, then  $x=0$, and it follows that $\overline{\bf X}_1$ is also injective in that case.
\end{proof}

Let ${\bf X}_1$ and ${\bf X}_2$ be as defined in Lemma \ref{L:X1X2} with closures $\overline{\bf X}_1$ and $\overline{\bf X}_2$. By Theorem VIII.32 in \cite{RS80}, $\overline{\bf X}_1$ and $\overline{\bf X}_2$ admit polar decompositions:
\[
\overline{\bf X}_1=U_1|\overline{\bf X}_1|
\quad\mbox{and}\quad \overline{\bf X}_2 =U_2|\overline{\bf X}_2|,
\]
where for $k=1,2$,  $|\overline{\bf X}_k|=(\overline{\bf X}_k^* \overline{\bf X}_k)^\frac{1}{2}$ is the positive self-adjoint square root of $\overline{\bf X}_k^* \overline{\bf X}_k$, which has $\dom{|\overline\bX_k|}=\dom {\overline\bX_k}$. If $\bX_k$ is injective, then $\overline{\bf X}_k$ is injective, and $U_k$ is then an isometry with $\range{U_k}$ equal to the closure of the range of $\overline{\bf X}_k$.

\begin{theorem}\label{T:S-S+quad}
Let $T_1$, $T_2$, $H$, $W_1$ and $W_2$ be as above with $W_1$ injective and $W_2$ having dense range. Define $S_-$ and $S_+$ as in \eqref{S-S+}. Then $\range{W_2}$ is contained in the domains of $|\overline{\bf X}_1|$ and $|\overline{\bf X}_2|^{-1}$ and we have
\[
S_-(x_0)=\||\overline{\bf X}_1|x_0\|^2 \quad\mbox{and}\quad S_+(x_0)=\||\overline{\bf X}_2|^{-1}x_0\|^2,\quad \mbox{for $x_0\in \range{W_2}$}.
\]
Moreover, $\range{W_2}$ is a core for $|\overline{\bf X}_1|$ and $\range{W_1^*}$ is a core for $|\overline{\bf X}_2|$.
\end{theorem}
\begin{proof}
We start with the formula for $S_-$. First note that
\[
\range{W_2}=\dom{{\bf X}_1}\subset \dom{\overline{\bf X}_1}= \dom{|\overline{\bf X}_1|}.
\]
Let $x_0\in \range{W_2}$ and $h\in R_1$. Then $W_1 x_0=D_{T_2^*}{\bf X}_1 x_0$ and
\begin{equation}\label{eq:S1proof}
\begin{aligned}
&\|W_1 x_0 + T_2 h\|^2-\|h\|^2
=\|D_{T_2^*}{\bf X}_1 x_0 + T_2 h\|^2-\|h\|^2\\
&\qquad\qquad=\|D_{T_2^*}{\bf X}_1 x_0\|^2 +2\re{\ipd{D_{T_2^*}{\bf X}_1 x_0}{T_2 h}} + \|T_2 h\|^2-\|h\|^2\\
&\qquad\qquad=\|D_{T_2^*}{\bf X}_1 x_0\|^2 +2\re{\ipd{D_{T_2^*}{\bf X}_1 x_0}{T_2 h}} - \|D_{T_2} h\|^2.
\end{aligned}
\end{equation}
Furthermore, $T_2^*D_{T_2^*}=D_{T_2}T_2^*$, see for instance \cite[p.\ 665]{GGKBookII}, and then
\begin{align*}
\ipd{D_{T_2^*}{\bf X}_1 x_0}{T_2 h}
=\ipd{T_2^*D_{T_2^*}{\bf X}_1 x_0}{ h}
=\ipd{D_{T_2} T_2^* {\bf X}_1 x_0}{ h}
=\ipd{ T_2^* {\bf X}_1 x_0}{D_{T_2} h},
\end{align*}
so that
\begin{align*}
2 \re{\ipd{D_{T_2^*}{\bf X}_1 x_0}{T_2 h}} & = 2 \re{\ipd{ T_2^* {\bf X}_1 x_0}{D_{T_2} h}}\\
&=\|T_2^* {\bf X}_1 x_0\|^2 + \|D_{T_2} h\|^2-\|T_2^* {\bf X}_1 x_0-D_{T_2} h\|^2.
\end{align*}
Inserting this back into \eqref{eq:S1proof}, we obtain
\begin{align*}
\|W_1 x_0 + T_2 h\|^2-\|h\|^2
&=\|D_{T_2^*}{\bf X}_1 x_0\|^2 + \|T_2^* {\bf X}_1 x_0\|^2 -\|T_2^* {\bf X}_1 x_0-D_{T_2} h\|^2\\
&=\Ipdp{(1-T_2T_2^*){\bf X}_1 x_0}{{\bf X}_1x_0} +\Ipdp{T_2T_2^* {\bf X}_1 x_0}{{\bf X}_1x_0}\\
&\qquad-\|T_2^* {\bf X}_1 x_0-D_{T_2} h\|^2\\
&=\|{\bf X}_1 x_0\|^2 -\|T_2^* {\bf X}_1 x_0-D_{T_2} h\|^2.
\end{align*}
Hence we find that
\begin{align*}
S_-(x_0)&=\|{\bf X}_1 x_0\|^2- \inf_{h\in R_1} \|T_2^* {\bf X}_1 x_0-D_{T_2} h\|^2\\
&=\||\overline{\bf X}_1| x_0\|^2- \inf_{h\in R_1} \|T_2^* {\bf X}_1 x_0-D_{T_2} h\|^2.
\end{align*}
It remains to show that the infimum over $R_1$ is 0.
By construction $\range{{\bf X}_1}\perp \Ker{D_{T_2^*}}$, and hence ${\bf X}_1 x_0$
is in $\overline{\range{D_{T_2^*}}}$. Note that $T_2^*$ maps $\overline{\range{D_{T_2^*}}}$ into $\overline{\range{D_{T_2}}}$, since for every $w\in \overline{\range{D_{T_2^*}}}$, there exists a sequence $v_k$, such that $D_{T_2^*}v_k\to w$ and then
$$
	T_2^*w=
	\lim_{k\to\infty}T_2^*D_{T_2^*}v_k=
	\lim_{k\to\infty}D_{T_2}T_2^*v_k\in
	\overline{\range{D_{T_2}}}.
$$
Thus $T_2^* {\bf X}_1 x_0$ is in $\overline{\range{D_{T_2}}}$, and this implies that we can approximate $T_2^* {\bf X}_1 x_0$
with vectors of the form $D_{T_2} h$, $h\in R_1$, so that the infimum is 0, as claimed.

Now we turn to $S_+$. We first argue that the factorization \eqref{W2*fact} transfers to
\begin{equation}\label{eq:W2fact}
W_2={\bf X}_2^* D_{T_1}|_{\dom{W_2}}.
\end{equation}
Indeed, for all $x\in\range{W_1^*}=\dom{{\bf X}_2}$ and $k\in\dom{W_2}$, since $\range{{W}_1^*} \subset
\dom{W_2^*}$ and $D_{T_1}$ is bounded, we see that
$$
\langle x, W_2k \rangle  = \langle W_2^*x, k \rangle = \langle D_{T_1} {\bf X}_2 x, k \rangle
= \langle {\bf X}_2 x, D_{T_1} k \rangle\,,
$$
from which we see that $D_{T_1} k \in \dom{{\bf X}_2^*}$ and ${\bf X}_2^* D_{T_1}k = W_2 k$ as claimed.

The polar decomposition $\overline{\bf X}_2=U_2|\overline{\bf X}_2|$ gives $\overline{\bf X}_2^*=|\overline{\bf X}_2| U_2^*$ by the boundedness of $U_2$. Hence $W_2=|\overline{\bf X}_2| U_2^* D_{T_1}|_{\dom{W_2}}$ and it follows that $\range{W_2}\subset \range{|\overline{\bf X}_2|}=\dom{|\overline{\bf X}_2|^{-1}}$.

Now, for $x_0\in \range{W_2}$ we have
\begin{align*}
S_+(x_0)
& =\inf_{k\in W_2^{-1}(\{x_0\})}\|k\|^2 -\|T_1 k\|^2
=\inf_{k\in\dom{W_2},\,
{\bf X}_2^* D_{T_1}k=x_0}\|D_{T_1}k\|^2\\
&=\inf_{v\in D_{T_1}\dom{W_2},\,{\bf X}_2^* v=x_0}\|v\|^2.
\end{align*}
Hence, we look for the infimum of $\|\ \|^2_{R_1}$ over the affine set
\[
\{ v\in D_{T_2}\dom{W_2} \bigmid {\bf X}_2^* v=x_0\}.
\]
Since $\dom{W_2}$ is dense, $D_{T_1}$ bounded, and $D_{T_1}\dom{W_2}\subset\dom{\bX_2^*}$ by \eqref{eq:W2fact}, the set in the infimum can be replaced by
\[
\left\{ v\in \overline{\range{D_{T_1}}} \bigcap \dom{{\bf X}_2^*} \Bigmid {\bf X}_2^* v=x_0\right\}.
\]
We thus have
\begin{equation}\label{eq:S+InfChar}
	S_+(x_0)=\inf_{v\in\overline{\range{D_{T_1}}}\cap({\bf X}_2^*)^{-1}(\{x_0\})}\|v\|^2
	\geq \inf_{v\in({\bf X}_2^*)^{-1}(\{x_0\})}\|v\|^2,
\end{equation}
because we in the right-hand side dropped one of the conditions on the set. Moreover, $({\bf X}_2^*)^{-1}(\{x_0\})=v_0+\Ker{{\bf X}_2^*}$ for some unique $v_0\in\Ker{{\bf X}_2^*}^\perp=\overline{\range{{\bf X}_2}}\subset \overline{\range{D_{T_1}}}$, and therefore the two infima in \eqref{eq:S+InfChar} are in fact both equal to $\|v_0\|^2$. We next verify that $v_0=U_2|\overline\bX_2|^{-1}x_0$; indeed this vector is in $\overline{\range{\bX_2}}\perp\Ker{\bX_2^*}$ and 
\[
\overline{\bf X}_2^* v_0 = 
|\overline{\bf X}_2| U_2^* U_2 |\overline{\bf X}_2|^{-1}x_0 = 
|\overline{\bf X}_2| |\overline{\bf X}_2|^{-1}x_0=x_0,
\]
where we used the isometricity of $U_2$. Then finally
$$
	S_+(x_0)=\|U_2|\overline\bX_2|^{-1}x_0\|^2=
	\||\overline\bX_2|^{-1}x_0\|^2.
$$

It remains only to prove that the claim regarding the core of $|\overline{\bf X}_k|$ follows from the corresponding property of $\overline\bX_k$ established in Lemma \ref{L:X1X2}. Pick $v\in \dom{|\overline\bX_k|}=\dom{\overline\bX_k}$ arbitrarily and let $D$ be a core for $\overline\bX_k$; then there exists a sequence $v_n\in D$ such that $v_n\to v$ and $\overline\bX_kv_n\to \overline\bX_kv$. Using that $U_k$ in the polar decomposition $\overline\bX_k=U_k|\overline\bX_k|$ is isometric, we get
$$
	|\overline\bX_k|v_n =
	U_k^*\overline\bX_kv_n \to
	U_k^*\overline\bX_kv =
	|\overline\bX_k|v,
$$
and hence, every core for $\overline\bX_k$ is also a core for $|\overline\bX_k|$.
\end{proof}

\end{document}